\documentclass[11pt, reqno]{amsart} % AMS article style
\usepackage{amssymb,amscd,amsfonts,amsbsy}
\usepackage{latexsym}
\usepackage{exscale}
\usepackage{amsmath,amsthm,amsfonts}
\usepackage{mathrsfs}
\usepackage{xcolor} % Enables the creations of colors.
\usepackage[colorlinks=true,linkcolor=blue,citecolor=red,urlcolor=red, %backref=page 
]{hyperref} % Add coloured links
\usepackage{esint} % For \fint symbol
\usepackage{stmaryrd}
\usepackage{pifont}
\usepackage{bbm}
\usepackage[shortlabels]{enumitem}

\usepackage[utf8]{inputenc}

\calclayout

\allowdisplaybreaks%[3] % Allow equation break between pages

\makeatletter
\@namedef{subjclassname@2020}{%
  \textup{2020} Mathematics Subject Classification}
\makeatother

\newtheorem{theorem}[equation]{Theorem}

\newtheorem{lemma}[equation]{Lemma}

\theoremstyle{definition}
\newtheorem{definition}[equation]{Definition}
\newtheorem{remark}[equation]{Remark}

\numberwithin{equation}{section} %numbers equations within sections.

\def\diam{\operatorname{diam}}
\def\div{\operatorname{div}}
\def\dist{\operatorname{dist}}
\def\supp{\operatorname{supp}}

\def\BMO{\operatorname{BMO}}

\def\loc{\operatorname{loc}}

\def\inter{\operatorname{int}}
\def\osc{\operatornamewithlimits{osc}}

\def\Tr{\operatorname{Tr}}

\def\N{\mathcal{N}}
\def\D{\mathbb{D}}

\def\F{\mathcal{F}}
\def\A{\mathcal{A}}

\def\G{\mathcal{G}}
\def\H{\mathcal{H}}
\def\Z{\mathbb{Z}}
\def\S{\mathcal{S}}
\def\R{\mathbb{R}}

\def\pom{\partial\Omega}
\def\Rn{\mathbb{R}^n}
\def\ree{\R^{n+1}}
\def\W{\mathcal{W}}

\def\w{\omega} 
\def\m{\mathfrak{m}}

\renewcommand{\emptyset}{\mbox{\textup{\O}}}
\newcommand{\tinyemptyset}{\mbox{\tiny \textup{\O}}}

\newcommand\restr[2]{\ensuremath{\left.#1\right|_{#2}}}

\def\Xint#1{\mathchoice
{\XXint\displaystyle\textstyle{#1}}%
{\XXint\textstyle\scriptstyle{#1}}%
{\XXint\scriptstyle\scriptscriptstyle{#1}}%
{\XXint\scriptscriptstyle\scriptscriptstyle{#1}}%
\!\int}
\def\XXint#1#2#3{{\setbox0=\hbox{$#1{#2#3}{\int}$ }
\vcenter{\hbox{$#2#3$ }}\kern-.585\wd0}}
\def\barint{\Xint-}

\newcommand{\bariint}{\barint\mkern-11.5mu\barint}

\newcommand{\normmm}[1]{{\left\vert\kern-0.25ex\left\vert\kern-0.25ex\left\vert #1
		\right\vert\kern-0.25ex\right\vert\kern-0.25ex\right\vert}}

\begin{document}
\allowdisplaybreaks

\title[Absolute continuity of degenerate elliptic measure]
{Absolute continuity of degenerate elliptic measure}

\author{Mingming Cao}
\address{Mingming Cao\\
Instituto de Ciencias Matem\'aticas CSIC-UAM-UC3M-UCM\\
Con\-se\-jo Superior de Investigaciones Cient{\'\i}ficas\\
C/ Nicol\'as Cabrera, 13-15\\ 
E-28049 Ma\-drid, Spain} 
\email{mingming.cao@icmat.es}

\author{K\^{o}z\^{o} Yabuta}
\address{K\^{o}z\^{o} Yabuta\\ 
Research Center for Mathematics and Data Science\\
Kwansei Gakuin University\\
Gakuen 2-1, Sanda 669-1337\\
Japan}\email{kyabuta3@kwansei.ac.jp}

\thanks{The first author acknowledges financial support from Spanish Ministry of Science and Innovation through the Ram\'{o}n y Cajal 2021 (RYC2021-032600-I), through the ``Severo Ochoa Programme for Centres of Excellence in R\&D'' (CEX2019-000904-S), and through PID2019-107914GB-I00, and from the Spanish National Research Council through the ``Ayuda extraordinaria a Centros de Excelencia Severo Ochoa'' (20205CEX001).}

\year=2024 \month=09 \day=16

\date{\today}


%31B05 Harmonic, subharmonic, superharmonic functions
%35J08 Green's functions
%35J25 Boundary value problems for second-order elliptic equations 
%42B37 Harmonic analysis and PDE


%%%%%%%%%%%%%%%%%%%%%% ABSTRACT ABSTRACT ABSTRACT %%%%%%%%%%%%%%%%%%%%%
\begin{abstract}
Let $\Omega \subset \ree$ be an open set whose boundary may be composed of pieces of different dimensions. Assume that $\Omega$ satisfies the quantitative openness and connectedness, and there exist doubling measures $m$ on $\Omega$ and $\mu$ on $\pom$ with appropriate size conditions. Let $Lu=-\div(A\nabla u)$ be a real (not necessarily symmetric) degenerate elliptic operator in $\Omega$. Write $\w_L$ for the associated degenerate elliptic measure. We establish the equivalence between the following properties: (i) $\w_L \in A_{\infty}(\mu)$, (ii) the Dirichlet problem for $L$ is solvable in $L^p(\mu)$ for some $p \in (1, \infty)$, (iii) every bounded null solution of $L$ satisfies Carleson measure estimates with respect to $\mu$, (iv) the conical square function is controlled by the non-tangential maximal function in $L^q(\mu)$ for all $q \in (0, \infty)$ for any null solution of $L$, and (v) the Dirichlet problem for $L$ is solvable in $\BMO(\mu)$. On the other hand, we obtain a qualitative analogy of the previous equivalence. Indeed, we characterize the absolute continuity of $\w_L$ with respect to $\mu$ in terms of local $L^2(\mu)$ estimates of the truncated conical square function for any bounded null solution of $L$. This is also equivalent to the finiteness $\mu$-almost everywhere of the truncated conical square function for any bounded null solution of $L$.  
\end{abstract}
%%%%%%%%%%%%%%%%%%%%%% ABSTRACT ABSTRACT ABSTRACT %%%%%%%%%%%%%%%%%%%%%

\maketitle

%\tableofcontents

%%%%%%%%%%%%%%%%%%%%%%% SECTION SECTION SECTION %%%%%%%%%%%%%%%%%%%%%%%
%%%%%%%%%%%%%%%%%%%%%%% SECTION SECTION SECTION %%%%%%%%%%%%%%%%%%%%%%%
\section{Introduction} 
\subsection{Background}  
Given a domain $\Omega \subset \ree$, consider the Dirichlet boundary value problem for an elliptic operator $L$ in $\Omega$:  
\begin{equation}\label{eq:D}
\begin{cases}
L u =0 & \quad\text{in} \quad \Omega, \\
u =f & \quad\text{on}\quad \pom. 
\end{cases}
\end{equation}
In the last fifty years, it has been of great interest to study this problem on domains with rough boundaries. For this purpose, one typically investigates the properties of the corresponding elliptic measure, which provides a representation formula for solutions to the Dirichlet problem \eqref{eq:D} with continuous data on the boundary. The techniques from harmonic analysis and geometric measure theory have allowed us to study the regularity of elliptic measures and hence understand this subject well (cf. \cite{K}). Roughly speaking, the solvability of the Dirichlet problem is equivalent to quantitatively absolute continuity of the elliptic measure with respect to the boundary surface measure. On the other hand, the good properties of elliptic measures enable us to effectively use the machinery from these fields to obtain information about the topological properties and the regularity of the domains. These promote great advances in the understanding of the connections between harmonic analysis, partial differential equations, and geometric measure theory.  

The study of absolute continuity of elliptic measure originated from the classical result of F. and M. Riesz \cite{RR}, which proved that for a simply connected planar domain with a rectifiable boundary, harmonic measure $\w_{-\Delta}$ is mutually absolutely continuous with respect to the surface measure $\sigma$. In dimensions greater than 2, a quantitative analogue ($\w_{-\Delta} \in A_{\infty}(\sigma)$, cf.  Definition \ref{def:Ainfty}) was established by Dahlberg \cite{D77} in Lipschitz domains, by Jerison and Kenig \cite{JK} in $\BMO_1$ domains, by David and Jerison \cite{DJ} in chord-arc domains, and by Azzam et al. \cite{AHMNT, HM, HMU} in $1$-sided chord-arc domains. This was culminated in the recent results of \cite{AHMMT} under some optimal background hypothesis. Beyond that, for the general uniformly elliptic operators satisfying the Dahlberg--Kenig--Pipher condition, Hofmann, Martell, Mayboroda, Toro, and Zhao \cite{HMMTZ} showed the absolute continuity of the elliptic measure with respect to the surface measure ($\w_L \in A_{\infty}(\sigma)$), which extended the previous results in \cite{HMT, KP}. 

The absolute continuity of elliptic measure is also closely linked with the solvability of the Dirichlet problem \eqref{eq:D}.  The importance of the quantitative absolute continuity of the elliptic measure with respect to the surface measure comes from the fact that $\w_L \in RH_q(\sigma)$ (the Reverse H\"{o}lder class with respect to $\sigma$) is equivalent to the $L^{q'}(\sigma)$-solvability of the Dirichlet problem (cf. \cite[Theorem~1.7.3]{K}) for the uniformly elliptic operators. Furthermore, the $\BMO$-solvability of \eqref{eq:D} can be characterized by $\w_L \in A_{\infty}(\sigma)$. This was shown in \cite{FS} for the Laplacian in $\R^{n+1}_+$, and extended to uniformly elliptic operators in Lipschitz domains \cite{DKP} and $1$-sided chord-arc domains \cite{Zhao}. In the non-connected case, Hofmann and Le \cite{HL} provided some criterion. A main difficulty in the proof of $\BMO$-solvability is the $\S<\N$ estimate (the conical square function is controlled by the non-tangential maximal function in $L^q(\mu)$). The latter goes back to the work of Dahlberg, Jerison, and Kenig \cite{DJK} in Lipschitz domains. 

On the other hand, the Carleson measure estimate (cf. Definition \ref{def:several}) can be viewed as a local version of the $\S<\N$ estimate. Kenig et al. \cite{KKiPT} first established the connection between the Carleson measure estimate and solvability of uniformly elliptic operators in bounded Lipschitz domains, which was extended to the degenerate setting by Hofmann, Le, and Morris \cite{HLM} in the upper-half space. A more general conclusion was obtained in \cite{CHMT} in $1$-sided chord-arc domains, which asserted that $\w_L \in A_{\infty}(\sigma)$ if and only if every bounded weak solution of \eqref{eq:D} satisfies the Carleson measure estimate. In another direction, the Carleson measure estimate is also correlated with the approximability property of harmonic functions. The equivalence between them was recently given in \cite{GMT, HMM} for the Laplacian and in \cite{AGMT} for elliptic operators whose matrix satisfies the very strong assumption that $|\nabla A(X)| \, dX$ is the density of a Carleson measure. It is worth mentioning that all of them worked in open sets with poor connectivity properties satisfying the corkscrew condition with ADR boundary. All the results aforementioned are established in domains with boundaries of co-dimension 1. 

However, when the boundary has co-dimension larger than 1, the study of absolute continuity of elliptic measure is severed, which essentially owe to deep dimensional constraints on the support of the elliptic measure. To this end, David, Feneuil and Mayboroda \cite{DFM1, DFM2, DFM3} started a continuous program to characterize the geometry of boundaries of high co-dimension via the theory of certain degenerate elliptic operators. In \cite{DFM3}, they formulated a general framework to study the  degenerate elliptic operators, including classical elliptic operators, Caffarelli--Silvestre fractional operators \cite{CS}, $t$-independent degenerate elliptic operators \cite{HLM} and so on. Also, it allows the boundary of domains to be composed of pieces of different dimensions, for examples, sawtooth domains (cf. \cite{HM, MP}), and balls in $\R^3$ deprived of one diameter. Recently, the authors in \cite{FP} formulated a general framework to establish Carleson perturbation for elliptic operators.

\subsection{Main results}
Based on the theory of degenerate elliptic operators in \cite{DFM3}, we will investigate the absolute continuity of the degenerate elliptic measure. To set the stage, we introduce the general machinery. Let $\Omega \subset \R^{n+1}$ be an open set equipped with a Borel measure $m$, and let $\mu$ be a Borel measure on $\pom$. Assume that $(\Omega, m, \mu)$ satisfies the following hypotheses: 
\begin{list}{(\theenumi)}{\usecounter{enumi}\leftmargin=1cm \labelwidth=1cm \itemsep=0.2cm \topsep=.2cm \renewcommand{\theenumi}{H\arabic{enumi}}}
 
\item \label{H1} There exists $c_1>0$ such that for any $x \in \pom$ and any $0<r<\diam(\Omega)$, there exists a point $X \in B(x, r)$ such that $B(X, c_1 r) \subset \Omega$. 

\item \label{H2} There exists a positive integer $c_2 =N+1$ such that if $X, Y \in \Omega$ satisfy $\min\{\delta(X), \delta(Y)\}>r$ and $|X-Y| \le 7c_1^{-1}r$, then there exist $N+1$ points $Z_0 :=X, Z_1,\ldots, Z_N:=Y$ such that for any $i \in \{0, \ldots, N-1\}$, $|Z_i - Z_{i+1}| < \delta(Z_i)/2$. 

\item \label{H3} The support of $\mu$ is $\pom$ and $\mu$ is positive, locally finite and doubling, i.e., there exists $c_3>1$ such that 
\[
0 < \mu(B(x, 2r)) \le c_3 \, \mu(B(x, r)) < \infty, \quad \forall x \in \pom, \, r>0. 
\]

\item \label{H4} The measure $m$ is mutually absolutely continuous with respect to the Lebesgue measure; that is, there exists a weight $w \in L^1_{\loc}(\Omega)$ such that 
\[
m(A) = \iint_A w(X) \, dX, \quad\text{ for any Borel set } A \subset \Omega
\]
and such that $w(X)>0$ for (Lebesgue) almost every $X \in \Omega$. In addition, $m$ is doubling, i.e. there exists $c_4 \ge 1$  such that
\[
m(B(X, 2r) \cap \Omega) \le c_4 \, m(B(X, r) \cap \Omega), \quad \forall X \in \overline{\Omega}, \, r>0. 
\]

\item \label{H5} The quantity $\rho$ given by 
\[
\rho(x, r) := \frac{m(B(x, r) \cap \Omega)}{r \mu(B(x, r))},\quad x \in \pom, \, r>0, 
\]
satisfies 
\[
\frac{\rho(x, r)}{\rho(x, s)} \le c_5 \bigg(\frac{r}{s}\bigg)^{1-\varepsilon}, 
\]
for some constant $c_5>0$ and $\varepsilon:=c_5^{-1}$. 

\item \label{H6} If $D$ is compactly contained in $\Omega$ and $u_j \in \mathscr{C}^{\infty}(\overline{D})$ is a sequence of functions such that $\iint_D |u_j| \, dm \to 0$ and $\iint_{D} |\nabla u_j - v|^2 \, dm \to 0$ as $j \to \infty$, where $v$ is a vector-valued function in $L^2(D, dm)$, then $v \equiv 0$. In addition, there exists $c_6$ such that for any ball $B$ satisfying $2B \subset \Omega$ and any function $u \in \mathscr{C}^{\infty}(\overline{B})$, one has 
\begin{align*}
\bariint_B |u-u_B| \, dm \lesssim \, r(B) \, \bigg(\bariint_B |\nabla u|^2 \, dm\bigg)^{\frac12}, 
\quad\text{where $u_B := \bariint_B u \, dm$}.
\end{align*} 
\end{list}

Some comments on the conditions \eqref{H1}--\eqref{H6} are in order. The assumption \eqref{H1} is known as the corkscrew condition, and can be seen as quantitative openness. The assumption \eqref{H2} is a condition of quantitative connectedness, and is a slightly weaker way to state the usual Harnack chain condition. In particular, together with \eqref{H1}, the condition \eqref{H2} encodes a self-improvement, see \cite[Proposition 2.18]{DFM3}. Moreover, it is easy to prove that the assumption \eqref{H3} is equivalent to the apparently stronger following condition: there exist $d_{\mu} > 0$ and $C_1 > 0$ (both depending only on $c_3$) such that 
\begin{align*}
\mu(B(x, \lambda r)) 
\le C_1 \, \lambda^{d_{\mu}} \mu(B(x, r)), \quad \forall \lambda \ge 1, \, x \in \pom, \, r>0. 
\end{align*}
Analogously, the assumption \eqref{H4} is equivalent to that 
\begin{align*}
m(B(X, \lambda r) \cap \Omega) 
\le C_2 \, \lambda^{d_m} \mu(B(X, r) \cap \Omega), \quad \forall \lambda \ge 1, \, X \in \overline{\Omega}, \, r>0. 
\end{align*}
The condition \eqref{H5} states the compared growths of $m$ and $\mu$. Specifically, $m$ and $\mu$ need to be intertwined, and it requires that $m(B(x, r) \cap \Omega)$ does not grow much faster than $\mu(B(x, r))$, with a precise limitation on the exponent. This is needed to establish a trace theorem, see \cite[Section 6]{DFM3}. Finally, the condition \eqref{H6} means that the measure $m$ is regular enough and satisfies a weak Poincar\'{e} inequality. It is also used to get the Poincar\'{e} inequality at the boundary, which is crucial to the boundary De Giorgi-Nash-Moser estimates (cf. \cite[Section 11]{DFM3}). The condition \eqref{H6} will be sometimes replaced by the much stronger condition:  
\begin{list}{(\theenumi)}{\usecounter{enumi}\leftmargin=1cm \labelwidth=1cm \itemsep=0.2cm \topsep=.2cm \renewcommand{\theenumi}{H6$'$}}

\item\label{H6'} 
$\qquad\qquad\qquad \sup\limits_{B} w \le c'_6 \inf\limits_B w$, \quad for any ball $B \subset \ree: 2B \subset \Omega$. 
\end{list}

In the sequel, we study the elliptic operator $Lu=-\div(A\nabla u)$ in $\Omega$. Suppose that $A$ is a real (not necessarily symmetric) $(n+1) \times (n+1)$ matrix on $\Omega$ such that the degenerate ellipticity condition holds:  
\begin{align}\label{eq:elli} 
\Lambda^{-1} w(X) |\xi|^2 \leq A(X) \xi \cdot \xi \quad\text{ and } \quad 
A(X) \xi \cdot \eta \leq \Lambda \, w(X) |\xi| |\eta|, 
\end{align}
for all $\xi, \eta \in \R^{n+1}$ and almost every $X \in \Omega$, where $\Lambda \ge 1$ and $w$ is given in \eqref{H4}. If $w \equiv 1$, \eqref{eq:elli} is called the uniformly elliptic condition. Let $\w_L$ be the degenerate elliptic measure associated with $L$ (cf. Section \ref{sec:PDE}). Our main goal is to investigate absolute continuity of $\w_L$ with respect to $\mu$.

The purpose of the current article is two-fold. First, considering those estimates above, we try to establish the equivalence between $\w_L \in A_{\infty}(\mu)$ and $L^p(\mu)$-solvability of the Dirichlet problem \eqref{eq:D} for the degenerate elliptic operators. This is also expected to be equivalent to $\BMO(\mu)$-solvability of \eqref{eq:D}, $\mathcal{S}<\mathcal{N}$ estimates for $L$ in $L^q(\mu)$, and Carleson measure estimate for $L$ with respect to $\mu$. Second, by means of the conical square function, we seek characterizations of absolute continuity of $\mu$ with respect to $\w_L$, which is a qualitative version of the preceding results.

Our first objective is to establish the estimates aforementioned in the degenerate case. More precisely, we prove the following equivalence.

%%%%%%%%%%%%%%%%%%%%%% THEOREM THEOREM THEOREM %%%%%%%%%%%%%%%%%%%%%%
\begin{theorem}\label{thm:main}
Let $(\Omega, m, \mu)$ satisfy \eqref{H1}--\eqref{H6} and let $Lu=-\div(A\nabla u)$ be a real (not necessarily symmetric) degenerate elliptic operator with the coefficients matrix $A$ satisfying \eqref{eq:elli}. Then the following statements are equivalent:

\begin{list}{\rm{(\theenumi)}}{\usecounter{enumi}\leftmargin=1.2cm \labelwidth=1.2cm \itemsep=0.2cm \topsep=.2cm \renewcommand{\theenumi}{\alph{enumi}}}
 	
\item\label{list-Ai} $\w_L \in A_{\infty}(\mu)$ $($cf.~Definition $\ref{def:Ainfty})$.
	
\item\label{list-Lp} The Dirichlet problem \eqref{eq:D} for $L$ is solvable in $L^p(\mu)$ for some $p\in (1,\infty)$ (cf. Definition $\ref{def:several}$).
	
\item\label{list-CME} Every bounded weak solution of $Lu = 0$ satisfies Carleson measure estimates (cf. Definition $\ref{def:several}$).  
		
\item\label{list-SN} Every weak solution of $Lu = 0$ satisfies $\S<\N$ estimates in $L^q(\mu)$ for all $q \in (0, \infty)$ (cf. Definition $\ref{def:several}$). 

\item\label{list-BMO}  The Dirichlet problem \eqref{eq:D} for $L$ is solvable in $\BMO(\mu)$ (cf. Definition $\ref{def:several}$). 
\end{list} 
\end{theorem}
%%%%%%%%%%%%%%%%%%%%%% THEOREM THEOREM THEOREM %%%%%%%%%%%%%%%%%%%%%%

Our next goal is to formulate a qualitative version of Theorem~\ref{thm:main} in the spirit of \cite{CMO}. The $A_\infty$ condition will be replaced by absolute continuity. The qualitative analog of $\mathcal{S}<\mathcal{N}$ estimates turns into local $L^2$ estimates for conical square functions on the boundary, which are the qualitative version of Carleson measure estimates. Eventually, we will see that all these are equivalent to the fact that the truncated conical square function $\S^{\alpha}_r$ (cf. Definition \ref{def:SN}) is finite almost everywhere with respect to the weight $\mu$.

%%%%%%%%%%%%%%%%%%%%%% THEOREM THEOREM THEOREM %%%%%%%%%%%%%%%%%%%%%%
\begin{theorem}\label{thm:abs}
Let $(\Omega, m, \mu)$ satisfy \eqref{H1}--\eqref{H6}. Then there exists $\alpha_0>0$ (depending only on $n$) such that for each fixed $\alpha \geq \alpha_0$ and for every real (not necessarily symmetric) degenerate elliptic operator $Lu=-\div(A\nabla u)$ with the coefficients matrix $A$ satisfying \eqref{eq:elli}, the following statements are equivalent:
\begin{list}{\rm{(\theenumi)}}{\usecounter{enumi}\leftmargin=1cm \labelwidth=1cm \itemsep=0.2cm
			\topsep=.2cm \renewcommand{\theenumi}{\roman{enumi}}}
		
\item\label{list:abs-1} $\mu \ll \omega_L$ on $\pom$.
		
\item\label{list:abs-2} $\pom=\bigcup_{N \geq 0} F_N$, where $\mu(F_0)=0$, for each $N \geq 1$, $F_N=\pom \cap \pom_N$ for some bounded $1$-sided NTA domain (cf. Definition $\ref{def:NTA}$) $\Omega_N \subset \Omega$,  and $\S^{\alpha}_r u \in L^2(F_N, \mu)$ for every weak solution $u \in W_r(\Omega) \cap L^{\infty}(\Omega)$ of $Lu=0$ in $\Omega$, and for all (or for some) $r>0$.
				
\item\label{list:abs-3} $\mathcal{S}^{\alpha}_r u(x)<\infty$ for $\mu$-a.e.~$x \in \pom$, for every weak solution $u \in W_r(\Omega) \cap L^{\infty}(\Omega)$ of  $Lu=0$ in $\Omega$ and for all (or for some) $r>0$.
				
\item\label{list:abs-4} For every weak solution $u \in W_r(\Omega)\cap L^{\infty}(\Omega)$ of $Lu=0$ in $\Omega$ and for $\mu$-a.e.~$x \in \pom$ there exists $r_x>0$ such that $\S^{\alpha}_{r_x} u(x)<\infty$. 
					
\end{list}
\end{theorem}
%%%%%%%%%%%%%%%%%%%%%% THEOREM THEOREM THEOREM %%%%%%%%%%%%%%%%%%%%%%

\subsection{Applications} 
Let us present some applications of Theorem \ref{thm:main}. 

\vspace{0.2cm}\noindent{$\bullet$ \textbf{Application 1: Uniformly elliptic operators.}} 
We start with the uniformly elliptic operators $Lu=-\div(A\nabla u)$, where $A(X)$ satisfies the standard uniform ellipticity conditions, that is, \eqref{eq:elli} holds for $w \equiv 1$. We require that $\Omega \subset \ree$ satisfies the 1-sided NTA conditions \eqref{H1} and \eqref{H2}, and $\pom$ is $n$-ADR. Write $\mu:=\restr{\H^n}{\pom}$. Then by definition, we see that $\mu$ is a doubling measure on $\pom$ that satisfies \eqref{H5} with $\rho \simeq 1$. Also, \eqref{H6'} is trivially satisfied.

%%%%%%%%%%%%%%%%%%%%%%% THEOREM THEOREM THEOREM %%%%%%%%%%%%%%%%%%%%%
\begin{theorem}\label{thm:CAD}
Let $\Omega \subset \ree$, $n \ge 2$, be a 1-sided CAD (cf. Definition $\ref{def:NTA}$), and let $Lu=-\div(A\nabla u)$ be a real  (not necessarily symmetric) uniformly elliptic operator in $\Omega$. Then all the properties \eqref{list-Ai}--\eqref{list-BMO} in Theorem $\ref{thm:main}$ are equivalent for $\mu=\restr{\H^n}{\pom}$. 
\end{theorem}
%%%%%%%%%%%%%%%%%%%%%%% THEOREM THEOREM THEOREM %%%%%%%%%%%%%%%%%%%%%

Indeed, under the background hypotheses of Theorem \ref{thm:CAD}, the equivalences $\eqref{list-Ai} \iff \eqref{list-CME}$ and $\eqref{list-Ai} \iff \eqref{list-BMO}$ were respectively shown in \cite{CHMT} and \cite{Zhao}, while the implications $\eqref{list-Ai} \Longrightarrow \eqref{list-Lp}$ and $\eqref{list-Ai} \Longrightarrow \eqref{list-SN}$ are essentially contained in Theorems 1.4.13 and 1.5.10 in \cite{K} respectively. Then, these along with Theorem \ref{thm:main} imply Theorem \ref{thm:CAD} as desired. We should point out that among all equivalences above, $\eqref{list-Ai} \iff \eqref{list-CME}$ seems to be the most difficult one, because once it is established other equivalences are easier to show.

Recall that $\BMO_1$ domains in $\ree$ are domains whose boundaries are given locally in some $\mathscr{C}^{\infty}$ coordinate system as the graph of a function $\phi$ with $\nabla \phi \in \BMO$ (cf. \cite[Definition 2.6]{JK}). As observed in \cite{JK}, Lipschitz domains are a subclass of $\BMO_1$ domains, which turn out to be a subclass of NTA domains (cf. Definition \ref{def:NTA}). 

It was proved that the harmonic measure $\w$ belongs to $A_{\infty}(\sigma)$ in Lipschitz domains by \cite{D77, D79}, and in $\BMO_1$ domains by \cite[Theorem 10.1]{JK}. Thus, in both scenarios, thanks to Theorem \ref{thm:CAD}, it yields properties \eqref{list-Lp}--\eqref{list-BMO} for the Laplacian. On the other hand, we recover the results in Lipschitz domains \cite{DKP} for the second order divergence form elliptic operators $L$, which established the equivalence between $\w_L \in A_{\infty}(\sigma)$ and the $\BMO$-solvability of $L$.

\vspace{0.2cm}\noindent{$\bullet$ \textbf{Application 2: $t$-independent elliptic operators.}}  
Let $\mu: \Rn \to \R^+$ be any $A_2$-weight on $\Rn$ and then define a weight $w$ on $\Omega=\R^{n+1}_+$ by 
$w(x, t) = \mu(x)$. We put the measures $dm(x, t) = w(x, t)dxdt$ on $\Omega$ and $d\mu = \mu(x)dx$ on $\pom=\Rn$ respectively. Note that $m$ is a doubling measure on $\R^{n+1}_+$ because $d\mu$ is doubling on $\Rn$. Then we easily see that the conditions \eqref{H1}--\eqref{H5} are satisfied since $\mu$ is doubling and $\rho \simeq 1$ in \eqref{H5}. Besides, \eqref{H6} can be deduced from \cite[Theorem 1.5]{FKS}  and the hypothesis $\mu \in A_2$.

%%%%%%%%%%%%%%%%%%%%%%% THEOREM THEOREM THEOREM %%%%%%%%%%%%%%%%%%%%%
\begin{theorem}\label{thm:t-independent}
Let $n \ge 2$, and let $Lu=-\div(A\nabla u)$ be a real elliptic operator in $\R^{n+1}_+$ with the $t$-independent coefficients matrix $A(x, t)=A(x)$ satisfying \eqref{eq:elli} for $w(x, t) = \mu(x)$, where $\mu \in A_2(\Rn)$. Then, all the properties \eqref{list-Ai}--\eqref{list-BMO} in Theorem $\ref{thm:main}$ hold. 
\end{theorem}
%%%%%%%%%%%%%%%%%%%%%%% THEOREM THEOREM THEOREM %%%%%%%%%%%%%%%%%%%%%

In this scenario, it was shown in \cite{HLM} that the Dirichlet problem \eqref{eq:D} for $L$ is solvable in $L^p(\mu)$ for some $p \in (1, \infty)$, and that every bounded weak solution of $L$ satisfies Carleson measure estimates. Therefore, this and  Theorem~\ref{thm:main} imply Theorem \ref{thm:t-independent}.  On the other hand, if we take $\mu \equiv 1$ above, then \eqref{H6'} is verified and it will recover Theorems 1.7, 1.10, and 1.23 in \cite{HKMP}.

\vspace{0.2cm}\noindent{$\bullet$ \textbf{Application 3: Degenerate elliptic operators on lower-dimensional sets.}}
The next example is $\Omega=\ree \setminus \Gamma$, where $\Gamma$ is a $d$-ADR set. Set $\mu=\restr{\H^d}{\Gamma}$. We assume that $\Omega$ satisfies 1-sided NTA conditions \eqref{H1} and \eqref{H2}, which hold automatically when $d<n$ (cf. \cite{DFM2}). The simplest option is to take $w(X) =\dist(X, \Gamma)^{-\gamma}$ for some $\gamma \in (n-d-1, n-d+1)$. Then $w$ is locally integrable by a simple estimate on the measure of the $\varepsilon$-neighborhoods of $\pom$. The same estimates yield that $m(B(X, r)) \simeq r^{n+1-\gamma}$ when $\delta(X) < 4r$ and $m(B(X,r)) \simeq r^{n+1} \delta(X)^{-\gamma}$ when $\delta(X) > 2r$. This implies that $m$ is doubling, then both \eqref{H4} and \eqref{H5} hold as soon as $\gamma>n-d-1$. Beyond that, it is easy to check that \eqref{H6'} holds.  

Now we are ready to apply our main Theorem \ref{thm:main} to the current setting. The first one is established in the complement of a Lipschitz graph with small enough constant.  
%%%%%%%%%%%%%%%%%%%%%%% THEOREM THEOREM THEOREM %%%%%%%%%%%%%%%%%%%%%
\begin{theorem}\label{thm:lower-1}
 Let $\Gamma \subset \ree$ be the graph of a Lipschitz function $\varphi: \R^d \to \R^{n-d+1}$, where $d<n$. Define $Lu=-\div(A \nabla u)$ in $\ree \setminus \Gamma$ with 
\[
A(X)=\bigg(\int_{\Gamma} \frac{d\mu(y)}{|X-y|^{d+\alpha}} \bigg)^{-\frac{1}{\alpha}}, \quad\text{ for some } \alpha>0. 
\] 
If the Lipschitz constant of $\varphi$ is small enough, then all the properties \eqref{list-Ai}--\eqref{list-BMO} in Theorem $\ref{thm:main}$ hold for $\mu=\restr{\H^d}{\Gamma}$. 
\end{theorem}
%%%%%%%%%%%%%%%%%%%%%%% THEOREM THEOREM THEOREM %%%%%%%%%%%%%%%%%%%%%

Indeed, under the background hypotheses above, we observe that $\Gamma$ is $d$-ADR, and that $A(X) \simeq \dist(X, \Gamma)^{d-n}$ uniformly for $X \in \ree \setminus \Gamma$ (cf. \cite[Lemma~5.1]{DFM1}). The authors in \cite[Theorem 1.18]{DFM1} showed that $\w_L \in A_{\infty}(\mu)$, which together with Theorem \ref{thm:main} and the preceding argument yields at once Theorem \ref{thm:lower-1}. We should mention that \cite[Theorem 1.27]{DFM1} has already proved Carleson measure estimates imply $\w_L \in A_{\infty}$ (even without smallness of the Lipschitz constant). 

Moreover, based on Theorem \ref{thm:main} and \cite[Corollary~8.11]{DFM1}, we obtain the following. 

%%%%%%%%%%%%%%%%%%%%%%% THEOREM THEOREM THEOREM %%%%%%%%%%%%%%%%%%%%%
\begin{theorem}\label{thm:lower-2}
Let $\mathcal{A}$ be a $(n+1)\times (n+1)$ matrix on $\ree$ satisfying the uniform elliptic conditions. Assume that $\mathcal{A}$ has the following structure:
\begin{align*}
\mathcal{A}(X) := 
\bigg(
 \begin{matrix}
\mathcal{A}_1 & \mathcal{A}_2 \\
\mathcal{A}_3 & bI_{n-d+1} + \mathcal{A}_4 \\
\end{matrix}
\bigg), 
\end{align*} 
where $d<n$, $\mathcal{A}_1$ and $\mathcal{A}_2$ can be any matrix valued measurable functions in respectively $\mathcal{M}_{d \times d}$ and $\mathcal{M}_{d \times (n-d+1)}$, $I_{n-d+1} \in \mathcal{M}_{(n-d+1) \times (n-d+1)}$ denotes the identity matrix, and $\lambda^{-1} \le b \le \lambda$ for some constant $\lambda>0$, 
\begin{align*} 
\sup_{x \in \R^d,\, r>0} \frac{1}{r^d} \iint_{B(x, r) \cap \ree} |\nu(y, t)|^2 \, \frac{dydt}{|t|^{n+1-d}} < \infty, 
\qquad\forall \nu \in \{|t| \nabla b, \, \mathcal{A}_3, \mathcal{A}_4\}. 
\end{align*} 
Let $Lu=-\div(A\nabla u)$ with $A(X)=\dist(X, \R^d)^{d-n} \mathcal{A}(X)$. Then all the properties \eqref{list-Ai}--\eqref{list-BMO} in Theorem $\ref{thm:main}$ hold, where $\mu$ is the Lebesgue measure on $\R^d$. 
\end{theorem}
%%%%%%%%%%%%%%%%%%%%%%% THEOREM THEOREM THEOREM %%%%%%%%%%%%%%%%%%%%%

This paper is organized as follows. Section~\ref{sec:prelim} contains some preliminaries, definitions, and tools that will be used throughout. Also, in Section~\ref{sec:PDE}, we collect the general degenerate elliptic theory concerning elliptic measures and Green functions which can be found in \cite{DFM3, FKJ, FKS}. Section \ref{sec:aux} presents some auxiliary results and key estimates to show our mains theorems. The proof of Theorem~\ref{thm:main} is given in Section~\ref{sec:main}.  Section \ref{sec:abs} is devoted to proving Theorem \ref{thm:abs}.

We should point out that the related work \cite{FP} has some overlap with our results because our framework taken from \cite{DFM3} falls inside the definition of PDE friendly domains in \cite{FP}. By $\rm{(\ref{list-CME}')}$, we mean that Carleson measure estimates hold for all solutions of the form $u(X)=\w_L^X(S)$ with $S \subset \pom$ being an arbitrary Borel set. Thus, if \cite[Corollary~1.32]{FP} is restricted to the setting of the current paper, it corresponds to $\eqref{list-Ai} \iff \rm{(\ref{list-CME}')}$ in Theorem \ref{thm:main}. To show $\rm{(\ref{list-CME}')}\ \Longrightarrow\ \eqref{list-Ai}$, both arguments borrowed the ideas from \cite{KKiPT, KKoPT} to establish a lower bound for dyadic square function. However, for the proof of  $\eqref{list-Ai} \Longrightarrow \rm{(\ref{list-CME}')}$, the authors in \cite{FP} used the local $\S<\N$ estimate and the transference of Carleson measure estimates between two elliptic/surface measures which are quantitatively mutually absolutely continuous, while our strategy is to prove $\eqref{list-Ai} \Longrightarrow \eqref{list-BMO} \Longrightarrow \rm{(\ref{list-CME}')}$.

%%%%%%%%%%%%%%%%%%%%%%% SECTION SECTION SECTION %%%%%%%%%%%%%%%%%%%%%%%
%%%%%%%%%%%%%%%%%%%%%%% SECTION SECTION SECTION %%%%%%%%%%%%%%%%%%%%%%%
\section{Preliminaries}\label{sec:prelim}

%%%%%%%%%%%%%%%%%% SUBSECTION SUBSECTION SUBSECTION %%%%%%%%%%%%%%%%%%%%%
\subsection{Notation}\label{sec:notation}
Let us introduce some notation used throughout this article. 
\begin{list}{$\bullet$}{\leftmargin=0.8cm  \itemsep=0.2cm}
	
\item  Our ambient space is $\ree$, $n\ge 1$.

\item Given $E \subset \ree$ we write $\diam(E)=\sup_{x,y \in E}|x-y|$ to denote its diameter.
	
\item Given an open set $\Omega \subset \ree$, we shall use lower case letters $x, y, z$, etc., to denote points on $\pom$, and capital letters $X, Y, Z$, etc., to denote generic points in $\ree$ (especially those in $\ree \setminus \pom$). 
	
\item The open $(n+1)$-dimensional Euclidean ball of radius $r$ will be denoted $B(x,r)$ when the center $x$ lies on $\pom$, or $B(X, r)$ when the center $X \in \ree \setminus \pom$.  A \textit{surface ball} is denoted $\Delta(x, r) := B(x,r) \cap \pom$, and unless otherwise specified it is implicitly assumed that $x \in \pom$.
	
\item If $\pom$ is bounded, it is always understood (unless otherwise specified) that all surface balls have radii controlled by the diameter of $\pom$, that is, if $\Delta=\Delta(x,r)$ then $r\lesssim \diam(\pom)$. Note that in this way $\Delta=\pom$ if $\diam(\pom)<r\lesssim \diam(\pom)$. 
	
\item For $X \in \ree$, we set $\delta(X):= \dist(X,\partial\Omega)$.
	
\item We let $\H^d$ denote the $d$-dimensional Hausdorff measure. 
	
\item For a Borel subset $A \subset \pom$ with $0<\mu(A)<\infty$, we set $\fint_A f d\mu := \mu(A)^{-1} \int_A f d\mu$.
	
\item We shall use the letter $I$ (and sometimes $J$) to denote a closed $(n+1)$-dimensional Euclidean cube with sides parallel to the coordinate axes, and we let $\ell(I)$ denote the side length of $I$. We use $Q$ to denote  dyadic ``cubes'' on $\pom$.  The latter exist as a consequence of Lemma \ref{lem:dyadic} below.

\item We use the letters $c,C$ to denote harmless positive constants, not necessarily the same at each occurrence, which depend only on dimension and the constants appearing in the hypotheses of the theorems (which we refer to as the ``allowable parameters'').  We shall also sometimes write $a\lesssim b$ and $a \simeq b$ to mean, respectively, that $a \leq C b$ and $0< c \leq a/b\leq C$, where the constants $c$ and $C$ are as above, unless explicitly noted to the contrary.   Unless otherwise specified upper case constants are greater than $1$  and lower case constants are smaller than $1$. In some occasions it is important to keep track of the dependence on a given parameter $\gamma$, in that case we write $a\lesssim_\gamma b$ or $a \simeq_\gamma b$ to emphasize  that the implicit constants in the inequalities depend on $\gamma$.

\end{list}

%%%%%%%%%%%%%%%%%% SUBSECTION SUBSECTION SUBSECTION %%%%%%%%%%%%%%%%%%%%%
\subsection{Geometric concepts}\label{sec:defs}

%%%%%%%%%%%%%%%%%%%%% DEFINITION DEFINITION DEFINITION %%%%%%%%%%%%%%%%%%%%
\begin{definition}
Let $\mu$ be a nonnegative locally integrable function on $\Rn$. Given $1 \le p<\infty$, we say that $\mu \in A_p(\Rn)$ if 
\[
[\mu]_{A_p(\Rn)} := \sup_{Q} \bigg(\fint_Q \mu\, dx \bigg) \bigg(\fint_Q \mu^{1-p'}\, dx \bigg)^{p-1} < \infty, 
\]
where the supremum is taken over all cubes $Q \subset \Rn$. 
\end{definition}
%%%%%%%%%%%%%%%%%%%%% DEFINITION DEFINITION DEFINITION %%%%%%%%%%%%%%%%%%%%

Observe that all $A_p$ weights with $1\le p<\infty$ enjoy the doubling property (cf. \cite[p. 504]{G}): 
\begin{align*}
\mu(\lambda Q) \le [\mu]_{A_p} \lambda^{np} \mu(Q), \quad \forall \mu \in A_p(\Rn),  \, \lambda>1, 
\end{align*}
for all cubes $Q \subset \Rn$.

%%%%%%%%%%%%%%%%%%%%% DEFINITION DEFINITION DEFINITION %%%%%%%%%%%%%%%%%%%%
\begin{definition}[\bf Ahlfors-David regular]\label{def:ADR}
Let $0<d \le n+1$ and let $E \subset \ree $ be a closed set. We say that $E$ is $d$-dimensional \textit{Ahlfors-David regular} ($d$-ADR for short) if there is some uniform constant $C \ge 1$ such that 
\begin{equation*}
C^{-1}\, r^d \leq \H^d(E \cap B(x,r)) \leq C\, r^d,\qquad \forall x \in E, \quad 0<r<\diam(E). 
\end{equation*} 
\end{definition}
%%%%%%%%%%%%%%%%%%%%% DEFINITION DEFINITION DEFINITION %%%%%%%%%%%%%%%%%%%%

%%%%%%%%%%%%%%%%%%%%% DEFINITION DEFINITION DEFINITION %%%%%%%%%%%%%%%%%%%%
\begin{definition}[\bf Corkscrew condition]\label{def:cks}
We say that a domain $\Omega\subset \ree$ satisfies the \textit{corkscrew condition} if for some uniform constant $0<c_0<1$ and for every surface ball $\Delta:=\Delta(x,r)$ with $x\in \partial\Omega$ and $0<r<\diam(\partial\Omega)$, there is a ball $B(X_\Delta, c_0 r) \subset B(x,r) \cap \Omega$. The point $X_\Delta \subset \Omega$ is called a \textit{corkscrew point relative to} $\Delta$. We note that  we may allow $r<C\diam(\pom)$ for any fixed $C$, simply by adjusting the constant $c_0$.
\end{definition}
%%%%%%%%%%%%%%%%%%%%% DEFINITION DEFINITION DEFINITION %%%%%%%%%%%%%%%%%%%%

%%%%%%%%%%%%%%%%%%%%% DEFINITION DEFINITION DEFINITION %%%%%%%%%%%%%%%%%%%%
\begin{definition}[\bf Harnack chain condition]\label{def1.hc}
We say that a domain $\Omega \subset \ree$ satisfies the \textit{Harnack chain condition} if there are uniform constants $C_1,C_2>1$ such that for every pair of points $X, Y\in \Omega$ there is a chain of balls $B_1, B_2, \dots, B_N\subset \Omega$ with $N \leq  C_1(2+\log_2^+ \Pi)$, where $\Pi:=\frac{|X-Y|}{\min\{\delta(X), \delta(Y)\}}$ such that $X\in B_1$, $Y \in B_N$, $B_k\cap B_{k+1}\neq\varnothing$ and $C_2^{-1} \diam(B_k) \leq \dist(B_k,\partial\Omega) \leq C_2 \diam(B_k)$ for every $1\le k\le N$. The chain of balls is called a \textit{Harnack chain}.
\end{definition}
%%%%%%%%%%%%%%%%%%%%% DEFINITION DEFINITION DEFINITION %%%%%%%%%%%%%%%%%%%%

%%%%%%%%%%%%%%%%%%%%% DEFINITION DEFINITION DEFINITION %%%%%%%%%%%%%%%%%%%%
\begin{definition}[\bf 1-sided NTA, CAD]\label{def:NTA}
We say that $\Omega \subset \ree$ is a $1$-sided NTA (non-tangentially accessible) domain if  $\Omega$ satisfies both the corkscrew and Harnack Chain conditions. Furthermore, we say that $\Omega$ is an NTA domain if it is a $1$-sided NTA domain and if, in addition, $\R^{n+1} \setminus \overline{\Omega}$ satisfies the Corkscrew condition. If a $1$-sided NTA domain, or an NTA domain, has an ADR boundary, then it is called a 1-sided CAD (chord-arc domain) or a CAD, respectively.
\end{definition}
%%%%%%%%%%%%%%%%%%%%% DEFINITION DEFINITION DEFINITION %%%%%%%%%%%%%%%%%%%%

%%%%%%%%%%%%%%%%%% SUBSECTION SUBSECTION SUBSECTION %%%%%%%%%%%%%%%%%%%%%
\subsection{Dyadic grids and sawtooths}\label{sec:dyadic} 
In what follows, we assume that $\Omega$ satisfies \eqref{H1}-\eqref{H2} and that the measures $\mu$ and $m$ satisfy \eqref{H3} and \eqref{H4}. We present a lemma describing the dyadic decomposition of $(\pom, \mu)$.  

%%%%%%%%%%%%%%%%%%%%%%%%%% LEMMA LEMMA LEMMA %%%%%%%%%%%%%%%%%%%%%%
\begin{lemma}[\cite{Ch, DFM3}]\label{lem:dyadic}
There exist constants $a_0 \in (0, 1]$, $A_0 \in [1,\infty)$, $\gamma \in (0,1)$, depending only on $c_3$, such that for each $k \in \Z$ there is a collection of Borel sets  (called ``dyadic cubes'')
\[
\D_k:=\big\{Q_j^k\subset \pom:\ j\in\mathfrak{J}_k\big\},
\]
where $\mathfrak{J}_k$ denotes some (possibly finite) index set depending on $k$ satisfying:
\begin{list}{\rm{(\theenumi)}}{\usecounter{enumi}\leftmargin=1cm \labelwidth=1cm \itemsep=0.2cm \topsep=.2cm \renewcommand{\theenumi}{\alph{enumi}}}
\item\label{eq:dyadic-1} $\pom=\bigcup_{j\in \mathfrak{J}_k} Q_j^k$ for each $k \in \Z$. 
\item\label{eq:dyadic-2} If $m\le k$ then either $Q_j^k \subset Q_i^m$ or $Q_i^m\cap Q_j^k=\emptyset$.
\item\label{eq:dyadic-3} For each $k \in \mathbb{Z}$, $j\in\mathfrak{J}_k$, and $m<k$, there is a unique $i\in\mathfrak{J}_m $ such that $Q_j^k \subset Q_i^m$. 
\item\label{eq:dyadic-4} $\diam(Q_j^k)\le A_0 2^{-k}$.
\item\label{eq:dyadic-5} Each $Q_j^k$ contains some surface ball $\Delta(x_j^k, a_0 2^{-k}) :=B(x_j^k, a_0 2^{-k})\cap \pom$. 
\item\label{eq:dyadic-6} $\mu\big(\big\{x\in Q_j^k:\,\dist(x, \pom \setminus Q_j^k) \leq \tau 2^{-k}\big\}\big)\leq A_0 \tau^{\gamma} \mu(Q_j^k)$, for all $j, k \in \Z$, and for some constant $\tau>0$.
\end{list}
\end{lemma}
%%%%%%%%%%%%%%%%%%%%%%%%%% LEMMA LEMMA LEMMA %%%%%%%%%%%%%%%%%%%%%%

A few remarks are in order concerning this lemma. In the setting of a general space of homogeneous type, this lemma was proved by Christ \cite[Theorem 11]{Ch}, with the dyadic parameter $1/2$ replaced by some constant $\delta \in (0, 1)$. In the current scenario, this result first appeared in \cite[Proposition 5.1]{DFM3}. 

Let us observe that if $\pom$ is bounded and $k \in \Z$ is such that $\diam(\pom) < C^{-1} 2^{-k}$, then there cannot be two distinct cubes in $\D_k$. Thus, $\D_k=\{Q^k\}$ with $Q^k=\pom$. Therefore, we are going to ignore those $k \in \Z$ such that $2^{-k} \gtrsim \diam(\pom)$. Hence, we shall denote by $\D=\D(\pom)$ the collection of all relevant $Q_j^k$, i.e.,
\[
\D=\D(\pom):=\bigcup_{k \in \Z} \D_k, 
\]
where, if $\diam(\pom)$ is finite, the union runs over those $k \in\mathbb{Z}$ such that $2^{-k} \lesssim \diam(\pom)$.  An element of $\D$ is given by a subset $Q$ of $\pom$ and a generation $k$. Indeed, if we only know the set $Q$, contrary to dyadic cubes in $\Rn$, we cannot be sure of the generation. Despite this, we shall abuse notation and use the term $Q$ for both an element of $\D$ and the corresponding subset of $\pom$. We write $k(Q)$ when we want to refer to the "dyadic generation`` of the cube $Q \in \D$, that is the only integer $k$ such that $Q \in \D_k$. The length of a dyadic cube is $\ell(Q)=2^{-k(Q)}$. It follows from \eqref{eq:dyadic-4} that $\diam(Q) \lesssim \ell(Q)$.

Properties \eqref{eq:dyadic-4} and \eqref{eq:dyadic-5} imply that for each cube $Q \in \D$, there exists $x_Q \in \pom$ such that 
\begin{equation}\label{deltaQ}
\Delta(x_Q, a_0 \ell(Q)) \subset Q \subset \Delta(x_Q, A_0 \ell(Q)). 
\end{equation}
We shall denote these balls and surface balls by
\begin{align}
\label{deltaQ2} B_Q:=B(x_Q, a_0 \ell(Q)), \qquad & \Delta_Q:=\Delta(x_Q, a_0 \ell(Q)), 
\\ 
\label{deltaQ3} \widetilde{B}_Q:=B(x_Q, A_0 \ell(Q)),\qquad &\widetilde{\Delta}_Q:=\Delta(x_Q, A_0 \ell(Q)),
\end{align}
and we shall refer to the point $x_Q$ as the ``center'' of $Q$. Given $Q \in \D$, we define the ``Corkscrew point relative to $Q$" as $X_Q :=X_{\Delta_Q}$ (cf. \eqref{H1} for the existence). Then, we have 
\[
\delta(X_Q)\simeq \dist(X_Q,Q) \simeq \ell(Q).
\]

We next introduce the ``discretized Carleson region'' relative to $Q\in \D$, $\D_{Q}=\{Q'\in \D: Q' \subset Q\}$. Let $\F=\{Q_j\}\subset \D$ be a family of pairwise disjoint cubes. The ``global discretized sawtooth'' relative to $\F$ is the collection of cubes $Q \in \D$ that are not contained in any $Q_j \in \F$, that is,
\[
\D_{\F} := \D \setminus \bigcup_{Q_j \in \F} \D_{Q_j}.
\]
For a given $Q \in \D$, the ``local discretized sawtooth'' relative to $\F$ is the collection of cubes in $\D_Q$ that are not contained in any $Q_j \in \F$ or, equivalently,
\[
\D_{\F, Q} := \D_Q \setminus \bigcup_{Q_j \in \F} \D_{Q_j} = \D_{\F} \cap \D_Q.
\]
We also allow $\F$ to be the null set in which case $\D_{\tinyemptyset}=\D(E)$ and $\D_{\tinyemptyset,Q}=\D_Q$.

We also introduce the ``geometric'' Carleson regions and sawtooths, which originated in \cite{HM} and were widely used to study elliptic operators (cf. \cite{AHMMT, CDMT, CHMT, HMM, HMMTZ}).  Given $Q\in\D$ we want to define some associated regions which inherit the good properties of $\Omega$. Let $\W=\W(\Omega)$ denote a collection of (closed) dyadic Whitney cubes of $\Omega \subset \ree$, so that the cubes in $\W$ form a covering of $\Omega$ with non-overlapping interiors, and satisfy
\begin{equation*}
4\diam(I) \leq \dist(4I, \pom) \leq \dist(I, \pom) \leq 40 \diam(I), \qquad\forall I \in \W,
\end{equation*}
and
\[
\diam(I_1)/4 \le \diam(I_2) \le 4 \diam(I_1),\,\text{ whenever } \partial I_1 \cap \partial I_2 \neq \emptyset. 
\] 
Let $X_I$ denote the center of $I$, let $\ell(I)$ denote the side length of $I$, and write $k(I)=k$ if $\ell(I)=2^{-k}$.

Given $0<\lambda<1$ and $I\in\mathcal{W}$ we write $I^*=(1+\lambda)I$ for the ``fattening'' of $I$. By taking $\lambda$ small enough, we can arrange matters, so that, first, $\dist(I^*,J^*)\approx\dist(I,J)$ for every $I,J\in\mathcal{W}$. Secondly, $I^*$ meets $J^*$ if and only if $\partial I$ meets $\partial J$ (the fattening thus ensures overlap of $I^*$ and $J^*$ for any pair $I,J\in\mathcal{W}$ whose boundaries touch, so that the Harnack Chain property then holds locally in $I^*\cup J^*$, with constants depending upon $\lambda$). By picking $\lambda$ sufficiently small, say $0<\lambda<\lambda_0$, we may also suppose that there is $\tau\in(\frac12,1)$ such that for distinct $I,J\in\mathcal{W}$, we have that $\tau J\cap I^*=\emptyset$. In what follows we will need to work with dilations $I^{**}=(1+2\lambda)I$ or $I^{***}=(1+4\lambda)I$, and in order to ensure that the same properties hold we further assume that $0<\lambda<\lambda_0/4$.

Given $\vartheta\in\mathbb{N}$,  for every cube $Q \in \D$ we set
\begin{equation}\label{eq:WQ}
\W_Q :=\big\{I \in \W: 2^{-\vartheta} \ell(Q) \leq \ell(I) \leq 2^\vartheta\ell(Q), 
\text { and } \dist(I, Q) \leq 2^\vartheta \ell(Q) \big\}.
\end{equation}
We will choose $\vartheta \ge \vartheta_0$,  with $\vartheta_0$ large enough depending only on $d$, $n$ and the constants of the dyadic cube construction (cf. Lemma \ref{lem:dyadic}), so that $X_Q \in I$ for some $I \in \W_Q$, and for each proper child $Q^j$ of $Q$, the respective corkscrew points $X_{Q^j} \in I^j$ for some $I^j  \in \W_Q$. 

For each $I \in \W_Q$, by \cite[Proposition~2.18]{DFM3}, we can form a Harnack chain from the center $X_I$ to the corkscrew point $X_Q$ and call it $\H_I$. We now let $\W_{Q}^*$ denote the collection of all Whitney cubes which meet at least one ball in the Harnack chain $\H_I$ with $I \in \W_Q$, that is,
\begin{equation*}
\W_{Q}^*:=\{J \in \W: \text{ there exists } I \in \W_Q \text{ such that } \H_I \cap J \neq \emptyset\}.
\end{equation*}
We also define
\begin{equation*}
U_{Q} :=\bigcup_{I \in \W_{Q}^*}(1+\lambda) I=: \bigcup_{I \in \W_{Q}^*} I^{*}.
\end{equation*}
By construction, we then have that
\begin{equation*}
\W_Q  \subset \W_Q^* \subset \W \quad \text{and}\quad X_Q \in U_Q, \quad X_{Q^j} \in U_Q,
\end{equation*}
for each child $Q^j$ of $Q$. It is also clear that there is a uniform constant $k^*$ (depending only on the allowable constants and $\vartheta$) such that
\begin{align*}
2^{-k^*} \ell(Q) \leq \ell(I) \leq 2^{k^*}\ell(Q), &\quad \forall\,I \in \W_Q^*,
\\ 
X_I \rightarrow_{U_Q} X_Q, &\quad \forall\,I \in \W_Q^*,
\\ 
\dist(I, Q) \leq 2^{k^*} \ell(Q), &\quad \forall\,I \in \W_Q^*.
\end{align*}
Here, $X_I \to_{U_Q} X_Q$ means that the interior of $U_Q$ contains all balls in Harnack chain (in $\Omega$) connecting $X_I$ to $X_Q$, and moreover, for any point $Z$ contained in any ball in the Harnack chain, we have
$\dist(Z, \partial \Omega) \simeq \dist(Z, \Omega \setminus U_Q)$ with uniform control of implicit constants. The
constant $k^*$ and the implicit constants in the condition $X(I) \to_{U_Q} X_Q$, depend at most on the allowable
parameters, on $\lambda$, and on $\vartheta$. Moreover, given $I \in \W$ we have that $I \in \W^*_{Q_I}$, where $Q_I \in \D$ satisfies $\ell(Q_I)=\ell(I)$, and contains any fixed $\widehat{y} \in \pom$ such that $\dist(I, \pom) = \dist(I, \widehat{y})$. The reader is referred to \cite{HM, MP} for full details, where the parameter $\vartheta$ is fixed. Here we need to allow $\vartheta$ to depend on the aperture of the cones and hence it is convenient to include the superindex $\vartheta$.

For a given $Q\in\D$, the ``Carleson box'' relative to $Q$ is defined by
\[
T_Q:=\inter\bigg(\bigcup_{Q' \in \D_Q} U_{Q'} \bigg).
\] 
For a given family $\F=\{Q_i\} \subset \D$ of pairwise disjoint cubes and a given $Q \in \D$, we define the ``local sawtooth region'' relative to $\F$ by
\begin{equation*}
\Omega_{\F, Q} := \inter\bigg(\bigcup_{Q' \in \D_{\F, Q}} U_{Q'} \bigg)
=\inter\bigg(\bigcup_{I\in\mathcal{W}_{\F, Q}} I^* \bigg),
\end{equation*}
where $\W_{\F, Q} := \bigcup_{Q' \in \D_{\F, Q}} \W_{Q'}^*$. Note that in the previous definition we may allow $\F$ to be empty in which case clearly $\Omega_{\tinyemptyset ,Q}=T_Q$. Similarly, the ``global sawtooth region'' relative to $\F$ is defined as
\begin{equation*}
\Omega_{\F} := \inter\bigg(\bigcup_{Q' \in \D_{\F}} U_{Q'} \bigg)
=\inter\bigg(\bigcup_{I \in \W_{\F}} I^*\bigg),
\end{equation*}
where $\W_{\F}:=\bigcup_{Q' \in \D_{\F}} \W_{Q'}^*$. If $\F$ is the empty set clearly $\Omega_{\tinyemptyset}=\Omega$.
For a given $Q \in \D$  and $x \in \pom$ let us introduce the dyadic cone and its truncated version 
\[
\Gamma_{\D}(x) := \bigcup_{x \in Q' \in \D}  U_{Q'} \quad\text{ and }\quad 
\Gamma_Q(x) := \bigcup_{x\in Q' \in \D_Q}  U_{Q'},
\]
where it is understood that $\Gamma_Q(x)=\emptyset$ if $x \notin Q$. Analogously, we can slightly fatten the Whitney boxes and use $I^{**}$ to define new fattened Whitney regions and sawtooth domains. More precisely, for every $Q \in \D$,
\[
T_Q^* := \inter\bigg(\bigcup_{Q' \in \D_Q} U_{Q'}^* \bigg),\quad 
\Omega^*_{\F, Q} := \inter\bigg(\bigcup_{Q' \in \D_{\F, Q}} U_{Q'}^* \bigg), \quad
\Gamma^*_Q(x) := \bigcup_{x\in Q'\in\D_{Q_0}}  U_{Q'}^*,
\]
where $U_Q^* := \bigcup_{I \in \W_Q^*} I^{**}$. Similarly, we can define $T_Q^{**}$, $\Omega^{**}_{\F, Q}$, $\Gamma_Q^{**}(x)$, and $U^{**}_{Q}$ by using $I^{***}$ in place of $I^{**}$.

To define  the ``Carleson box'' $T_{\Delta}$ associated with a surface ball $\Delta=\Delta(x,r)$, let $k(\Delta)$ denote the unique $k \in \mathbb{Z}$ such that $2^{-k-1}<200r\leq 2^{-k}$, and set
\begin{equation*}
\D^{\Delta}:=\big\{Q \in \D_{k(\Delta)}: \: Q \cap 2 \Delta \neq \emptyset \big\}.
\end{equation*}
We then define
\begin{equation*}
T_{\Delta} := \inter\bigg(\bigcup_{Q \in \D^{\Delta}} \overline{T_Q} \bigg).
\end{equation*}
We can also consider fattened versions of $T_{\Delta}$ given by
\begin{align*}
T_{\Delta}^* := \inter\bigg(\bigcup_{Q \in \D^{\Delta}} \overline{T_Q^{*}} \bigg) 
\quad \text{ and }\quad 
T_{\Delta}^{**} := \inter\bigg(\bigcup_{Q \in \D^{\Delta}} \overline{T_Q^{**}} \bigg).
\end{align*}

Following \cite{HM, HMT:book}, one can easily see that there exist constants $0<\kappa_1<1$ and $\kappa_0\geq 16A_0/a_0$ (with the constants $a_0$ and $A_0$ in Lemma \ref{lem:dyadic}), depending only on the allowable parameters and on $\vartheta$, so that
\begin{gather}\label{eq:kappa-1}
\kappa_1 B_Q \cap \Omega \subset T_Q \subset T_Q^* \subset T_Q^{**} \subset \overline{T_Q^{**}} 
\subset \kappa_0 B_Q \cap \overline{\Omega}=:\tfrac{1}{2}B_Q^* \cap \overline{\Omega},
\\[6pt]
\label{eq:kappa-2}
\tfrac{5}{4}B_\Delta \cap \Omega \subset T_{\Delta} \subset T_{\Delta}^* \subset T_{\Delta}^{**}
\subset \overline{T_{\Delta}^{**}} \subset \kappa_0 B_\Delta \cap \overline{\Omega} 
=: \tfrac{1}{2} B_\Delta^* \cap \overline{\Omega},
\end{gather}
and also
\begin{equation}\label{propQ0}
Q \subset \kappa_0 B_\Delta \cap \pom = \tfrac{1}{2}B_\Delta^* \cap \pom 
=:\tfrac{1}{2} \Delta^*, \qquad \forall\, Q \in \D^{\Delta},
\end{equation}
where $B_Q$ is defined as in \eqref{deltaQ2}, $\Delta=\Delta(x,r)$ with $x\in\Gamma$, $r>0$, and $B_{\Delta}=B(x, r)$ is so that $\Delta=B_\Delta \cap \Gamma$. From our choice of the parameters one also has that $B_Q^*\subset B_{Q'}^*$ whenever $Q \subset Q'$.

%%%%%%%%%%%%%%%%%%%% SUBSECTION SUBSECTION SUBSECTION %%%%%%%%%%%%%%%%%%%
\subsection{Elliptic measure and Green function}\label{sec:PDE}
Let us recall the elliptic theory of degenerate elliptic operators in \cite{DFM3}, which generalizes the elliptic theory developed in \cite{DFM2}. The latter focused on domains $\Omega \subset \R^{n+1}$ whose boundary $\Gamma = \pom$ is ADR of dimension $d<n$. Such theory has found many applications, for example, \cite{DLM, MP, MZ}.

As before, we assume that $\Omega$ satisfies \eqref{H1}--\eqref{H6}.  Throughout we consider the degenerate elliptic operators $L$ of the form $Lu=-\div(A\nabla u)$ with $A(X)=(a_{i,j}(X))_{i,j=1}^{n+1}$ being a real (not necessarily symmetric) matrix such that the following ellipticity condition holds: 
\begin{align}\label{eq:elliptic} 
\Lambda^{-1} w(X) |\xi|^2 \leq A(X) \xi \cdot \xi 
\quad \text{ and }\quad 
A(X) \xi \cdot \eta \leq \Lambda w(X) |\xi| |\eta|, 
\end{align}
for all $\xi, \eta \in\ree$ and for almost every $X \in \Omega$, where $w$ is the weight associated with the measure $m$ given in \eqref{H4}. We write $L^\top$ to denote the transpose of $L$, or, in other words, $L^\top u = -\div(A^\top \nabla u)$ with $A^\top$ being the transpose matrix of $A$.

Now we define the suitable function spaces. A function $u$ belongs to $W(\Omega)$ if $u \in L^1_{\loc}(\Omega, m)$ and there exists a vector valued function $v \in L^2(\Omega, m)$ such that for some sequence $\{\varphi_j\}_{j \in \mathbb{N}} \subset \mathscr{C}^{\infty}(\overline{\Omega})$, one has 
\begin{align*}
&\sup_{j \in \mathbb{N}} \iint_{\Omega} |\nabla u_j|^2 \, dm < \infty, \quad 
\lim_{j \to \infty} \iint_{\Omega} |\nabla \varphi_j - v|^2 \, dm =0, 
\\
&\lim_{j \to \infty} \iint_B |\varphi_j-u|\, dm =0, \quad\text{ for any ball } B: 2B \subset \Omega, 
\end{align*}
Observe that in view of \eqref{H6}, if $u \in W(\Omega)$, then the vector $v$ from the definition is unique. In what follows, if $u \in W(\Omega)$, we use the notation $\nabla u$ for the unique vector valued function $v$ given by the definition above. In particular, we can equip $W(\Omega)$ with the semi-norm 
\begin{equation*}
\|u\|_{W(\Omega)} := \bigg(\iint_{\Omega} |\nabla u|^2\, dm \bigg)^{\frac12} \quad\text{ for any } u \in W(\Omega). 
\end{equation*}
Under the hypotheses \eqref{H1}, \eqref{H2} and \eqref{H6}, one can see that $\|u\|_{W(\Omega)}=0$ if and only if $u$ is $m$-a.e. equal to a constant function. 

We also define a local version of $W(\Omega)$ as follows: let $E' \subset \Rn$ be an open set, write $E=E' \cap \overline{\Omega}$, and define 
\begin{equation*}
W_r(E) := \{u \in L^1_{\loc}(E \cap \Omega, m): 
\varphi u \in W(\Omega) \text{ for all } \varphi \in \mathscr{C}_c^{\infty}(E')\}. 
\end{equation*}
It follows from \cite[Section~10]{DFM3} that 
\begin{align}\label{eq:WrE}
W_r(E) =\{u \in L^1_{\loc}(E, m): \nabla u \in L^2_{\loc}(E, m)\}. 
\end{align}

Given an open set $E \subset \Omega$, we say that $u \in W_r(E)$ is a weak solution to $Lu=0$ in $E$ if 
\[
\iint_{\Omega} A\nabla u \cdot \nabla\varphi \, dX=0,  \quad\forall \varphi\in \mathscr{C}^{\infty}_{c}(E).
\]
Similarly, $u \in W_r(E)$ is a subsolution (respectively supersolution) to $Lu=0$ in $\Omega$ when 
\[
\iint_{\Omega} A\nabla u \cdot \nabla \varphi \, dX \le 0 \, (\text{resp.} \ge 0),  \quad\forall \varphi\in \mathscr{C}^{\infty}_{c}(E): \varphi \ge 0.
\]

As shown in \cite[Corollary 7.9]{DFM3}, the Poincar\'{e} inequality in \eqref{H6} can be extended to the following: 

%%%%%%%%%%%%%%%%%%%%%%%%% LEMMA LEMMA LEMMA %%%%%%%%%%%%%%%%%%%%%%%
\begin{lemma}\label{lem:Poin}
There exists $p_0 \in (1, 2)$ and $k>1$ such that for any $p \in [p_0, 2]$ and $u \in W(\Omega)$, 
\begin{align*}
\bigg(\bariint_{B} |u-u_B|^{pk} \, dm \bigg)^{\frac{1}{pk}} 
\le C \, r(B) \bigg(\bariint_{B} |\nabla u|^p \, dm \bigg)^{\frac1p},  
\end{align*}
where $u_B := \bariint_{B} u \, dm$. 
\end{lemma}
%%%%%%%%%%%%%%%%%%%%%%%%% LEMMA LEMMA LEMMA %%%%%%%%%%%%%%%%%%%%%%%

The following result records properties of solutions (cf. \cite[Section 11]{DFM3}). 

%%%%%%%%%%%%%%%%%%%%%%%%% LEMMA LEMMA LEMMA %%%%%%%%%%%%%%%%%%%%%%%
\begin{lemma}\label{lem:PDE}
Let $B$ be a ball such that $2B \subset \Omega$. Suppose that $u$ is a non-negative subsolution of $Lu=-\div(A \nabla u)$ in $2B$, where $A$ satisfies \eqref{eq:elliptic}. Then there exists a constant $C$ (depending only on the allowable parameters) so that the following estimates hold: 

\begin{list}{\rm{(\theenumi)}}{\usecounter{enumi}\leftmargin=1cm \labelwidth=1cm \itemsep=0.1cm \topsep=.2cm \renewcommand{\theenumi}{\alph{enumi}}}

\item\label{list:Cacci} Caccioppoli inequality: 
\[
\iint_{B} |\nabla u|^2\, dm \le C \, r(B)^{-2} \iint_{2B} u^2\, dm. 
\]

\item\label{list:Moser} Moser estimate: 
\[
\sup_B u \le C \bigg(\bariint_{2B} u^p\, dm \bigg)^{\frac1p},\qquad 0<p<\infty. 
\]

\item\label{list:Harnack} Harnack inequality: 
\[
\sup_B u \le C \inf_B u, \quad\text{ whenever  $Lu=0$ in $2B$}. 
\]

\item\label{list:Holder} H\"{o}lder regularity: Let $B(X, 2R) \subset \Omega$ with $X \in \Omega$ and $R>0$, and let $u \in W_r(B(X, 2R))$ be a solution in $B(X, 2R)$. Write $\osc_B u:=\sup_B u - \inf_B u$. Then there exist $\varrho \in (0, 1]$ and $C>0$ such that for any $0<r<R$, 
\begin{align*}
\osc_{B(X, r)} u \le C \bigg(\frac{r}{R}\bigg)^{\varrho} \bigg(\bariint_{B(X, R)} u^2\, dm\bigg)^{\frac12}. 
\end{align*}
Moreover, the boundary H\"{o}lder continuity holds as well. That is, the same estimate above is true for balls $B$ centered on $\pom$. 
\end{list}
\end{lemma}
%%%%%%%%%%%%%%%%%%%%%%%%% LEMMA LEMMA LEMMA %%%%%%%%%%%%%%%%%%%%%%%

Next, let us introduce the degenerate elliptic measure associated with $L$. Associated with $L$ one can construct a family of regular probability measures $\{\w_L^X\}_{X \in \Omega}$, called the degenerate elliptic measure (see \cite{DFM3} for full details), such that for a given $f\in \mathscr{C}_c(\pom)$, if we define 
\[
u(X):=\int_{\pom} f(z) \, d\w^{X}_{L}(z), \quad \, X \in \Omega, 
\]
and $u:=f$ on $\pom$, then $u \in W_r(\Omega) \cap \mathscr{C}(\Omega)$ is a weak solution of $Lu=0$ in $\Omega$.

To proceed, we recall the Green function constructed in \cite{DFM3}, which generalized the results from \cite{FJK, FKJ} in the bounded domains.  
%%%%%%%%%%%%%%%%%%%%%%%%% LEMMA LEMMA LEMMA %%%%%%%%%%%%%%%%%%%%%%%
\begin{lemma}\label{lem:Green}
Given $L=-\div(A\nabla)$ with a coefficients matrix $A$ satisfying \eqref{eq:elliptic}, there exists a function $G_L: \Omega \times \Omega \to [0, \infty]$  such that  
\begin{list}{\textup{(\theenumi)}}{\usecounter{enumi}\leftmargin=1.2cm \labelwidth=1.2cm \itemsep=0.2cm \topsep=.2cm \renewcommand{\theenumi}{\roman{enumi}}} 

\item For any $Y \in \Omega$, $G_L(\cdot, Y) \in W_r(\overline{\Omega} \setminus \{Y\}) \subset L^1_{\loc}(\overline{\Omega} \setminus \{Y\}, dm)$ and $\Tr G(\cdot, Y)=0$ on $\pom$. 

\item $G_L(X,Y)=G_{L^\top}(Y, X) \ge 0$, for all $X\neq Y$.

\item For all $X, Y \in \Omega$ with $|X-Y|<\frac12 \delta(Y)$, 
\[
G_L(X,Y) \simeq \int_{|X-Y|}^{\delta(Y)} \frac{t^2}{m(B(Y, t))} \frac{dt}{t}, 
\] 
where the implicit constant depends only on the allowable parameters.  
\item For all $X, Y \in \Omega$ with $|X-Y|>\frac{1}{10} \delta(Y)$, 
\[
G_L(X,Y) \lesssim \frac{|X-Y|^2}{\mu(B(Y, |X-Y|) \cap \Omega)}, 
\] 
where the implicit constant depends only on the allowable parameters.  
\item For any $Y \in \Omega$, 
\begin{equation*}
\iint_{\Omega} A(X) \nabla_X G_L(X, Y) \cdot \nabla \Phi(X)\,dX 
=\Phi(Y), \quad\forall\, \Phi \in  \mathscr{C}_c^{\infty}(\Omega). 
\end{equation*}
In particular, $G_L(\cdot,Y)$ is a weak solution to $L G_L(\cdot,Y)=0$ in the open set $\Omega \setminus \{Y\}$.
\end{list}
\end{lemma}
%%%%%%%%%%%%%%%%%%%%%%%%% LEMMA LEMMA LEMMA %%%%%%%%%%%%%%%%%%%%%%%

The following result lists a number of properties (cf. \cite{DFM3, FKJ}) which will be used below. 

%%%%%%%%%%%%%%%%%%%%%%%%% LEMMA LEMMA LEMMA %%%%%%%%%%%%%%%%%%%%%%%
\begin{lemma}\label{lem:wL-G}
Let $L=-\div(A\nabla)$ with a coefficients matrix $A$ satisfying \eqref{eq:elliptic}. Let $B=B(x_0, r)$ with $x_0 \in \pom$, $0<r<\diam(\pom)$, and $\Delta=B \cap \pom$. Then the following hold:

\begin{list}{\rm{(\theenumi)}}{\usecounter{enumi}\leftmargin=1cm \labelwidth=1cm \itemsep=0.1cm \topsep=.2cm \renewcommand{\theenumi}{\alph{enumi}}}
		
\item\label{list-1} For any $\lambda>1$, 
\[
\w_L^X(\Delta) \ge C_{\lambda}^{-1}, \quad\forall X \in B(x_0, r/\lambda) \cup B(X_{\Delta}, \delta(X_{\Delta})/\lambda). 
\]
		
\item\label{list-2} For any $X \in \Omega \setminus 2B$, 
\[
\w_L^X(\Delta)  \simeq  r^{-2} \, m(B \cap \Omega) \, G_L(X, X_{\Delta}) .
\]
			
\item\label{list-3}  For any $\lambda>1$ and $X \in \Omega \setminus 2 \lambda B$, 
\[
\w_L^X(2\Delta) \leq C_{\lambda} \, \w_L^X(\Delta). 
\]
		
\item\label{list-4} Let $E, F \subset \Delta$ be two Borel sets of $\pom$ such that $\w_L^{X_{\Delta}}(E)>0$ and $\w_L^{X_{\Delta}}(F)>0$. Then 
\[
\frac{\w_L^X(E)}{\w_L^X(F)} \simeq  \frac{\w_L^{X_{\Delta}}(E)}{\w_L^{X_{\Delta}}(F)}, \quad\forall X \in \Omega \setminus 2B. 
\]
In particular, 
\[
\frac{\w_L^X(E)}{\w_L^X(\Delta)} \simeq  \w_L^{X_{\Delta}}(E), \quad\forall X \in \Omega \setminus 2B. 
\]

\end{list}
\end{lemma}
%%%%%%%%%%%%%%%%%%%%%%%%% LEMMA LEMMA LEMMA %%%%%%%%%%%%%%%%%%%%%%%

%%%%%%%%%%%%%%%%%%% DEFINITION DEFINITION DEFINITION %%%%%%%%%%%%%%%%%%%%%%
\begin{definition}[Reverse H\"older and $A_\infty$ classes]\label{def:Ainfty}
Fix a cube $\Delta_0=B_0 \cap \pom$ where $B_0=B_0(x_0, r_0)$ with $x_0 \in \pom$ and $0<r<\diam(\pom)$. Given $p \in (1, \infty)$, we say that the degenerate elliptic measure $\w_L \in RH_p(\Delta_0, \mu)$, provided that $\w_L^{X_{\Delta_0}} \ll \mu$ on $\Delta_0$, and there exists $C \geq 1$ such that
\begin{align}\label{eqn:def:RHp}
\bigg(\fint_{\Delta} (k_L)^p \, d\mu \bigg)^{\frac1p} \leq C \fint_{\Delta} k_L \, d\mu,
\end{align}
for all surface balls $\Delta \subset \Delta_0$, where $k_L:=d\w_L^{X_{\Delta_0}}/d\mu$ denotes the Radon-Nikodym derivative. The infimum of the constants $C$ as above is denoted by $[\w_L]_{RH_p(\Delta_0, \mu)}$.
	
Similarly, we say that $\w_L \in RH_p(\mu)$ provided that for every surface ball $\Delta_0 \subset \pom$ one has $\w_L \in RH_p(\Delta_0, \mu)$ uniformly on $\Delta_0$, that is,
\[
[\w_L]_{RH_p(\mu)} :=\sup_{\Delta_0} [\w_L]_{RH_p(\Delta_0, \mu)}<\infty.
\]
Finally, we define 
\[
A_{\infty}(\Delta_0, \mu) := \bigcup_{p>1} RH_p(\Delta_0, \mu) \quad\text{ and }\quad
A_\infty(\mu) := \bigcup_{p>1} RH_p(\mu). 
\]
\end{definition}
%%%%%%%%%%%%%%%%%%% DEFINITION DEFINITION DEFINITION %%%%%%%%%%%%%%%%%%%%%%

Let us record some properties of the $A_{\infty}$ class. For our setting, the following results have appeared in \cite{CF, GR}.
%%%%%%%%%%%%%%%%%%%%%%%%% LEMMA LEMMA LEMMA %%%%%%%%%%%%%%%%%%%%%%%
\begin{lemma}\label{lem:Ainfity}
Let $L=-\div(A\nabla)$ with a coefficients matrix $A$ satisfying \eqref{eq:elliptic}. Then the following are equivalent: 
\begin{enumerate}[{\rm (1)}]
\item $\w_L \in A_{\infty}(\mu)$. 
\vspace{0.2cm}
\item For any $\varepsilon \in (0, 1)$, there exists $\delta \in (0, 1)$ such that given an arbitrary surface ball $\Delta_0 \subset \pom$, for every surface ball $\Delta \subset \Delta_0$, and for every Borel set $E \subset \Delta$, we have 
\begin{align*}
\w_L^{X_{\Delta_0}}(E) \le \delta \w_L^{X_{\Delta_0}}(\Delta) 
 \quad\Longrightarrow\quad \mu(E) \le \varepsilon \mu(\Delta). 
\end{align*}
\vspace{0.2cm}
\item There exist $C, \theta>0$ such that for every surface ball $\Delta_0 \subset \pom$, for every surface ball $\Delta \subset \Delta_0$, and for every Borel set $E \subset \Delta$, 
\[
\frac{\w_L^{X_{\Delta_0}}(E)}{\w_L^{X_{\Delta_0}}(\Delta)} \le C \bigg(\frac{\mu(E)}{\mu(\Delta)}\bigg)^{\theta}. 
\]
\end{enumerate}

\end{lemma}
%%%%%%%%%%%%%%%%%%%%%%%%% LEMMA LEMMA LEMMA %%%%%%%%%%%%%%%%%%%%%%%

Given $\alpha>0$ and $x \in \pom$ we introduce the cone with vertex at $x$ and aperture $\alpha$ defined as $\Gamma^{\alpha}(x) = \{Y \in \Omega: |Y-x| < (1+\alpha) \delta(Y)\}$. One can also introduce the ``truncated cone'', for every $x \in \pom$ and $r>0$ we set  $\Gamma_r^{\alpha}(x) = B(x, r) \cap \Gamma^{\alpha}(x)$.

\begin{definition}\label{def:SN}
Let $\alpha>0$. The conical square function and  the non-tangential maximal function are defined respectively as
\begin{equation*}
\S^{\alpha} u(x):=\bigg(\iint_{\Gamma^{\alpha}(x)} |\nabla u(X)|^2 \delta(X)^2 \, \frac{dm(X)}{m(B(x, \delta(X)))}  \bigg)^{\frac12}, \quad 
\N^{\alpha} u (x) := \sup_{X \in \Gamma^{\alpha}(x)} |u(X)|,
\end{equation*}
for every $x \in \pom$, $u\in W_r(\Omega)$ and $u \in \mathscr{C}(\Omega)$ respectively. Analogously, we write $\S_{\D}$ and $\N_{\D}$ to denote the dyadic non-tangential maximal function and square function respectively. Additionally, for any $r>0$ and $Q \in \D$, $\S^{\alpha}_r$, $\N^{\alpha}_r$, $\S_Q$, and $\N_Q$ stand for the corresponding objects associated to the truncated cone $\Gamma^{\alpha}_r(x)$ and $\Gamma_Q(x)$.
\end{definition}

%%%%%%%%%%%%%%%%%%% DEFINITION DEFINITION DEFINITION %%%%%%%%%%%%%%%%%%%%%%
\begin{definition}[$\BMO$ space]
We say that $f\in \BMO(\mu)$ if 
\begin{equation*}
\|f\|_{\BMO(\mu)} = \sup_{\Delta \subset \pom} \inf_{c \in \R} \fint_{\Delta} |f-c| \, d\mu<\infty, 
\end{equation*}
where the supremum is taken over all surface balls $\Delta \subset \pom$.
\end{definition}
%%%%%%%%%%%%%%%%%%% DEFINITION DEFINITION DEFINITION %%%%%%%%%%%%%%%%%%%%%%

%%%%%%%%%%%%%%%%%%%%%%%% REMARK REMARK REMARK %%%%%%%%%%%%%%%%%%%%%%
\begin{remark}\label{rem:JN}
By the assumption \eqref{H3}, $\mu$ is doubling. Then $f \in \BMO(\mu)$ implies John-Nirenberg's inequality. Thus, for every $1 \le q<\infty$, 
\begin{align*}
\|f\|_{\BMO(\mu)} \simeq \sup_{\Delta \subset \pom} \inf_{c \in \R} \bigg(\fint_{\Delta} |f-c|^q \, d\mu \bigg)^{\frac1q}
\simeq \sup_{\Delta \subset \pom} \bigg(\fint_{\Delta} |f-f_{\Delta}|^q \, d\mu \bigg)^{\frac1q} < \infty,
\end{align*}
where the supremum is taken over all surface balls $\Delta \subset \pom$ and $f_{\Delta}:=\fint_{\Delta} f\, d\mu$. Note that the implicit constants depend only on the allowable parameters.  
\end{remark}
%%%%%%%%%%%%%%%%%%%%%%%% REMARK REMARK REMARK %%%%%%%%%%%%%%%%%%%%%%

Given $X \in \Omega$, we pick $\widehat{x} \in \pom$ such that $|X-\widehat{x}|=\delta(X)$. Then recalling the quantity $\rho$ in \eqref{H5}, we set 
\begin{align}\label{RHO}
\rho(X) := \rho(\widehat{x}, \delta(X)) 
= \frac{m(B(\widehat{x}, \delta(X)) \cap \Omega)}{\delta(X) \mu(B(\widehat{x}, \delta(X)))}. 
\end{align}

%%%%%%%%%%%%%%%%%%% DEFINITION DEFINITION DEFINITION %%%%%%%%%%%%%%%%%%%%%%
\begin{definition}[Solvability, CME, $\mathcal{S}<\mathcal{N}$] \label{def:several}
Let $L=-\div(A\nabla)$ with a coefficients matrix $A$ satisfying \eqref{eq:elliptic}. Let $\rho$ be the function defined in \eqref{RHO}. 
\begin{enumerate}

\item[{\rm (i)}] Given $p \in (1,\infty)$, we say that \textit{the Dirichlet problem \eqref{eq:D} for $L$ is solvable in $L^p(\mu)$} if for a given $\alpha>0$ there exists $C \ge 1$  (depending only on the allowable parameters, $\alpha$ and $p$) such that for every $f \in \mathscr{C}(\pom)$, the solution $u$ of \eqref{eq:D} is given by 
\begin{equation}\label{eq:u-sol}
u(X):=\int_{\pom} f\, d\w_L^X, \qquad X \in \Omega,  
\end{equation}
and satisfies 
\begin{equation*}
\|\N^{\alpha} u\|_{L^p(\pom, \mu)} \leq C \|f\|_{L^p(\pom, \mu)} .
\end{equation*}

\item[{\rm (ii)}] We say that \textit{the Dirichlet problem \eqref{eq:D} for $L$ is solvable in $\BMO(\mu)$} if there exists $C\ge 1$ (depending only on the allowable parameters, $\alpha$ and $p$) such that for every $f \in \mathscr{C}(\pom)$, the solution $u$ of \eqref{eq:D} is given by \eqref{eq:u-sol} and satisfies 
\begin{equation*}
\sup_{\substack{x \in \pom \\ 0<r<\infty}} \frac{1}{\mu(B(x, r) \cap \pom)} 
\iint_{B(x, r) \cap \Omega} |\nabla u(X)|^2 \delta(X) \, \rho(X)^{-1} \, dm(X)
\leq C \|f\|_{\BMO(\mu)}^2. 
\end{equation*}

\item[{\rm (iii)}] We say that \textit{every bounded weak solution of $Lu=0$ satisfies Carleson measure estimates}, if there exists $C\ge 1$ (depending only on the allowable parameters) such that every bounded weak solution $u \in W_r(\Omega) \cap L^{\infty}(\Omega)$ of $Lu=0$ in $\Omega$ satisfies 
\begin{equation*}
\sup_{\substack{x \in \pom \\ 0<r<\infty}} \frac{1}{\mu(B(x, r) \cap \pom)} 
\iint_{B(x, r) \cap \Omega} |\nabla u(X)|^2 \delta(X) \, \rho(X)^{-1} \, dm(X)
\leq C \|u\|_{L^{\infty}(\Omega)}^2. 
\end{equation*}

\item[{\rm (iv)}]  Given $q \in (0, \infty)$, we say that \textit{every weak solution of $Lu=0$ satisfies $\S<\N$ estimates in $L^q(\mu)$} if for some given $\alpha>0$, there exists $C \ge 1$ (depending only on the allowable parameters, $\alpha$, and $q$) such that every weak solution $u \in W_r(\Omega)$ of $Lu=0$ in $\Omega$ satisfies  
\begin{equation*}
\|\S^{\alpha} u\|_{L^q(\pom, \mu)} \leq C \|\N^{\alpha} u\|_{L^q(\pom, \mu)}.
\end{equation*}	
\end{enumerate} 
\end{definition}
%%%%%%%%%%%%%%%%%%% DEFINITION DEFINITION DEFINITION %%%%%%%%%%%%%%%%%%%%%%

%%%%%%%%%%%%%%%%%%%%%%% SECTION SECTION SECTION %%%%%%%%%%%%%%%%%%%%%%%
%%%%%%%%%%%%%%%%%%%%%%% SECTION SECTION SECTION %%%%%%%%%%%%%%%%%%%%%%%
\section{Auxiliary results}\label{sec:aux}
The goal of this section is to present some auxiliary lemmas which will be used to show our main Theorems \ref{thm:main} and \ref{thm:abs}. In what follows, we always assume that $(\Omega, m, \mu)$ satisfies the hypotheses \eqref{H1}--\eqref{H6}.

\subsection{Changing the apertures of cones} 
Recall the truncated conical square function $\mathcal{S}_r^{\alpha}$ in Definition \ref{def:SN}. This section is devoted to presenting estimates between $\mathcal{S}_r^{\alpha_1}$ and $\mathcal{S}_r^{\alpha_2}$ in order to change the apertures of cones conveniently.

Given $0<\gamma<1$ and a closed set $E \subset \pom$, we say that a point $ x\in \pom$ has global $\gamma$-density with respect to $E$, if $\mu(E \cap \Delta(x, r)) \ge \gamma \mu(\Delta(x, r))$ for all balls $\Delta(x, r) \subset \pom$. For any $\alpha>0$, we denote $\mathscr{R}_{\alpha}(E):=\bigcup_{x \in E} \Gamma^{\alpha}(x)$. 

%%%%%%%%%%%%%%%%%%%%%%%%% LEMMA LEMMA LEMMA %%%%%%%%%%%%%%%%%%%%%%%
\begin{lemma}\label{lem:FRF}
Let $0<\alpha<\beta<\infty$. There exists $\gamma \in (0, 1)$ sufficiently close to $1$ so that for any nonnegative function $\Phi$ on $\Omega$ and for any closed set $E \subset \pom$ with $\mu(E^c)<\infty$, 
\begin{align}\label{eq:FRF} 
\int_{E} \iint_{\Gamma^{\alpha}(x)} \frac{\Phi(Y) \, dm(Y)}{m(B(x, \delta(Y)))} d\mu(x) 
\le C_{\alpha} \iint_{\mathscr{R}_{\alpha}(E)}  
\frac{\Phi(Y) \mu(\Delta(\widehat{y}, \delta(Y)))}{m(B(\widehat{y}, \delta(Y)))} \, dm(Y), 
\end{align}
and 
\begin{align}\label{eq:RFF} 
\iint_{\mathscr{R}_{\beta}(E^*)} 
\frac{\Phi(Y) \mu(\Delta(\widehat{y}, \delta(Y)))}{m(B(\widehat{y}, \delta(Y)))} \, dm(Y) 
\le C_{\alpha,\beta,\gamma}  \int_{E} \iint_{\Gamma^{\alpha}(x)} \frac{\Phi(Y) \, dm(Y)}{m(B(x, \delta(Y)))} d\mu(x),  
\end{align}
where $E^*$ is the set of all points of global $\gamma$-density with respect to $E$, and $\widehat{y} \in \pom$ such that $\delta(Y)=|Y-\widehat{y}|$. 
\end{lemma}
%%%%%%%%%%%%%%%%%%%%%%%%% LEMMA LEMMA LEMMA %%%%%%%%%%%%%%%%%%%%%%%

%%%%%%%%%%%%%%%%%%%%%%%%%% PROOF PROOF PROOF %%%%%%%%%%%%%%%%%%%%%%%
\begin{proof}
Let us first show \eqref{eq:FRF}. Write $\Delta_Y^{\alpha} :=B(Y, (1+\alpha)\delta(Y)) \cap \pom$ for any $Y \in \Omega$. Then $\Delta(\widehat{y}, \alpha \delta(Y)) \subset \Delta_Y^{\alpha} \subset \Delta(\widehat{y}, (2+\alpha) \delta(Y))$, and 
$B(\widehat{y}, \delta(Y)) \subset B(x, (3+\alpha) \delta(Y))$ for all $Y \in \Gamma^{\alpha}(x)$. By the doubling property of $\mu$, \eqref{H4}, we obtain  
\begin{align*}
&\int_{E} \iint_{\Gamma^{\alpha}(x)} \Phi(Y) \frac{dm(Y)}{m(B(x, \delta(Y)))} d\mu(x) 
\\
&=\iint_{\mathscr{R}_{\alpha}(E)} \bigg(\int_{E \cap \Delta_Y^{\alpha}} 
\frac{d\mu(x)}{m(B(x, \delta(Y)))} \bigg)  \Phi(Y) dm(Y) 
\\
&\lesssim \iint_{\mathscr{R}_{\alpha}(E)} \bigg(\int_{\Delta(\widehat{y}, (2+\alpha)\delta(Y))} \frac{d\mu(x)}{m(B(\widehat{y}, \delta(Y)))}  \bigg) \Phi(Y) dm(Y) 
\\
&\lesssim \iint_{\mathscr{R}_{\alpha}(E)} \Phi(Y) 
\frac{\mu(\Delta(\widehat{y}, \delta(Y)))}{m(B(\widehat{y}, \delta(Y)))} dm(Y). 
\end{align*}

To obtain \eqref{eq:RFF}, we see that 
\begin{align*}
&\int_{E} \iint_{\Gamma^{\alpha}(x)} \Phi(Y) \frac{dm(Y)}{m(B(x, \delta(Y)))} d\mu(x) 
\\ 
&=\iint_{\Omega} \bigg(\int_{E \cap \Delta_Y^{\alpha}} \frac{d\mu(x)}{m(B(x, \delta(Y)))}\bigg)  \Phi(Y) dm(Y) 
\\
&\ge \iint_{\mathscr{R}_{\beta}(E^*)} \bigg(\int_{E \cap \Delta_Y^{\alpha}} \frac{d\mu(x)}{m(B(x, \delta(Y)))} \bigg) \Phi(Y) dm(Y). 
\end{align*}
It suffices to prove that there exists $\gamma \in (0, 1)$ sufficiently close to $1$ such that 
\begin{equation}\label{eq:RE}
\int_{E \cap \Delta_Y^{\alpha}} \frac{d\mu(x)}{m(B(x, \delta(Y)))} 
\gtrsim \frac{\mu(\Delta(\widehat{y}, \delta(Y)))}{m(B(\widehat{y}, \delta(Y)))}, 
\qquad\forall \, Y \in \mathscr{R}_{\beta}(E^*).
\end{equation}
Now let $Y \in \mathscr{R}_{\beta}(E^*)$. Then there exists $z \in E^*$ so that $|Y-z|<(1+\beta) \delta(Y)$, and hence, 
\begin{equation}\label{eq:uBuB}
\mu(\Delta(z, (1+\beta) \delta(Y)) \setminus \Delta_Y^{\alpha}) \le c \, \mu(\Delta(z, (1+\beta) \delta(Y))), 
\end{equation} 
where $c=c(\alpha, \beta) \in (0, 1)$. Note that 
\begin{equation}\label{eq:del}
\Delta(\widehat{y}, \delta(Y)) \subset \Delta(z, (3+\beta)\delta(Y)), \qquad\forall x \in \Delta_Y^{\alpha}. 
\end{equation}
By the global $\gamma$-density property, 
\begin{align}\label{eq:uFB}
\mu(E \cap \Delta(z, (1+\beta) \delta(Y))) \ge \gamma \mu(\Delta(z, (1+\beta) \delta(Y))). 
\end{align}
Now since 
\begin{align*}
E \cap \Delta(z, (1+\beta) \delta(Y)) \subset \big(E \cap \Delta_Y^{\alpha}\big) 
\cup \big(\Delta(z, (1+\beta) \delta(Y)) \setminus \Delta_Y^{\alpha} \big), 
\end{align*}
we obtain by \eqref{eq:uBuB}--\eqref{eq:uFB} 
\begin{align*}
\mu(E \cap \Delta_Y^{\alpha}) 
& \ge \mu(E \cap \Delta(z, (1+\beta) \delta(Y))\big) - \mu(\Delta(z, (1+\beta) \delta(Y)) \setminus \Delta_Y^{\alpha} \big)
\\
&\ge (\gamma-c) \mu(\Delta(z, (1+\beta) \delta(Y))) 
\gtrsim \mu(\Delta(\widehat{y}, \delta(Y))), 
\end{align*}
where we take $\gamma$ close enough to $1$ so that $\gamma>c$. This shows \eqref{eq:RE}. 
\end{proof}
%%%%%%%%%%%%%%%%%%%%%%%%%% END END END PROOF %%%%%%%%%%%%%%%%%%%%%%%

For the following result we need to introduce some notation:
\begin{equation*}
\A_{\alpha} F(x):=\bigg(\iint_{\Gamma_{\alpha}(x)} |F(Y)|^2 
\frac{dm(Y)}{m(B(x, \delta(Y)))} \bigg)^{\frac12},
\qquad x \in \pom. 	
\end{equation*}

We are going to extend \cite[Proposition~4]{CMS} to our setting. 

%%%%%%%%%%%%%%%%%%%%%%%%% LEMMA LEMMA LEMMA %%%%%%%%%%%%%%%%%%%%%%%
\begin{lemma}\label{lem:AANN}
Let $0<p<\infty$ and $0<\alpha \le \beta < \infty$. Then the following hold: 
\begin{align}
\label{eq:AFAF} & \|\A_{\beta} F\|_{L^p(\pom, \mu)} 
\le C_{\alpha, \beta, p}  \|\A_{\alpha} F\|_{L^p(\pom, \mu)}, 
\\
\label{eq:NFNF}& \|\N_{\beta} F\|_{L^p(\pom, \mu)} 
\le C_{\alpha, \beta, p} \|\N_{\alpha} F\|_{L^p(\pom, \mu)}. 
\end{align}
\end{lemma}
%%%%%%%%%%%%%%%%%%%%%%%%% LEMMA LEMMA LEMMA %%%%%%%%%%%%%%%%%%%%%%%

%%%%%%%%%%%%%%%%%%%%%%%%%% PROOF PROOF PROOF %%%%%%%%%%%%%%%%%%%%%%%
\begin{proof}
The proof is similar in spirit to that of \cite[Proposition 4]{CMS}. 
We begin with showing \eqref{eq:AFAF}. We may assume that $\|\A_{\alpha}F\|_{L^p(\pom, \mu)}<\infty$ and show the case $0<p<2$. Fix $\lambda>0$, let $E:=\{x \in \pom: \A_{\alpha}F(x) \le \lambda\}$ and $\Omega:=E^c$. We let $E^* \subset E$ be the set of all points of global $\gamma$-density of $E$, and write $\Omega^*=(E^*)^c$. Then, writing $\Omega^*=\{x \in \pom: M_{\mu} \mathbf{1}_{\Omega}(x)>1-\gamma\}$, 
where 
\[
M_{\mu} f(x) := \sup_{r>0} \fint_{B(x, r)} |f| \, d\mu, \quad x \in \pom. 
\]
we have by the doubling property of $\mu$, 
\begin{align}\label{eq:OO}
\mu(\Omega^*) \lesssim \mu(\Omega). 
\end{align}
Lemma \ref{lem:FRF} applied to $\Phi(Y)=|F(Y)|^2$ implies 
\begin{align*}
\int_{E^*} (\A_{\beta}F(x))^2 \, d\mu(x) 
\lesssim \int_{E} (\A_{\alpha}F(x))^2 \, d\mu(x),  
\end{align*}
which together with \eqref{eq:OO} further gives 
\begin{align*}
&\mu(\{x \in \pom: \A_{\beta}F(x)>\lambda\}) 
\lesssim \mu(\Omega^*) + \frac{1}{\lambda^2} \int_E \A_{\alpha}F(x)^2 \, d\mu(x)
\\
&\lesssim \mu(\{x \in \pom: \A_{\alpha}F(x)>\lambda\}) + \frac{1}{\lambda^2} \int_E \A_{\alpha}F(x)^2 \, d\mu(x). 
\end{align*}
Using this estimate and $0<p<2$, we conclude 
\begin{align*}
\|\A_{\beta}F\|_{L^p(\pom, \mu)}^p 
&=p\int_{0}^{\infty} \lambda^{p-1} \mu(\{x \in \pom: \A_{\beta}F(x)>\lambda\}) d\lambda
\\ 
&\lesssim p \int_{0}^{\infty} \lambda^{p-1} \mu(\{x \in \pom: \A_{\alpha}F(x)>\lambda\}) d\lambda 
\\
&\qquad + \int_{0}^{\infty} \lambda^{p-3} \int_{\{\A_{\alpha} F(x)<\lambda\}} \A_{\alpha}F(x)^2 \, d\mu(x) d\lambda 
\\
& \lesssim \|\A_{\alpha}F\|_{L^p(\pom, \mu)}^p.  
\end{align*}

Let us consider the case $p=2$. Write $\Delta_Y^{\alpha} :=B(Y, (1+\alpha)\delta(Y)) \cap \pom$. Note that for any $Y \in \Omega$,  
\begin{equation*}
B(x_0, \delta(Y)) \subset B(x, (3+\alpha+\beta)\delta(Y)), \quad \forall \, x_0 \in \Delta_Y^{\alpha}, \, x \in \Delta_Y^{\beta}, 
\end{equation*}
which gives that for any $x_0 \in \Delta_Y^{\alpha}$, 
\begin{align*}
\int_{\Delta_Y^{\beta}} \frac{d\mu(x)}{m(B(x, \delta(Y)))}  
\lesssim \frac{\mu(\Delta_Y^{\beta})}{m(B(x_0, \delta(Y)))}
\simeq \frac{\mu(\Delta_Y^{\alpha})}{m(B(x_0, \delta(Y)))} 
\lesssim \int_{\Delta_Y^{\alpha}} \frac{d\mu(x)}{m(B(x, \delta(Y)))}. 
\end{align*}
This leads to 
\begin{align*}
\|\A_{\beta} F\|_{L^2(\mu)}^2 &= \iint_{\Omega} \bigg( 
\int_{\Delta_Y^{\beta}} \frac{d\mu(x)}{m(B(x, \delta(Y)))} \bigg) |F(Y)|^2 dm(Y)
\\
&\lesssim \int_{\Omega}  
\bigg(\int_{\Delta_Y^{ \alpha}} \frac{d\mu(x)}{m(B(x, \delta(Y)))} \bigg) |F(Y)|^2 dm(Y) 
= \|\A_{\alpha} F\|_{L^2(\mu)}^2. 
\end{align*}

For the case $2<p<\infty$, we pick a nonnegative function $g \in L^{(p/2)'}(\mu)$ with $\|g\|_{L^{(p/2)'}(\mu)} \le 1$. Observe that for all $x, x_0 \in \Delta_Y^{\beta}$, 
\begin{align*}
\Delta_Y^{\beta} \subset \Delta(x_0, 2(1+\beta) \delta(Y)) \quad\text{and}\quad 
B(x_0, \delta(Y)) \subset B(x, (3+2\beta)\delta(Y)). 
\end{align*}
By this and the doubling property of $m$, we have for any $x_0 \in \Delta_Y^{\alpha}$, 
\begin{align*}
&\int_{\Delta_Y^{\beta}} \frac{g(x)}{m(B(x, \delta(Y)))} d\mu(x) 
\lesssim \int_{\Delta(x_0, (3+2\beta) \delta(Y)))} \frac{g(x)}{m(B(x_0, \delta(Y)))} d\mu(x)
\nonumber \\
&\lesssim \frac{\mu(\Delta(x_0, \delta(Y)))}{m(B(x_0, \delta(Y)))} \fint_{\Delta(x_0, (3+2\beta) \delta(Y)))}g \, d\mu
\le \frac{\mu(\Delta(x_0, \delta(Y)))}{m(B(x_0, \delta(Y)))} M_{\mu} g(x_0). 
\end{align*}
Noting that $\Delta(x, \delta(Y)) \subset \Delta_Y^{1+\alpha}$ for all $x \in \Delta_Y^{\alpha}$, we obtain 
\begin{align*}
\int_{\Delta_Y^{\beta}} \frac{g(x) \, d\mu(x)}{m(B(x, \delta(Y)))} 
\lesssim \fint_{\Delta_Y^{\alpha}} \frac{\mu(\Delta(x, \delta(Y)))}{m(B(x, \delta(Y)))} M_{\mu} g(x) \, d\mu(x)   
\lesssim \int_{\Delta_Y^{\alpha}} \frac{M_{\mu} g(x) \, d\mu(x)}{m(B(x, \delta(Y)))}.  
\end{align*}
Along with Fubini theorem, this implies 
\begin{align*}
\int_{\pom} \A_{\beta} F(x)^2 g(x) \, d\mu(x) 
&=\iint_{\Omega} \bigg(\int_{\Delta_Y^{\beta}} \frac{g(x)}{m(B(x, \delta(Y)))} \bigg) |F(Y)|^2 dm(Y)
\\
&\lesssim \iint_{\Omega} \bigg(\int_{\Delta_Y^{\alpha}} \frac{M_{\mu} g(x)}{m(B(x, \delta(Y)))} 
\, d\mu(x)\bigg) |F(Y)|^2 dm(Y)
\\
&= \int_{\pom} \A_{\alpha} F(x)^2 M_{\mu} g(x) \, d\mu(x) 
\\ 
&\le \|(\A_{\alpha} F)^2\|_{L^{p/2}(\mu)} \|M_{\mu} g\|_{L^{(p/2)'}(\mu)} 
\lesssim \|\A_{\alpha} F\|_{L^p(\mu)}^2, 
\end{align*}
where we used that $\|M_{\mu}\|_{L^q(\mu) \to L^q(\mu)} \le C_q$ for any $1<q<\infty$. This and the duality immediately gives \eqref{eq:AFAF} for the case $2<p<\infty$. 

Next, we turn our attention to the proof of \eqref{eq:NFNF}. It suffices to consider the case $\alpha=1$. For any $\lambda>0$, we set 
\begin{align*}
E_{\lambda} := \{x \in \pom: \N^1 F(x)>\lambda\} \quad\text{and}\quad
E_{\lambda, \beta} := \{x \in \pom: M_{\mu}(\mathbf{1}_{E_{\lambda}})(x)>1/(2C_{\beta})\}, 
\end{align*}
where $C_{\beta}$ is given by that $\mu(\Delta(\widehat{y}, (4+3\beta) \delta(Y))) \le C_{\beta} \mu(\Delta(\widehat{y}, \delta(Y)))$, and $\widehat{y} \in \pom$ such that $\delta(Y)=|Y-\widehat{y}|$. Since $M_{\mu}$ is bounded from $L^r(\mu)$ to $L^{r,\infty}(\mu)$ for any $1\le r<\infty$, we get 
\begin{align}\label{eq:muEE}
\mu(E_{\lambda, \beta}) \lesssim \mu(E_{\lambda}). 
\end{align}
On the other hand, there holds 
\begin{align}\label{eq:NFEb}
\{x \in \pom: \N_{\beta}F(x)>\lambda\} \subset E_{\lambda, \beta}. 
\end{align}
Indeed, let $x \not\in E_{\lambda, \beta}$ and pick $Y \in \Gamma_{\beta}(x)$. Then $\Delta_Y^1 \not\subset E_{\lambda}$, otherwise, noting that 
\[
\Delta(\widehat{y}, \delta(Y)) \subset \Delta_Y^1 \subset \Delta_Y^{\beta} 
\subset \Delta(x, 2(1+\beta)\delta(Y)) \subset \Delta(\widehat{y}, (4+3\beta)\delta(Y)), 
\]
we have 
\begin{align*}
M_{\mu}(\mathbf{1}_{E_{\lambda}})(x) 
& \ge \frac{\mu(E_{\lambda} \cap \Delta(x, 2(1+\beta)\delta(Y)))}{\mu(\Delta(x, 2(1+\beta)\delta(Y)))} 
\\
&\ge \frac{\mu(\Delta_Y^1)}{\mu(\Delta(x, 2(1+\beta)\delta(Y)))} 
\ge \frac{\mu(\Delta(\widehat{y}, \delta(Y)))}{\mu(\Delta(\widehat{y}, (4+3\beta)\delta(Y)))} 
\ge \frac{1}{C_{\beta}}, 
\end{align*} 
which gives $x \in E_{\lambda, \beta}$ and is a contradiction. Thus, there exists $z \in B(Y, 2\delta(Y))$ such that $\N^1F(z) \le \lambda$. The latter further implies $F(Y) \le \N^1 F(z) \le \lambda$, so $\N_{\beta}F(x) \le \lambda$. This shows \eqref{eq:NFEb}. Using \eqref{eq:muEE} and \eqref{eq:NFEb}, we conclude that 
\begin{align*}
&\int_{\pom} (\N^{\beta}F)^p \, d\mu 
= p \int_{0}^{\infty} \lambda^{p-1} \mu(\{x \in \pom: \N^{\beta} F(x)>\lambda \}) \, d\lambda
\\
&\le p \int_{0}^{\infty} \lambda^{p-1} \mu(E_{\lambda, \beta}) \, d\lambda 
\lesssim  p \int_{0}^{\infty} \lambda^{p-1} \mu(E_{\lambda}) \, d\lambda 
= \int_{\pom} (\N^1 F)^p \, d\mu. 
\end{align*}
The proof is complete. 
\end{proof}
%%%%%%%%%%%%%%%%%%%%%%%%%% END END END PROOF %%%%%%%%%%%%%%%%%%%%%%%

The result below asserts that for the same aperture $\alpha$, the finiteness of $\S^{\alpha}_{r_1}u$ is equivalent to that of $\S^{\alpha}_{r_2}u$ for all $r_1, r_2 \in (0, \infty)$.

%%%%%%%%%%%%%%%%%%%%%%%%% LEMMA LEMMA LEMMA %%%%%%%%%%%%%%%%%%%%%%%
\begin{lemma}\label{lem:two-trunc}
Let $Lu=-\div(A\nabla u)$ be a real degenerate elliptic operator with the coefficients matrix $A$ satisfying \eqref{eq:elliptic}, and let $u$ be a bounded weak solution of $Lu=0$ in $\Omega$. Then for every $\alpha>0$ and $0<r_1<r_2<\infty$, 
\begin{equation*}
\sup_{x\in F} \big[\S^{\alpha}_{r_2}u(x)^2 - \S^{\alpha}_{r_1}u(x)^2 \big] <\infty.  
\end{equation*}
whenever $F \subset \pom$ is a bounded Borel set. 
\end{lemma}

%%%%%%%%%%%%%%%%%%%%%%%%% PROOF PROOF PROOF %%%%%%%%%%%%%%%%%%%%%%%%
\begin{proof} 
The proof follows the same scheme as that of \cite[Remark 3.1]{CMO}. Fix $\alpha>0$ and $0<r_1<r_2<\infty$. Pick $x_0 \in \pom$ and $r_0>\max\{r_1, r_2\}$ so that $F \subset \Delta(x_0, r_0)$. Let $x \in F$. By definition, for any $Y \in \Gamma^{\alpha}_{r_2}(x) \backslash \Gamma^{\alpha}_{r_1}(x)$, we have 
\[
r_1 \le |Y-x|<(1+\alpha) \delta(Y), \qquad \delta(Y) \le |Y-x|<r_2, 
\]
which implies that 
\[
\Gamma^{\alpha}_{r_2}(x) \backslash \Gamma^{\alpha}_{r_1}(x) 
\subset \{Y \in \Omega: r_1/(1+\alpha) \le \delta(Y) \le r_2\} =: K, 
\]
and 
\[
B(x_0, r_0) \subset B(x, 2r_0 (1+\alpha)r_1^{-1} \delta(Y)), \qquad 
B(x, \delta(Y)) \subset B(x_0, 2r_0). 
\]
Since $u \in W_r(\Omega)$, it follows from \eqref{eq:WrE} that 
\begin{align*}
\sup_{x\in F} \big[\S^{\alpha}_{r_2}u(x)^2 - \S^{\alpha}_{r_1}u(x)^2\big] 
&=\sup_{x \in F} \iint_{\Gamma^{\alpha}_{r_2}(x) \setminus \Gamma^{\alpha}_{r_1}(x)} 
|\nabla u(Y)|^2 \delta(Y)^2 \frac{dm(Y)}{m(B(x, \delta(Y)))} 
\\
&\leq \sup_{x \in F} \frac{C_{\alpha, r_0, r_1} r_2^2}{m(B(x_0, r_0))} \iint_{K}  |\nabla u|^2 dm 
<\infty. 
\end{align*}
This yields the desired estimate. 
\end{proof}
%%%%%%%%%%%%%%%%%%%%%%%%%% END END END PROOF %%%%%%%%%%%%%%%%%%%%%%%

\subsection{Non-tangential maximal functions estimates} 
In this section, let us record some estimates for the non-tangential maximal functions. The first one say that the local dyadic non-tangential maximal function is pointwise controlled by the local Hardy-Littlewood maximal function.

%%%%%%%%%%%%%%%%%%%%%%%%% LEMMA LEMMA LEMMA %%%%%%%%%%%%%%%%%%%%%%%
\begin{lemma}\label{lem:NQuf}
Given $Q_0 \in \D$ and $f \in \mathscr{C}(\pom)$ with $\supp f \subset 2 \widetilde{\Delta}_{Q_0}$, let 
\[
u(X)=\int_{\pom} f\, d\w_L^X, \quad X \in \Omega. 
\]
Then we have 
\begin{equation*}
\N_{Q_0} u(x) \lesssim \sup_{\substack{\Delta \ni x \\ 0 < r_{\Delta} < 4A_0 \ell(Q_0)}} 
\fint_{\Delta} |f| \, d \w_L^{X_{Q_0}}, \quad x \in Q_0. 
\end{equation*}
\end{lemma}
%%%%%%%%%%%%%%%%%%%%%%%%% LEMMA LEMMA LEMMA %%%%%%%%%%%%%%%%%%%%%%%

%%%%%%%%%%%%%%%%%%%%%%%%% PROOF PROOF PROOF %%%%%%%%%%%%%%%%%%%%%%%%
\begin{proof} 
The proof is essentially contained in that of \cite[Proposition 2.57]{AHMT}. Indeed,  one can see that to prove \cite[Proposition 2.57]{AHMT}, it just needs Harnack's inequality, Bourgain's estimate, ``change of pole" formulas, and the doubling property of $\w_L$. In our current scenario, all of these estimates have been presented in Lemmas \ref{lem:PDE} and \ref{lem:wL-G}. This provides what we need to conclude the proof. We omit the detailed proof.  
\end{proof}
%%%%%%%%%%%%%%%%%%%%%%%%%% END END END PROOF %%%%%%%%%%%%%%%%%%%%%%%

The following inequality allows us to dominate the truncated non-tangential maximal function by its dyadic version, which along with Lemma \ref{lem:NQuf} will be used to show $\w_L \in A_{\infty}(\mu)$ implies the $L^p$-solvability of \eqref{eq:D} for $L$ (cf. Section \ref{sec:Lp}).  

%%%%%%%%%%%%%%%%%%%%%%%%% LEMMA LEMMA LEMMA %%%%%%%%%%%%%%%%%%%%%%%
\begin{lemma}\label{lem:NrNQ}
For any $\alpha>0$ and any surface ball $\Delta_0=\Delta(x_0, r_0)$ with $x_0 \in \pom$ and $0<r_0<\diam(\pom)$, 
\begin{equation*}
\N_{r_0}^{\alpha} u (x) 
\lesssim_\alpha \sup_{0<r \leq c_\alpha r_0} \fint_{\Delta(x,r)}  
\sup_{Q \in \F_{\Delta_0}} \N_Q u \, d\mu, \qquad x \in \Delta_0.
\end{equation*} 
where $\F_{\Delta_0} := \{Q \in \D: Q \cap 4A_0 \Delta_0 \neq \varnothing, \, 5A_0 r_0 < \ell(Q) \le 10 A_0 r_0\}$ with $A_0$  given in \eqref{deltaQ}. 
\end{lemma}
%%%%%%%%%%%%%%%%%%%%%%%%% LEMMA LEMMA LEMMA %%%%%%%%%%%%%%%%%%%%%%%

%%%%%%%%%%%%%%%%%%%%%%%%% PROOF PROOF PROOF %%%%%%%%%%%%%%%%%%%%%%%%
\begin{proof} 
Fix $x \in \Delta_0$ and $Y \in \Gamma_{r_0}^{\alpha}(x)$. Let $I \in \mathcal{W}$ be such that $Y \in I$. Take $z \in \pom$  such that $\dist(I, \pom) = \dist(I, z)$ and let $Q_Y \in \D$ be the unique dyadic cube satisfying $\ell(Q_Y) = \ell(I)$ and $z \in Q_Y$, which implies that $I \subset U_{Q_Y}$. We claim that 
\begin{align}\label{eq:QYQ}
Q_Y \subset 4A_0 \Delta_0 \cap \Delta(x, c_{\alpha} \delta(Y)). 
\end{align}
Indeed, by the properties of the Whitney cubes, we have 
\[
|z-Y| \le \dist(z, I) + \diam(I) \le \textstyle{\frac54} \dist(I, \pom), 
\]
and
\begin{align*}
4\ell(Q_Y) = 4\ell(I) \le \dist(I, \pom) \le \delta(Y) \le |Y-x|<r_0. 
\end{align*} 
In view of \eqref{deltaQ}, for every $z' \in Q_Y$, 
\begin{align*}
|z'-x_0| & \le |z'-z| + |z-Y| + |Y-x| + |x-x_0| 
\\
&< 2A_0 \ell(Q_Y)+ \textstyle{\frac54} r_0 + r_0 +r_0 
\le \textstyle{\frac12} (A_0+7) r_0 
\le 4A_0 r_0, 
\end{align*}
and 
\begin{align*}
|z'-x| \le |z'-z| + |z-Y| + |Y-x| \le 2A_0 \ell(Q_Y) + \textstyle{\frac54} \delta(Y) + (1+\alpha) \delta(Y) < c_{\alpha} \delta(Y),
\end{align*}
This shows \eqref{eq:QYQ}.  

By \eqref{eq:QYQ} and definition of $\F_{\Delta_0}$, there exists a unique $\widetilde{Q}_Y \in \F_{\Delta_0}$ such that $Q_Y \subsetneq \widetilde{Q}_Y$. In particular,  $Y \in I \subset U_{Q_Y}\subset \Gamma_{\widetilde{Q}_Y}(z)$ for all $z \in Q_Y$ and
\begin{equation*}
|u(Y)| \leq \N_{\widetilde{Q}_Y} u (z), \qquad \forall z \in Q_Y.
\end{equation*}
Taking the average over $Q_Y$ with respect to $\mu$, we obtain 
\begin{align*}
|u(Y)| &\leq \fint_{Q_Y} \N_{\widetilde{Q}_Y} u \, d\mu
\le \fint_{Q_Y} \sup_{Q \in \mathcal{F}_{\Delta_0}} \N_Q u \, d\mu
\\
&\lesssim \fint_{\Delta(x, c_\alpha t)} \sup_{Q \in \F_{\Delta_0}} \N_Q u \, d\mu 
\le \sup_{0<r \leq c_\alpha r_0}\fint_{ \Delta(x,r)}  \sup_{Q \in \F_{\Delta_0}} \N_Q u \, d\mu. 
\end{align*}
The immediately gives the desired result. 
\end{proof}
%%%%%%%%%%%%%%%%%%%%%%%%% PROOF RPOOF PROOF %%%%%%%%%%%%%%%%%%%%%%%%

We establish the relationship between the boundary data and the related non-tangential maximal function as follows. 

%%%%%%%%%%%%%%%%%%%%%%%%% LEMMA LEMMA LEMMA %%%%%%%%%%%%%%%%%%%%%%%
\begin{lemma}\label{lem:fNu}
For every surface ball $\Delta=\Delta(x, r)$ with $x \in \pom$ and $0<r<\diam(\pom)$, and for every $ f \in \mathscr{C}(\pom)$ with $\supp f \subset \Delta$, 
\begin{equation*}
\bigg|\int_{\Delta} f \, d\w_L^{X_{\Delta}} \bigg|  
\lesssim_{\alpha} \mu(\Delta)^{-\frac1p} \|\N^{\alpha}_r u\|_{L^p(\Delta, \mu)}, 
\end{equation*}
where $u$ is defined in \eqref{eq:u-sol}.
\end{lemma}
%%%%%%%%%%%%%%%%%%%%%%%%% LEMMA LEMMA LEMMA %%%%%%%%%%%%%%%%%%%%%%%

%%%%%%%%%%%%%%%%%%%%%%%%% PROOF RPOOF PROOF% %%%%%%%%%%%%%%%%%%%%%%%
\begin{proof}
Let $u$ be a solution of $Lu=0$ with boundary data $|f|$, that is, 
\[
u(X) := \int_{\pom} |f| \, d\w_L^X, \qquad X \in \Omega. 
\] 
Write $X_0:=X_{\Delta}$ and $X_1:=X_{(2+\alpha)^{-1}\Delta}$. It is easy to see that $|X_0-X_1| \lesssim r \simeq \delta(X_0) \simeq \delta(X_1)$. Thus, $u(X_0) \simeq_\alpha u(X_1)$ by Harnack's inequality.   If we set $\widehat{x}_1 \in \pom$ so that  $\delta(X_1)=|X_1-\widehat{x}_1|$, then 
\begin{align}\label{eq:DQQ}
\Delta(\widehat{x}_1, \alpha \delta(X_1)) \subset \Delta \subset \Delta(\widehat{x}_1, 2r). 
\end{align}
Indeed, for any $y \in \Delta(\widehat{x}_1, \alpha \delta(X_1))$, 
\begin{align*}
|y-x| &\le |y-\widehat{x}_1| + |\widehat{x}_1-X_1| + |X_1-x| 
\\
&< \alpha \delta(X_1) + \delta(X_1) + |X_1-x| 
\le (2+\alpha) |X_1-x| < r, 
\end{align*}
and for any $z \in \Delta$, 
\begin{align*}
|z-\widehat{x}_1| & \le |z-x| + |x-X_1| + |x_1-\widehat{x}_1|  
\\
&< r + |x-X_1| + \delta(X_1) \le r + 2|x-X_1| 
\le \Big(1 + \frac{2}{2+\alpha}\Big)r < 2r.  
\end{align*}
On the other hand, there holds 
\begin{align*}
X_1 \in \Gamma_r^{\alpha}(z), \qquad \forall z \in \Delta(\widehat{x}_1, \alpha \delta(X_1)), 
\end{align*}
since 
\[
|X_1-z| \le |X_1-\widehat{x}_1| + |\widehat{x}_1-z| < (1+\alpha) \delta(X_1) \le \frac{1+\alpha}{2+\alpha} r < r. 
\]
Thus, we get 
\begin{equation}\label{eq:uuu}
\N^{\alpha}_r u(z) \ge u(X_1) \simeq u(X_0), \qquad\forall z\in \Delta(\widehat{x}_1, \alpha \delta(X_1)).  
\end{equation}
Using \eqref{eq:DQQ}, \eqref{eq:uuu}, and the doubling property of $\mu$, we conclude that 
\begin{align*}
\mu(\Delta)^{\frac1p} \bigg|\int_{\Delta} f\,d\w_L^{X_{\Delta}}\bigg|
\le u(X_0) \mu(\Delta)^{\frac1p} & \simeq u(X_1) \mu(\Delta(\widehat{x}_1, \alpha \delta(X_1)))^{\frac1p}
\le \|\N^{\alpha}_r u\|_{L^p(\Delta, \mu)}.
\end{align*}
This readily gives the desired estimate.  
\end{proof}
%%%%%%%%%%%%%%%%%%%%%%%%% END END END PROOF %%%%%%%%%%%%%%%%%%%%%%%

\subsection{Localization of Carleson and elliptic measures} 
Let $\{\alpha_Q\}_{Q \in \D}$ be a sequence of non-negative numbers indexed on the dyadic cubes. For any collection $\D' \subset \D$, we define
\begin{equation}\label{eq:ma}
\m_{\alpha}(\D') := \sum_{Q \in \D'} \alpha_Q. 
\end{equation} 
Given $Q_0 \in \D$, we say that $\m_{\alpha}$ is a discrete \lq\lq Carleson measure\rq\rq \ on $\D_{Q_0}$ (with respect to $\mu$), and we write $\m \in \mathcal{C}(Q_0)$, if  
\begin{equation*}
\|\m_{\alpha}\|_{\mathcal{C}(Q_0)} :=\sup _{Q \in \D_{Q_0}} \frac{\m_{\alpha}(\D_Q)}{\mu(Q)} < \infty. 
\end{equation*} 
We also write 
\begin{equation*}
\|\m_{\alpha}\|_{\mathcal{C}(\pom)} :=\sup _{Q \in \D} \frac{\m_{\alpha}(\D_Q)}{\mu(Q)} < \infty 
\end{equation*} 
to denote the \lq\lq global\rq\rq \ Carleson norm on $\D$. 

Since $\mu$ is doubling, modifying the proof of \cite[Lemmas 3.5]{HMT}, we obtain stopping time cubes as follows.  
%%%%%%%%%%%%%%%%%%%%%%%% LEMMA LEMMA LEMMA %%%%%%%%%%%%%%%%%%%%%%%%
\begin{lemma}\label{lem:FQj}
Let $Q_0 \in \D$ and let $\w$ be a non-negative regular Borel measure on $Q_0$. Assume that $\w \ll \mu$ on $Q_0$ and write $k=d\w/d\mu$. Assume also that there exist $K_0 \geq 1$, 
$\theta>0$ such that 
\begin{align*}
1 \leq \frac{\w(Q_0)}{\mu(Q_0)} \leq K_0 \quad \text{and} \quad 
\frac{\w(F)}{\mu(Q_0)} \leq K_0 \bigg(\frac{\mu(F)}{\mu(Q_0)} \bigg)^{\theta}, \quad \forall F \subset Q_0. 
\end{align*} 
Then there exists a pairwise disjoint family $\F=\{Q_j\}_j \subset \D_{Q_0} \setminus\{Q_0\}$ such that 
\begin{align*}
\mu\bigg(Q_0 \setminus \bigcup_{Q_j \in \F}Q_j \bigg) \geq K_1^{-1} \mu(Q_0) 
\quad \text{and} \quad \frac12 \leq \frac{\w(Q)}{\mu(Q)} \leq K_0 K_1, \quad \forall Q \in \D_{\F, Q_0}.  
\end{align*}
\end{lemma}
%%%%%%%%%%%%%%%%%%%%%%%% LEMMA LEMMA LEMMA %%%%%%%%%%%%%%%%%%%%%%%%

As argued in \cite[Lemma 3.12]{HMT}, invoking Lemma \ref{lem:FQj}, one can obtain the next result. It states that to show $\m_{\alpha}$ is a Carleson measure, it suffices to consider local sawtooths with an ample contact.

%%%%%%%%%%%%%%%%%%%%%%%% LEMMA LEMMA LEMMA %%%%%%%%%%%%%%%%%%%%%%%%
\begin{lemma}\label{lem:disCarleson}
Let $Q^0$ be either $\partial \Omega$ or a fixed cube in $\D$. Let $\alpha=\{\alpha_Q\}_{Q \in \D_{Q^0}}$ 
be a sequence of non-negative numbers and consider $\m_{\alpha}$ as defined above. Given $M_1>0$ and 
$K_1 \geq 1$, we assume that for every $Q_0 \in \D_{Q^0}$ there exists a pairwise disjoint family 
$\F_{Q_0} = \{Q_j\}_j \subset \D_{Q_0} \setminus \{Q_0\}$ such that 
\begin{align*}
\mu\bigg(Q_0 \setminus \bigcup_{Q_j \in \F_{Q_0}}Q_j \bigg) \geq K_1^{-1} \mu(Q_0) 
\quad \text{and} \quad \m_{\alpha}(\D_{\F_{Q_0}, Q_0}) \leq M_1 \mu(Q_0).   
\end{align*}
Then, $\m_{\alpha}$ is a Carleson measure in $Q^0$ and $\|\m_{\alpha}\|_{\mathcal{C}(Q^0)} \leq K_1 M_1$.  
\end{lemma}
%%%%%%%%%%%%%%%%%%%%%%%% LEMMA LEMMA LEMMA %%%%%%%%%%%%%%%%%%%%%%%%

The following result enables us to reduce the $A_{\infty}$ property to the dyadic setting. It will be quite useful and convenient in practice (cf. Section \ref{sec:CMEAi}).

%%%%%%%%%%%%%%%%%%%%%%%% LEMMA LEMMA LEMMA %%%%%%%%%%%%%%%%%%%%%%%%
\begin{lemma}\label{lem:Ai-Ai}
Assume that for any $\alpha \in (0, 1)$, there exists $\beta \in (0,1)$ such that for every $Q_0 \in \D$ and every Borel set $F \subset Q_0$,  
\begin{equation}\label{eq:FQ}
\mu(F) \geq \alpha \mu(Q_0) \quad \Longrightarrow \quad 
\w_L^{X_{Q_0}}(F) \geq \beta \w_L^{X_{Q_0}}(Q_0). 
\end{equation}  
Then $\w_L \in A_{\infty}(\mu)$. 
\end{lemma}
%%%%%%%%%%%%%%%%%%%%%%%% LEMMA LEMMA LEMMA %%%%%%%%%%%%%%%%%%%%%%%%

%%%%%%%%%%%%%%%%%%%%%%%%% PROOF PROOF PROOF %%%%%%%%%%%%%%%%%%%%%%%%
\begin{proof}
We follow the argument in \cite[p. 7916]{CHMT}. We first show that \eqref{eq:FQ} implies a dyadic-$A_{\infty}$ condition, that is, for any $\alpha \in (0, 1)$, there exists $\beta=\beta(\alpha) \in (0,1)$ such that for every $Q^0 \in \D$, for every $Q_0 \in \D_{Q^0}$,  and for every Borel set $F \subset Q_0$, 
\begin{equation}\label{eq:dyadic-Ainfty}
\mu(F) \geq \alpha \mu(Q_0) \quad \Longrightarrow \quad 
\w_L^{X_{Q^0}}(F) \geq \beta \w_L^{X_{Q^0}}(Q_0).  
\end{equation}  
Indeed, let $\alpha \in (0, 1)$ so that $\mu(F) \ge \alpha \mu(Q_0)$. Then it follows from \eqref{eq:FQ} that 
\begin{equation*}
\w_L^{X_{Q_0}}(F) \ge \beta \w_L^{X_{Q_0}}(Q_0).  
\end{equation*} 
Due to Lemma \ref{lem:wL-G} part \eqref{list-4} and Harnack's inequality, the estimate above leads 
\begin{equation*}
\frac{\w_L^{X_{Q^0}}(F)}{\w_L^{X_{Q^0}}(Q_0)}
\ge \frac{1}{C} \frac{\w_L^{X_{Q_0}}(F)}{\w_L^{X_{Q_0}}(Q_0)} \geq \frac{\beta}{C}. 
\end{equation*}

Next, we need to see that \eqref{eq:dyadic-Ainfty} gives $\w_L \in A_{\infty}(\mu)$. Let $\widetilde{\alpha}\in (0,1)$ and fix a surface ball $\Delta_0=B_0 \cap \pom$, with $B_0=B(x_0,r_0)$, $x_0 \in \pom$, and $0<r_0<\diam(\pom)$. Take an arbitrary surface ball $\Delta=B \cap \pom$ centered at $\pom$ with $B=B(x,r) \subset B_0$, and let $F \subset \Delta$ be a Borel set such that $\mu(F) > \widetilde{\alpha} \mu(\Delta)$. Consider the pairwise disjoint family 
\[
\mathcal{Q}=\Big\{Q \in \D: \, Q \cap \Delta \neq \emptyset, \, \frac{r}{4A_0}<\ell(Q)\le \frac{r}{2A_0} \Big\}, 
\] 
where $A_0$ is the constant in \eqref{deltaQ}. In particular,  
\begin{align}\label{DD}
F \subset \Delta \subset \bigcup_{Q\in \mathcal{Q}} Q \subset 2\Delta.
\end{align} 
Recall that $c_3>1$ is the doubling constant of $\mu$. Denote $C_0 := 2c_3>1$. By the pigeon-hole principle, there exists some $Q_0 \in \mathcal{Q}$ so that 
\begin{align}\label{FQQ}
\mu(F \cap Q_0)/\mu(Q_0) > \widetilde{\alpha}/C_0.
\end{align} 
Indeed, if \eqref{FQQ} is not true, then $\mu(F \cap Q)/\mu(Q) \le \widetilde{\alpha}/C_0$ for all $Q \in \mathcal{Q}$, which along with \eqref{DD} and the disjointness of $\mathcal{Q}$ implies 
\begin{align*}
\mu(F)
= \sum_{Q \in \mathcal{Q}} \mu(F \cap Q)
\le \widetilde{\alpha} \, C_0^{-1} \sum_{Q \in \mathcal{Q}} \mu(Q)
\le \widetilde{\alpha} \, C_0^{-1} \mu(2\Delta)
\le \widetilde{\alpha} \, C_0^{-1} \, c_3 \, \mu(\Delta)
= \widetilde{\alpha} \mu(\Delta)/2, 
\end{align*}
where the condition \eqref{H3} was used in the last inequality. This is a contradiction because $\mu(F) > \widetilde{\alpha} \mu(\Delta)$. 

Let $Q^0 \in \D$ be the unique dyadic cube such that $Q_0 \subset Q^0$ and  $\frac{r_0}{2}<\ell(Q^0)\le r_0$. Then invoking \eqref{eq:dyadic-Ainfty} with ${\widetilde{\alpha}}/{C_0}$ in place of $\alpha$ to find $\beta \in (0,1)$, we arrive at
\begin{align*}
\frac{\w_L^{X_{\Delta_0}}(F)}{\w_L^{X_{\Delta_0}}(\Delta)}
\ge \frac{\w_L^{X_{\Delta_0}}(F \cap Q_0)}{\w_L^{X_{\Delta_0}}(\Delta)}
\approx \frac{\w_L^{X_{\Delta_0}}(F \cap Q_0)}{\w_L^{X_{\Delta_0}}(Q_0)}
\approx \frac{\w_L^{X_{Q^0}}(F \cap Q_0)}{\w_L^{X_{Q^0}}(Q_0)}
>\frac{\beta}{C_0}, 
\end{align*}
where we used Lemma \ref{lem:PDE}, Harnack's inequality, and \eqref{FQQ}. In short, we have obtained that for every $\widetilde{\alpha}\in (0,1)$ there exists $\widetilde{\beta}\in (0,1)$ such that
\begin{align*}
\mu(F) > \widetilde{\alpha} \mu(\Delta) 
\implies \w_L^{X_{\Delta_0}}(F) > \widetilde{\beta} \w_L^{X_{\Delta_0}}(\Delta). 
\end{align*}
This shows that $\w_L \in A_{\infty}(\pom)$.  
\end{proof}
%%%%%%%%%%%%%%%%%%%%%%%%% END END END PROOF %%%%%%%%%%%%%%%%%%%%%%%%

\subsection{From CME to non-tangential maximal functions}  
To establish PDE estimates, we present the cut-off function adapted to sawtooth domains. Such construction originated in \cite[Lemma 4.44]{HMT}, which was extended in \cite{CDMT} to 1-sided NTA domains satisfying CDC. In our context, it follows from Lemma \ref{lem:dyadic} that $\diam(Q) \lesssim \ell(Q)$ holds for every cube $Q \in \D$. Although we are not sure if the lower bound holds, the proof of \cite[Lemma 6.4]{CDMT} can be modified to our current setting. Indeed, properties (i) and (ii) below directly follow from the definition of $\Psi_N$ and some elementary calculations. The most difficult part is to show the bounded overlap of $\{Q_I\}_{I \in \W_N^{\Sigma}}$. In view of the fact that $\diam(Q) \lesssim \ell(Q)$ for every cube $Q \in \D$, this can be done by following the geometric argument in the proof of \cite[Lemma 6.4]{CDMT}. 

%%%%%%%%%%%%%%%%%%%%%%%% LEMMA LEMMA LEMMA %%%%%%%%%%%%%%%%%%%%%%%%
\begin{lemma}\label{lem:Psi}
Let $(\Omega, m, \mu)$ satisfy \eqref{H1}--\eqref{H6}. Given $Q_0 \in \D$, a pairwise disjoint collection $\F \subset \D_{Q_0}$, and $N\ge 4$ sufficiently large, let $\F_N$ be the family of maximal cubes of the collection ${\F}$ augmented by adding all the cubes $Q \in \D_{Q_0}$ such that $\ell(Q) \leq 2^{-N} \ell(Q_0)$. There exist $\Psi_N \in \mathscr{C}_c^{\infty}(\ree)$ and a constant $C \ge 1$ depending only on the allowable parameters but independent of $N$, $\F$, and $Q_0$ such that the following hold:
\begin{enumerate}		
\item[\rm{(i)}] $C^{-1}\,\mathbf{1}_{\Omega_{\F_N,Q_0}} \le \Psi_N \le \mathbf{1}_{\Omega_{\F_N,Q_0}^*}$,  	
\quad\rm{(ii)} $\sup\limits_{X \in \Omega} |\nabla \Psi_N(X)|\,\delta(X)\le C$, 
		
\item[{\rm (iii)}] Setting
\begin{equation*}
\W_N:=\bigcup_{Q\in\D_{\F_{N},Q_0}} \W_Q^*, \quad
\W_N^{\Sigma}:= \big\{I\in \W_{N}:\, \exists\,J\in \W\setminus \W_N\ \mbox{with}\ \partial I\cap\partial J\neq\varnothing \big\},
\end{equation*}
one has
\begin{equation*}
\nabla \Psi_N^{\vartheta} \equiv 0 \quad \mbox{in} \quad
\bigcup_{I\in \W_N \setminus \W_N^{\Sigma} }I^{**},
\end{equation*}
and there exists a family $\{Q_I\}_{I \in \W_N^{\Sigma}}$ so that 
\begin{align*}
C^{-1} \ell(I) \le \ell(Q_I) \le C \ell(I), \quad 
\dist(I, Q_I) \le C \ell(I), 
\quad\text{ and }\quad \sum_{I \in \W_N^{\Sigma}} \mathbf{1}_{Q_I} \le C. 
\end{align*} 
\end{enumerate}
\end{lemma}
%%%%%%%%%%%%%%%%%%%%%%%% LEMMA LEMMA LEMMA %%%%%%%%%%%%%%%%%%%%%%%%

We are going to use the dyadic non-tangential maximal function to dominate Carleson measure estimates. A key ingredient of the  proof is integration by parts, which has been used by many authors (cf. \cite{CDMT, CMO, CHMT, HMT}).

%%%%%%%%%%%%%%%%%%%%%%%%% LEMMA LEMMA LEMMA %%%%%%%%%%%%%%%%%%%%%%%
\begin{lemma}\label{lem:uGNQ}
Given a cube $Q_0 \in \D$, let $\F:=\{Q_j\} \subset \D_{Q_0} \setminus\{Q_0\}$ be a pairwise disjoint family. If we write $F_Q:=Q \setminus \bigcup_{Q_j \in \F} Q_j$, then   
\begin{align}\label{eq:uGNQ}
\frac{1}{\w_L^{X_{Q_0}}(Q_0)} \iint_{\Omega_{\F, Q_0}} |\nabla u(X)|^2 \, G_L(X_{Q_0}, X)\, dm(X)
\lesssim \sup_{x \in F_Q} \N_Q u(x)^2. 
\end{align}
\end{lemma}
%%%%%%%%%%%%%%%%%%%%%%%%% LEMMA LEMMA LEMMA %%%%%%%%%%%%%%%%%%%%%%%

%%%%%%%%%%%%%%%%%%%%%%%%% PROOF RPOOF PROOF %%%%%%%%%%%%%%%%%%%%%%%%
\begin{proof}
We are going to apply the technique of integration by parts as in the proof of \cite[Proposition 3.58]{HMT} and \cite[Proposition 4.36]{CHMT}.  For convenience, we write $\w:=\w_L^{X_{Q_0}}$ and $\G(\cdot):= G_L(X_{Q_0}, \cdot)$. For every $N \geq 1$, we set $\F_N$ to be the family of maximal cubes of the collection $\F$ augmented by adding all the cubes $Q \in \D_{Q_0}$ such that $\ell(Q) \leq 2^{-N} \ell(Q_0)$. This means that $Q \in \D_{\F_N, Q_0}$ if and only if $Q \in \D_{\F, Q_0}$ and $\ell(Q)>2^{-N} \ell(Q_0)$. Observe that $\D_{\F_N, Q_0} \subset \D_{\F_M, Q_0}$ for all $N \leq M$, and hence $\Omega_{\F_N, Q_0} \subset \Omega_{\F_M, Q_0} \subset \Omega_{\F, Q_0}$. This, together with the 
monotone convergence theorem, gives 
\begin{align}\label{eq:lim-FN}
\iint_{\Omega_{\F, Q_0}} |\nabla u|^2 \G \, dm  
=\lim_{N \to \infty} \iint_{\Omega_{\F_N, Q_0}} |\nabla u|^2 \G \, dm.  
\end{align}
We claim that 
\begin{align}\label{eq:int-parts}
\mathcal{J}_N := \iint_{\Omega_{\F_N, Q_0}} |\nabla u|^2 \G \, dm
\lesssim \w(Q_0) \sup_{x \in F_Q} \N_Q u(x)^2, 
\end{align}
where the implicit constant is independent of $N$. Consequently, the estimate \eqref{eq:uGNQ} as desired follows from \eqref{eq:lim-FN} and \eqref{eq:int-parts}. 

To get \eqref{eq:int-parts}, by Lemma \ref{lem:Psi}, \eqref{H4} and the ellipticity of $A$, one has 
\begin{align}\label{eq:JM}
\mathcal{J}_N  \leq \iint_{\Omega} |\nabla u|^2 \G \Psi_N \, dm
\lesssim \iint_{\Omega} A \nabla u \cdot \nabla u \ \G \Psi_N\, dX. 
\end{align}
The Leibniz's rule gives that 
\begin{align}\label{eq:Leibniz}
A \nabla u \cdot \nabla u \  \G \Psi_N 
&=A \nabla u \cdot \nabla (u \G \Psi_N) 
- \frac12 A \nabla (u^2 \Psi_N) \cdot \nabla \G 
\nonumber\\
&\qquad+ \frac12 A \nabla \Psi_N \cdot \nabla \G \ u^2 
-\frac12 A \nabla (u^2) \cdot \nabla\Psi_N \, \G. 
\end{align}
Note that $u \in W_r(\Omega)$, $\G \in W_r(\Omega \setminus \{X_{Q_0}\})$, and that $\overline{\Omega_{\F_N, Q_0}^*}$ is a compact subset of $\Omega$. Hence, $u \in W^{1,2}(\Omega_{\F_N, Q_0}^*)$ and $u \G \Psi_N \in W_0^{1,2}(\Omega_{\F_N, Q_0}^*)$. Together with the fact that $Lu=0$ in the weak sense in $\Omega$, these give
\begin{align}\label{eq:Lu}
\iint_{\Omega} A \nabla u \cdot \nabla(u \G \Psi_N) \, dX
=\iint_{\Omega_{\F_N, Q_0}^*} A \nabla u \cdot \nabla(u \G \Psi_N) \, dX=0. 
\end{align}
Also, Lemma \ref{lem:Green} part (v) implies that $\G \in W^{1,2}(\Omega_{\F_N, Q_0}^*)$ and $L^{\top}\G=0$ in the weak sense in $\Omega \setminus \{X_{Q_0}\}$. Thanks to $u^2 \Psi_N \in W^{1,2}_0(\Omega_{\F_N, Q_0}^*)$, we obtain 
\begin{align}\label{eq:LTG}
\iint_{\Omega} A \nabla (u^2 \Psi_N) \cdot \nabla \G \, dX
=\iint_{\Omega_{\F_N, Q_0}^*} A^{\top} \nabla \G \cdot \nabla(u^2 \Psi_N) \, dX=0. 
\end{align}
Applying Lemma~\ref{lem:Psi}, Lemma \ref{lem:PDE} parts \eqref{list:Cacci} and \eqref{list:Harnack}, and Lemma \ref{lem:wL-G} parts \eqref{list-2} and \eqref{list-3}, we conclude 
\begin{align}\label{eq:part-1}
&\iint_{\Omega} |A \nabla \Psi_N \cdot \nabla \G \ u^2| \, dX 
\nonumber\\
& \lesssim \sum_{I \in \W_N^{\Sigma}} \iint_{I^*} |\nabla \Psi_N| |\nabla \G| u^2\, dm 
\nonumber\\
&\lesssim \sum_{I \in \W_N^{\Sigma}} u(X_I)^2 \, \ell(I)^{-1} m(I) \bariint_{I^*} |\nabla \G| \, dm
\nonumber\\
&\le \sum_{I \in \W_N^{\Sigma}} u(X_I)^2 \, \ell(I)^{-1} m(I) \bigg(\bariint_{I^*} |\nabla \G|^2 \, dm  \bigg)^{\frac12}
\nonumber\\
&\lesssim \sum_{I \in \W_N^{\Sigma}} u(X_I)^2 \, \ell(I)^{-2} m(I) \bigg(\bariint_{I^{**}} \G^2 \, dm \bigg)^{\frac12}
\nonumber\\
&\lesssim \sum_{I \in \W_N^{\Sigma}} u(X_I)^2 \,  \ell(I)^{-2} m(I) \G(X_I) 
\nonumber\\
&\lesssim \sum_{I \in \W_N^{\Sigma}} u(X_I)^2 \, \w(Q_I) 
\lesssim \w\bigg(\bigcup_{I \in \W_N^{\Sigma}} Q_I \bigg) \sup_{x \in F_Q} \N_Q u(x)^2
\nonumber\\
&\leq \w(C Q_0) \sup_{x \in F_Q} \N_Q u(x)^2 \lesssim \w(Q_0) \sup_{x \in F_Q} \N_Q u(x)^2. 
\end{align}
Finally, applying the same techniques, we get 
\begin{align}\label{eq:part-2}
&\iint_{\Omega} |A \nabla (u^2) \cdot \nabla\Psi_N \G| dX 
\nonumber\\
&\lesssim \sum_{I \in \W_N^{\Sigma}} \ell(I)^{-1} \iint_{I^*} |u| |\nabla u| \G \, dm
\nonumber\\
&\lesssim \sum_{I \in \W_N^{\Sigma}} u(X_I) \G(X_I)  \ell(I)^{-1} m(I) \bigg(\bariint_{I^*} |\nabla u|^2 \, dm \bigg)^{\frac12} 
\nonumber\\
&\lesssim \sum_{I \in \W_N^{\Sigma}} u(X_I) \ell(I)^{-2} \G(X_I) m(I) \bigg(\bariint_{I^{**}} |u|^2 \, dm\bigg)^{\frac12} 
\nonumber\\
&\lesssim \sum_{I \in \W_N^{\Sigma}} u(X_I)^2 \ell(I)^{-2} \G(X_I) m(I) 
\lesssim \w(Q_0) \sup_{x \in F_Q} \N_Q u(x)^2. 
\end{align}
Therefore, \eqref{eq:int-parts} is a consequence of \eqref{eq:JM}--\eqref{eq:part-2}.
\end{proof}
%%%%%%%%%%%%%%%%%%%%%%%%% PROOF PROOF PROOF %%%%%%%%%%%%%%%%%%%%%%%%

\subsection{A local good-$\lambda$ inequality}
The goal of this section is to establish some useful estimates, which will be utilized to show $\w_L \in A_{\infty}(\mu)$ implies $\mathcal{S} < \mathcal{N}$ estimates (cf. Section \ref{sec:SNN}).

First, we prove the openness of the level set of $\S^{\alpha}u$, which will be used in Section \ref{sec:SNN} to construct stopping-time cubes.

%%%%%%%%%%%%%%%%%%%%%%%% LEMMA LEMMA LEMMA %%%%%%%%%%%%%%%%%%%%%%%%
\begin{lemma}\label{lem:open}
Let $0<\alpha<\beta<\infty$. Assume that $u \in W_r(\Omega)$ and $\S^{\beta}u(x)<\infty$ $\mu$-a.e. $x \in \pom$. Then the set $E_{\lambda}:=\{x \in \pom: \S^{\alpha}u(x)>\lambda\}$ is open for any $\lambda>0$.
\end{lemma}
%%%%%%%%%%%%%%%%%%%%%%%% LEMMA LEMMA LEMMA %%%%%%%%%%%%%%%%%%%%%%%%

%%%%%%%%%%%%%%%%%%%%%%%%% PROOF PROOF PROOF %%%%%%%%%%%%%%%%%%%%%%%%
\begin{proof}
We will follow the strategy of \cite[Lemma~3.48]{MZ} and \cite[Lemma~4.6]{MPT}. Fix $\lambda>0$ and $x \in E_{\lambda}$. Then there exists $\eta>0$ such that 
\begin{align}\label{eq:Ex}
\Xi(x, x) > \lambda_1^2 := \bigg(\frac{\S^{\alpha}u(x)+\lambda}{2}\bigg)^2, 
\end{align}
where
\begin{align*}
\Xi(y, z) := \iint_{\Gamma^{\alpha}(y) \setminus B(y, \eta)} |\nabla u(X)|^2 \delta(X) \frac{dm(X)}{m(B(z, \delta(X)))}. 
\end{align*}
It suffices to prove that there exists $\varepsilon>0$ such that for every $y \in \Delta(x, \varepsilon \eta)$, 
\begin{align*}
\Xi(y, y) > \lambda^2. 
\end{align*}
Considering \eqref{eq:Ex}, we are reduced to showing 
\begin{align}\label{eq:ExEy}
|\Xi(x, x) - \Xi(y, y)| < \lambda_1^2 - \lambda^2, \quad\forall y \in \Delta(x, \varepsilon \eta), 
\end{align}
whenever $\varepsilon>0$ is small enough.  

To proceed, we fix $y \in \Delta(x, \varepsilon \eta)$ and split 
\begin{align}\label{eq:EEE}
|\Xi(x, x) - \Xi(y, y)| \le |\Xi(x, x) - \Xi(y, x)| + |\Xi(y, x) - \Xi(y, y)|. 
\end{align}
For any $X \in \Gamma^{\alpha}(y) \setminus B(y, \eta)$, we have $\eta \le |X-y|<(1+\alpha) \delta(X)$, and then 
\[
B(y, \delta(X)) \subset B(x, [1+\varepsilon(1+\alpha)]\delta(X)) \quad\text{and}\quad 
B(x, \delta(X)) \subset B(y, [1+\varepsilon(1+\alpha)]\delta(X)). 
\]
Since $m$ is doubling, we see that if $\varepsilon$ is sufficiently small, then 
\begin{align}\label{eq:EEE-1}
\textstyle{|\Xi(y, x) - \Xi(y, y)| < \frac12(\lambda_1^2 - \lambda^2)}. 
\end{align}
On the other hand, setting $V := (\Gamma^{\alpha}(x) \setminus B(x, \eta)) \Delta (\Gamma^{\alpha}(y) \setminus B(y, \eta))$, we have 
\begin{align*}
|\Xi(x, x) - \Xi(y, x)| \le \Xi(V) := \iint_{V} |\nabla u(X)|^2 \delta(X) \frac{dm(X)}{m(B(x, \delta(X)))}. 
\end{align*}
For any $0<\varepsilon<1/2$, write 
\begin{align*}
V_{\varepsilon} &:= \big\{X \in \Omega: |X-x| > \textstyle{\frac{\eta}{2}},\, (1+\alpha)(1-\varepsilon) \delta(X) 
< |X-x| < (1+\alpha)(1+\varepsilon) \delta(X) \big\}, 
\\
V'_{\varepsilon} &:= \{X \in \Gamma^{\beta}(x): (1-\varepsilon) \eta < |X-x| < (1+\varepsilon) \eta\}. 
\end{align*}
If $X \in (\Gamma^{\alpha}(x) \setminus B(x, \eta)) \setminus (\Gamma^{\alpha}(y) \setminus B(y, \eta))$, then $\eta \le |X-x|<(1+\alpha)\delta(X)$ and 
\[ 
|X-y| \ge (1+\alpha) \delta(X) \quad\text{ or } \quad |X-y| < \eta,
\] 
which together with $|x-y|<\varepsilon \eta < \varepsilon (1+\alpha)\delta(X)$ give  
\begin{align}\label{eq:VV-1}
(1+\alpha)\delta(X) > |X-x| \ge |X-y| - |x-y| > (1+\alpha)(1-\varepsilon)\delta(X), 
\end{align}
or 
\begin{align}\label{eq:VV-11}
|X-x|<(1+\alpha)\delta(X) \quad\text{ and }\quad \eta \le |X-x| \le |X-y| + |x-y| < (1+\varepsilon) \eta. 
\end{align}
If $X \in (\Gamma^{\alpha}(y) \setminus B(y, \eta)) \setminus (\Gamma^{\alpha}(x) \setminus B(x, \eta))$, then $\eta \le |X-y| < (1+\alpha)\delta(X)$ and 
\begin{align}\label{eq:VV-2}
|X-x| \ge (1+\alpha) \delta(X) \quad\text{ or }\quad |X-x| < \eta, 
\end{align}
which along with $|x-y|<\varepsilon \eta < \varepsilon (1+\alpha)\delta(X)$ imply 
\begin{align}\label{eq:VV-22}
|X-x|>(1-\varepsilon)\eta \quad\text{ and }\quad 
|X-x| \le |X-y| + |x-y| < (1+\alpha)(1+\varepsilon) \delta(X), 
\end{align}
Picking $\varepsilon<1/2$ small enough so that $(1+\alpha)(1+\varepsilon)<1+\beta$, we obtain from \eqref{eq:VV-1}--\eqref{eq:VV-22} that 
\begin{align}\label{eq:VVT}
V \subset (V_{\varepsilon} \cup V'_{\varepsilon}) \subset \Gamma^{\beta}(x). 
\end{align}
We may assume that $\S^{\beta}u(x)<\infty$. Then \eqref{eq:VVT} yields that 
\[
\Xi(V_{\varepsilon} \cup V'_{\varepsilon}) \le \Xi(\Gamma^{\beta}(x)) = \S^{\beta}u(x)^2<\infty, 
\]
and $V_{\varepsilon} \cup V'_{\varepsilon}$ decreases to an empty set as $\varepsilon$ tends to zero. Thus, 
\[
\lim_{\varepsilon \to 0} \Xi(V_{\varepsilon} \cup V'_{\varepsilon}) =0. 
\]
Choosing $\varepsilon$ sufficiently small, we deduce that 
\begin{align}\label{eq:EEE-2}
|\Xi(y, x) - \Xi(y, y)| \le \Xi(V) \le \Xi(V_{\varepsilon} \cup V'_{\varepsilon}) < \textstyle{\frac12}(\lambda_1^2 - \lambda^2). 
\end{align}
Consequently, \eqref{eq:ExEy} follows from \eqref{eq:EEE}, \eqref{eq:EEE-1} and \eqref{eq:EEE-2}. 
\end{proof}
%%%%%%%%%%%%%%%%%%%%%%%%% END END END PROOF %%%%%%%%%%%%%%%%%%%%%%%

Next, we give a lower bound for the truncated square function, which will be helpful to obtain a local good-$\lambda$ inequality. 

%%%%%%%%%%%%%%%%%%%%%%%%% LEMMA LEMMA LEMMA %%%%%%%%%%%%%%%%%%%%%%%
\begin{lemma}\label{lem:EQSr}
Fix $Q \in \D$ and set $r:=\ell(Q)$. Let $0<\alpha<\alpha_1<\beta$ and $\lambda>0$. Suppose that there exists $x_1 \in C_0 \Delta_Q$ such that $\S^{\alpha_1}u(x_1)<\lambda/\sqrt{c_3}$, where $c_3$ is given in \eqref{H3}. Then for any $\tau>0$ there exists $\delta>0$ sufficiently small such that 
\begin{align*}
\inf_{x \in E_Q} \S_{\tau r}^{\alpha}u(x) > \lambda, 
\end{align*}
where $E_Q:=\{x\in Q: \S^{\alpha}u(x)>2\lambda,\, \N^{\beta} u(x)\le \delta\,\lambda\}$. 
\end{lemma}
%%%%%%%%%%%%%%%%%%%%%%%%% LEMMA LEMMA LEMMA %%%%%%%%%%%%%%%%%%%%%%%

%%%%%%%%%%%%%%%%%%%%%%%%% PROOF RPOOF PROOF %%%%%%%%%%%%%%%%%%%%%%%%
\begin{proof}
Fix $x \in E_Q$. It suffices to show 
\begin{align}\label{eq:UU}
\mathscr{U}(x):=\iint_{U(x)} |\nabla u(Y)|^2 \delta(Y)^2 \, \frac{dm(Y)}{m(B(x, \delta(Y)))} < 2\lambda^2, 
\end{align}
where $U(x):=\Gamma^{\alpha}(x) \setminus B(x, \tau r)$, since $4\lambda^2 < \S^{\alpha}u(x)^2=\S_{\tau r}^{\alpha}u(x)^2 + \mathscr{U}(x)$.  Denote 
\[
U_1(x):=\Gamma^{\alpha}(x) \setminus B(x, \tau_1 r) \quad\text{ and }\quad 
U_2(x):=\Gamma^{\alpha}(x) \cap (B(x, \tau_1 r) \setminus B(x, \tau r)),
\] 
where $\tau_1>2\tau$ will be chosen later. Let $Y \in U_1(x)$. Then, we see that 
\begin{align*}
|Y-x_1| & \le |Y-x| + |x-x_1| \le |Y-x| + c r 
\\
&\le (1 + c \tau_1^{-1}) |Y-x| 
\le (1 + c \tau_1^{-1})(1+\alpha) \delta(Y)
< (1+\alpha_1) \delta(Y), 
\end{align*}
whenever $\tau_1>2\tau$ is large enough. That is, $U_1(x) \subset \Gamma^{\alpha_1}(x_1)$. Moreover, for all $Y \in U_1(x)$, $\tau_1 r \le |Y-x|<(1+\alpha) \delta(Y)$, which gives for any $z \in B(x_1, \delta(Y))$, 
\[
|z-x| \le |z-x_1| + |x_1-x| \le \delta(Y)+c r<2 \delta(Y),   
\]
provided $\tau_1>2\tau$ is large enough. That is, $B(x_1, \delta(Y)) \subset B(x, 2\delta(Y))$, and then $m(B(x_1, \delta(Y))) \le m(B(x, 2\delta(Y))) \le c_3 \, m(B(x, \delta(Y)))$. As a consequence, 
\begin{align}\label{eq:UU-1}
\mathscr{U}_1(x):=\iint_{U_1(x)} |\nabla u(Y)|^2 \delta(Y)^2 \, \frac{dm(Y)}{m(B(x, \delta(Y)))}  
\le c_3 \, \S^{\alpha_1}u(x_1)^2 
< \lambda^2. 
\end{align}
To analyze the integration over $U_2(x)$, we see that $U_2(x) \subset \bigcup_{j:2^{j-1} \tau<\tau_1} \Gamma_j(x)$, where $\Gamma_j(x):=\Gamma^{\alpha}(x) \cap (B(x, 2^j \tau r) \setminus B(x, 2^{j-1}\tau r))$. We also observe that every $\Gamma_j(x)$ can be covered by a finite collection of balls $B_j^k$ with radius $r(B_j^k) \simeq 2^j \tau r$, and that $B(x, 2^j \tau r) \subset C B_j^k$ for all $j$ and $k$. Let $(B_j^k)^*:=(1+\theta)B_j^k$, where $\theta \in (0, 1)$ is small enough such that $\bigcup_{j,k}(B_j^k)^* \subset \Gamma^{\beta}(x)$. Thus, using these and Lemma \ref{lem:PDE} part \eqref{list:Cacci}, we obtain 
\begin{align}\label{eq:UU-2}
\mathscr{U}_2(x) &:=\iint_{U_2(x)} |\nabla u(Y)|^2 \delta(Y)^2 \, \frac{dm(Y)}{m(B(x, \delta(Y)))} 
\nonumber\\
&=\sum_{j:2^{j-1}\tau<\tau_1} \iint_{\Gamma_j(x)} |\nabla u(Y)|^2 \delta(Y)^2 \, \frac{dm(Y)}{m(B(x, \delta(Y)))} 
\nonumber\\
&\lesssim \sum_{j:2^{j-1}\tau<\tau_1} \sum_{1 \le k \le C_n} \frac{(2^j \tau r)^2 \, m(B_j^k)}{m(B(x, 2^j \tau r))} \bariint_{B_j^k} |\nabla u(Y)|^2 \, dm(Y) 
\nonumber\\
&\lesssim \sum_{j:2^{j-1}\tau<\tau_1} \sum_{1 \le k \le C_n} \frac{m(B_j^k)}{m(B(x, 2^j \tau r))} \bariint_{(B_j^k)^*} |u(Y)|^2 dm(Y)
\nonumber\\
&\lesssim \sum_{j:2^{j-1}\tau<\tau_1} \N^{\beta}u(x)^2  
\le (\delta \lambda)^2 \log \bigg(\frac{\tau_1}{\tau}\bigg) < \lambda^2,  
\end{align}
provided $\delta=\delta(\tau)$ is sufficiently small depending only on $n$, $\alpha$, $\beta$, $\tau$, and $\tau_1$. Gathering \eqref{eq:UU-1} and \eqref{eq:UU-2}, we deduce 
\begin{align*}
\mathscr{U}(x)=\mathscr{U}_1(x) + \mathscr{U}_2(x) < 2\lambda^2. 
\end{align*}
This demonstrates \eqref{eq:UU}. 
\end{proof}
%%%%%%%%%%%%%%%%%%%%%%%%% PROOF PROOF PROOF %%%%%%%%%%%%%%%%%%%%%%%%

The main result of this section is the local good-$\lambda$ inequality below, which in turn gives the global inequality (cf. \eqref{eq:global}). The latter will be the core of the proof of $\w_L \in A_{\infty}(\mu)$ implies $\mathcal{S} < \mathcal{N}$ estimates (cf. Section \ref{sec:SNN}).

%%%%%%%%%%%%%%%%%%%%%%%%% LEMMA LEMMA LEMMA %%%%%%%%%%%%%%%%%%%%%%%
\begin{lemma}\label{lem:good-local}
Under the same hypotheses as Lemma $\ref{lem:EQSr}$, for every $Q \in \D$, we have 
\begin{align*}
\w_L^{X_Q} \big(\big\{x\in Q: \S^{\alpha}u(x)>2\lambda,\, \N^{\beta} u(x)\le \delta\,\lambda \big\}\big)
\lesssim \delta^2 \w_L^{X_Q}(Q),
\end{align*}
\end{lemma}
%%%%%%%%%%%%%%%%%%%%%%%%% LEMMA LEMMA LEMMA %%%%%%%%%%%%%%%%%%%%%%%

%%%%%%%%%%%%%%%%%%%%%%%%% PROOF RPOOF PROOF %%%%%%%%%%%%%%%%%%%%%%%%
\begin{proof}
Fix $Q \in \D$. Set $E_Q:=\big\{x\in Q: \S^{\alpha}u(x)>2\lambda,\, \N^{\beta} u(x)\le \delta\,\lambda \big\}$ and $F_Q:=\big\{x\in Q: \N^{\beta} u(x) \le \delta\,\lambda \big\}$. We may assume that $F_Q \neq \varnothing$, otherwise, there is nothing to prove. Since $\N^{\beta}u(x)$ is a continuous function, $Q\setminus F_Q=\{x \in Q: \N^{\beta}u(x)>\delta \lambda\}$ is relatively open in $Q$. By a stopping time argument, one can get a collection of maximal dyadic cubes $\F:=\{Q_j\} \subset \D_Q \setminus \{Q\}$ such that 
\[
Q\setminus F_Q=\{x \in Q: \N^{\beta}u(x)>\delta \lambda\}=\bigcup_{Q_j \in \F} Q_j. 
\]
Note that if we pick $\widehat{y} \in \pom$ such that $\delta(Y)=|Y-\widehat{y}|$, then $\Delta(\widehat{y}, \delta(Y)) \subset $$\Delta(x, (3+\alpha)\delta(Y))$ for any $Y \in \Gamma_{\tau r}^{\alpha}(x)$. This and Lemma \ref{lem:EQSr} give 
\begin{align*}
\lambda^2 \w_L^{X_Q}(E_Q) 
&\le \int_{E_Q} (\S_{\tau r}^{\alpha}u)^2 \, d\w_L^{X_Q}
\le \int_{F_Q} \iint_{\Gamma_{\tau r}^{\alpha}(x)} |\nabla u(Y)|^2 \delta(Y)^2 
\frac{dm(Y)}{m(B(x, \delta(Y)))} \, d\w_L^{X_Q}(x) 
\nonumber\\
&\lesssim \iint_{\bigcup_{x \in F_Q} \Gamma_{\tau r}^{\alpha}(x)} |\nabla u(Y)|^2 \delta(Y)^2 \bigg(\int_{F_Q} 
\mathbf{1}_{\Gamma_{\tau r}^{\alpha}(x)}(Y) d\w_L^{X_Q}(x)\bigg) \frac{dm(Y)}{m(B(\widehat{y}, \delta(Y)))}. 
\end{align*}
Note that choosing $\tau$ small enough, we can guarantee that $\Gamma_{\tau r}^{\alpha}(x) \subset \Gamma_Q(x)$ for all $x \in Q$. Additionally, for any $x \in Q$ and $Y \in \Gamma_Q(x)$, there exist $Q' \in \D_Q$ containing $x$ and a cube $I \subset U_{Q'}$ such that $Y \in I$ and $|Y-x| \simeq \ell(Q') \simeq \ell(I) \simeq \delta(Y)$. Then it follows from Lemma \ref{lem:wL-G} part \eqref{list-2} that 
\begin{align*}
\int_{F_Q} \mathbf{1}_{\Gamma_{\tau r}^{\alpha}(x)}(Y) \, d\w_L^{X_Q}(x) 
&\lesssim \w_L^{X_Q} \bigg(\bigcup_{\substack{Q' \in \D_Q \\ \ell(Q') \simeq \delta(Y) \simeq \dist(Y, Q')}} Q' \bigg) 
\\
&\lesssim \delta(Y)^{-2} m(B(\widehat{y}, \delta(Y))) G_L(X_Q, Y). 
\end{align*}
Together with Lemma \ref{lem:uGNQ}, this gives 
\begin{align*}
\lambda^2 \w_L^{X_Q}(E_Q) 
&\lesssim \iint_{\Omega_{\F, Q}} |\nabla u(Y)|^2 G_L(X_Q, Y) \, dm(Y)
\\
&\lesssim \w_L^{X_Q}(Q) \sup_{x \in F_Q} \N_Q u(x)^2 
\le \delta^2 \lambda^2 \w_L^{X_Q}(Q).  
\end{align*}  
This gives the desired estimate. 
\end{proof}
%%%%%%%%%%%%%%%%%%%%%%%%% END END END PROOF %%%%%%%%%%%%%%%%%%%%%%%%

We end up this section with the following inequality, which allows us to transfer Carleson measure estimates to the square function estimates.

%%%%%%%%%%%%%%%%%%%%%%%% LEMMA LEMMA LEMMA %%%%%%%%%%%%%%%%%%%%%%%%
\begin{lemma}\label{lem:CME-S}
Let $\rho$ be the function defined in \eqref{RHO}. For any $\alpha>0$, there exists $C_{\alpha}>1$ such that for all $x \in \pom$ and $0<r<\diam(\pom)$,  
\begin{align*}
\iint_{B(x, r) \cap \Omega} |\nabla u(X)|^2 \delta(X) \,  \rho(X)^{-1} \, dm(X)
\le C_{\alpha} \int_{\Delta(x, (2+\alpha)r)} \S^{\alpha}_{(1+\alpha)r}u(x)^2 \, d\mu(x). 
\end{align*}
\end{lemma}
%%%%%%%%%%%%%%%%%%%%%%%% LEMMA LEMMA LEMMA %%%%%%%%%%%%%%%%%%%%%%%%

%%%%%%%%%%%%%%%%%%%%%%%%% PROOF PROOF PROOF %%%%%%%%%%%%%%%%%%%%%%%%
\begin{proof}
Let $x \in \pom$ and $0<r<\diam(\pom)$. Given $X \in B(x, r) \cap \Omega$, we denote $\Delta_X :=\{z \in \pom: X \in \Gamma^{\alpha}(z)\}$ and choose $\widehat{x} \in \pom$ such that $|X-\widehat{x}|=\delta(X)$. It is easy to check that 
\begin{align*}
&\Delta(\widehat{x}, \alpha \delta(X)) \subset \Delta_X \subset \Delta(\widehat{x}, (2+\alpha) \delta(X)),
\\
&\Delta(z, \delta(X)) \subset \Delta(\widehat{x}, (3+\alpha) \delta(X)), \quad\forall z \in \Delta_X,  
\\  
&|X-z|<(1+\alpha)r, \quad |z-x|<(2+\alpha) r, \quad\forall z \in \Delta_X. 
\end{align*}
Since $\mu$ is doubling, we have $\rho(X) \simeq \rho(z, \delta(X))$ for any $z \in \Delta_X$, and 
\begin{align*}
&\iint_{B(x, r) \cap \Omega} |\nabla u(X)|^2 \delta(X) \, \rho(X)^{-1} \, dm(X) 
\\
&\simeq \iint_{B(x, r) \cap \Omega} |\nabla u(X)|^2 \delta(X) \, \rho(X)^{-1} \, 
\frac{\mu(\Delta_X)}{\mu(\Delta(\widehat{x}, \delta(X)))} \, dm(X) 
\\
&\lesssim \iint_{B(x, r) \cap \Omega} |\nabla u(X)|^2 \, \delta(X)  
\bigg(\int_{\Delta_X} \frac{\rho(z, \delta(X))^{-1}}{\mu(\Delta(z, \delta(X)))} d\mu(z) \bigg) dm(X) 
\\ 
&\le \int_{\Delta(x, (2+\alpha)r)} \bigg(\iint_{\Gamma^{\alpha}_{(1+\alpha)r}(z)} |\nabla u(X)|^2 \delta(X)^2\, 
\frac{dm(X)}{m(B(z, \delta(X)))} \bigg) d\mu(z) 
\\
&= \int_{\Delta(x, (2+\alpha)r)} \S^{\alpha}_{(1+\alpha)r} u(z)^2 \, d\mu(z). 
\end{align*}
This completes the proof. 
\end{proof}
%%%%%%%%%%%%%%%%%%%%%%%%% END END END PROOF %%%%%%%%%%%%%%%%%%%%%%%%

\subsection{Pointwise estimates for solutions}
The pointwise estimate below was shown by \cite[Lemma 4.9]{HKMP} in the upper-half space, extended to the general domain without any connectivity in \cite[Lemma~2.11]{AGMT}. 

%%%%%%%%%%%%%%%%%%%%%%%% LEMMA LEMMA LEMMA %%%%%%%%%%%%%%%%%%%%%%%%
\begin{lemma}\label{lem:u-point}
Let $u$ be a bounded weak solution of $Lu=0$ in $\Omega$, vanishing at $\infty$, and let $B(x_0, r)$ be a ball with $x_0 \in \pom$ and $0<r<\infty$. Suppose that $u$ vanishes continuously in $\pom \setminus B(x_0, r)$. Then, there are constants $\vartheta, C>0$ (depending only on the allowable parameters) such that 
\begin{align*}
|u(X)| \le C \bigg(\frac{r}{|X-x_0|}\bigg)^{n-1+\vartheta} 
\|u\|_{L^{\infty}(\Omega)},\quad X \in \Omega \setminus B(x_0, 2r). 
\end{align*} 
\end{lemma}
%%%%%%%%%%%%%%%%%%%%%%%% LEMMA LEMMA LEMMA %%%%%%%%%%%%%%%%%%%%%%%%

The result below concerns pointwise domination of solutions by $\BMO$ functions, which will be applied to show $\w_L \in A_{\infty}(\mu)$ implies the $\BMO$-solvability of the Dirichlet problem \eqref{eq:D} (see Section \ref{sec:AiBMO}).

%%%%%%%%%%%%%%%%%%%%%%%% LEMMA LEMMA LEMMA %%%%%%%%%%%%%%%%%%%%%%%%
\begin{lemma}\label{lem:u-BMO}
Fix a surface ball $\Delta=\Delta(x, r)$ with $x \in \pom$ and $0<r<\diam(\pom)$. Denote $\Delta^*:=8\Delta$ and  
\[
v(X) := \int_{\pom \setminus \Delta^*} |f-f_{\Delta^*}| \, d\w_L^X. 
\]
If $\w_L \in A_{\infty}(\mu)$, then 
\begin{align*}
v(X) \lesssim (\delta(X)/r)^{\varrho} \|f\|_{\BMO(\mu)}, \quad\forall X \in 2B(x, r) \cap \Omega, 
\end{align*}
where the implicit constant is independent of $\Delta$, and $\varrho \in (0, 1)$ is from Lemma \ref{lem:PDE}. 
\end{lemma}
%%%%%%%%%%%%%%%%%%%%%%%% LEMMA LEMMA LEMMA %%%%%%%%%%%%%%%%%%%%%%%%

%%%%%%%%%%%%%%%%%%%%%%%%% PROOF PROOF PROOF %%%%%%%%%%%%%%%%%%%%%%%%
\begin{proof}
For every $j \ge 1$, let $X_j:=X_{2^j \Delta^*}$ denote a corkscrew point relative to $2^j \Delta^*$. Since $\w_L \in A_{\infty}(\mu)$, we have for every $p \in (1, \infty)$, 
\begin{align}\label{eq:kLkL}
\bigg(\fint_{2^j \Delta^*} (k_L^{X_j})^{p'} \, d\mu\bigg)^{\frac{1}{p'}} \lesssim \fint_{2^j \Delta^*} k_L^{X_j} \, d\mu, 
\end{align}
where $k_L^{X_j}=d\w_L^{X_j}/d\mu$ is the Radon-Nikodym derivative with respect to $\mu$. For every $j \ge 2$ and every cube $\Delta' \subset 2^j \Delta^* \setminus 2^{j-1}\Delta^*$, we use Lemma \ref{lem:wL-G} and Harnack's inequality to obtain 
\[
\w_L^X(\Delta') \lesssim 2^{-j \varrho} \w_L^{X_{j-2}}(\Delta') \simeq 2^{-j \varrho} \w_L^{X_j}(\Delta'), 
\]
and hence, for any $y \in 2^j \Delta^* \setminus 2^{j-1}\Delta^*$, 
\begin{align}\label{eq:KLX}
k_L^X(y) = \lim_{\substack{\Delta' \to x \\ \Delta' \subset 2^j \Delta^* \setminus 2^{j-1} \Delta^*}} \frac{\w_L^X(\Delta')}{\mu(\Delta')} 
\lesssim  2^{-j \varrho} \lim_{\substack{\Delta' \to x \\ \Delta' \subset 2^j \Delta^* \setminus 2^{j-1} \Delta^*}} \frac{\w_L^{X_j}(\Delta')}{\mu(\Delta')} 
=2^{-j \varrho} k_L^{X_j}(y). 
\end{align} 
This along with H\"{o}lder's inequality and \eqref{eq:kLkL} implies 
\begin{align}\label{eq:VX}
v(X) &\lesssim \sum_{j=1}^{\infty} 2^{-j \varrho} \int_{2^j \Delta^* \setminus 2^{j-1} \Delta^*} |f-f_{\Delta^*}| k_L^{X_j} \, d\mu 
\nonumber\\ 
& \lesssim \sum_{j=1}^{\infty} 2^{-j \varrho} \mu(2^j \Delta^*) 
\bigg(\fint_{2^j \Delta^*}  |f-f_{\Delta^*}|^p \, d\mu \bigg)^{\frac1p} 
\bigg(\fint_{2^j \Delta^*} (k_L^{X_j})^{p'} \, d\mu \bigg)^{\frac{1}{p'}}
\nonumber\\  
& \lesssim \sum_{j=1}^{\infty} 2^{-j \varrho} \w_L^{X_j} (2^j \Delta^*) \|f\|_{\BMO(\mu)} 
\lesssim \|f\|_{\BMO(\mu)}. 
\end{align} 
Now let $X \in 2B(x, r) \cap \Omega$. By Lemma \ref{lem:PDE} part \eqref{list:Holder} and \eqref{eq:VX}, we conclude 
\begin{align*}
v(X) & \lesssim \bigg(\frac{\delta(X)}{3r}\bigg)^{\varrho} 
\bigg(\bariint_{B(x, 3r) \cap \Omega} v^2 \, dm \bigg)^{\frac12}
\\
& \lesssim (\delta(X)/r)^{\varrho} \sup_{B(x, r) \cap \Omega} v 
\lesssim (\delta(X)/r)^{\varrho} \|f\|_{\BMO(\mu)}. 
\end{align*} 
This completes the proof. 
\end{proof}
%%%%%%%%%%%%%%%%%%%%%%%%% END END END PROOF %%%%%%%%%%%%%%%%%%%%%%%%

%%%%%%%%%%%%%%%%%%%%%%% SECTION SECTION SECTION %%%%%%%%%%%%%%%%%%%%%%%
%%%%%%%%%%%%%%%%%%%%%%% SECTION SECTION SECTION %%%%%%%%%%%%%%%%%%%%%%%
\section{Proof of Theorem~\ref{thm:main}}\label{sec:main}
Given $p\in (1,\infty)$, we let $\eqref{list-Ai}_{p'}$ denote $\w_L \in RH_{p'}(\mu)$, and let $\eqref{list-Lp}_p$ denote that $L$ is $L^p(\mu)$-solvable. By $\rm{(\ref{list-CME}')}$, we mean that the Carleson measure estimate holds for all solutions of the form $u(X)=\w_L^X(S)$ with $S \subset \pom$ being an arbitrary Borel set. We first observe that $\eqref{list-CME}
\Longrightarrow {\rm (\ref{list-CME}')}$ trivially, and for any $p \in (1, \infty)$ the equivalence $\eqref{list-Ai}_{p'} \Longleftrightarrow \eqref{list-Lp}_p$ easily implies $\eqref{list-Ai} \iff \eqref{list-Lp}$. Also, since $\w_L \in RH_{p'}(\mu)$ implies $\w_L \in RH_{q'}(\mu)$ for all $q\ge p$, the equivalence $\eqref{list-Ai}_{p'} \Longleftrightarrow \eqref{list-Lp}_p$ yields $\eqref{list-Lp}_p \Longrightarrow \eqref{list-Lp}_q$ for all $q \ge p$. Thus, it suffices to show the following implications: 
\begin{align*}
&\eqref{list-Ai}_{p'} \Longrightarrow \eqref{list-Lp}_p \Longrightarrow \eqref{list-Ai}_{p'}, \quad 
\eqref{list-Ai} \Longrightarrow \eqref{list-CME}
\Longrightarrow {\rm (\ref{list-CME}')} \Longrightarrow \eqref{list-Ai}, 
\\
&\text{and }\quad \eqref{list-Ai} \Longrightarrow \eqref{list-SN} \Longrightarrow  \eqref{list-Ai} 
\Longrightarrow \eqref{list-BMO} \Longrightarrow {\rm (\ref{list-CME}')}. 
% \eqref{list-CME} \Longrightarrow \ref{list-App} \Longrightarrow \eqref{list-Ai}. 
\end{align*}
%We mention that the hypothesis \eqref{H6'} is only used in the proof of $\ref{list-App} \Longrightarrow \eqref{list-Ai}$.

%%%%%%%%%%%%%%%%%%% SUBSECTION SUBSECTION SUBSECTION %%%%%%%%%%%%%%%%%%%%
\subsection{Proof of $\eqref{list-Ai}_{p'} \Longrightarrow \eqref{list-Lp}_p$}\label{sec:Lp}
Let $u$ be the associated elliptic measure $L$-solution as in \eqref{eq:u-sol}. Assume $\w_L \in RH_{p'}(\mu)$. Then there exists $s>1$ such that $\w_L\in RH_{sp'} (\mu)$ (cf. \cite{CF, GR}), which in turn implies $\w_L \in RH_{sp'}(\Delta_0, \mu)$ uniformly for any surface ball $\Delta_0$. Therefore, using H\"{o}lder's inequality and recalling that $k_L^X=d\w_L^X/d\mu$, we get for any surface ball $\Delta \subset \Delta_0$, 
\begin{align}\label{eq:ff}
\fint_{\Delta} |f| \, d\w_L^{X_{\Delta_0}} 
& \simeq \frac{\mu(\Delta)}{\w_L^{X_{\Delta_0}}(\Delta)} 
\fint_{\Delta} |f| k_L^{X_{\Delta_0}}  \,d\mu 
\nonumber\\
& \leq \frac{\mu(\Delta)}{\w_L^{X_{\Delta_0}}(\Delta)} 
\bigg(\fint_{\Delta} (k_L^{X_{\Delta_0}})^{sp'} d\mu \bigg)^{\frac{1}{sp'}} 
\bigg(\fint_{\Delta} |f|^{(sp')'} d\mu \bigg)^{\frac{1}{(sp')'}}  
\nonumber\\
& \lesssim \frac{\mu(\Delta)}{\w_L^{X_{\Delta_0}}(\Delta)} 
\bigg(\fint_{\Delta} k_L^{X_{\Delta_0}} d\mu\bigg) 
\bigg(\fint_{\Delta} |f|^{(sp')'} d\mu \bigg)^{\frac{1}{(sp')'}} 
\nonumber\\
&= \bigg(\fint_{\Delta} |f|^{(sp')'} \,d\mu \bigg)^{\frac{1}{(sp')'}}. 
\end{align}

We may assume that $f$ has compact support. Pick a large surface ball $\Delta_0$ with radius $r_0$ such that $\supp f \subset \Delta_0$. Recall $\F_{\Delta_0}$ defined in Lemma \ref{lem:NrNQ}. For any $Q \in \F_{\Delta_0}$, we pick $z_Q \in Q \cap 4A_0 \Delta_0$ and then for any $z \in \Delta_0$, 
\[
|z-x_Q| \le |z-x_0| + |x_0-z_Q| + |z_Q-x_Q| 
\le r_0 + 4A_0 r_0 + A_0 \ell(Q) < 2A_0 \ell(Q). 
\] 
This gives that $\Delta_0 \subset 2\widetilde{\Delta}_Q$ for any $Q \in \F_{\Delta_0}$. Applying Lemmas \ref{lem:NQuf} and \ref{lem:NrNQ}, and \eqref{eq:ff}, we have for any $\alpha>0$, 
\begin{align}\label{eq:local-Nuf}
\|\N^{\alpha}_{r_0} u\|_{L^p(\Delta_0, \mu)}^p 
&\lesssim \bigg\|\sup_{0<r \leq c_{\alpha} r_0}\fint_{ \Delta(\cdot,r)}  
\sup_{Q \in \F_{\Delta_0}} \N_Q u \, d\mu \bigg\|_{L^p(\Delta_0, \mu)}^p 
\nonumber\\
&\lesssim \Big\|\sup_{Q \in \F_{\Delta_0}} \N_Q u \Big\|_{L^p(\Delta_0, \mu)}^p 
\lesssim \sup_{Q \in \F_{\Delta_0}}  \| \N_Q u\|_{L^p(\Delta_0, \mu)}^p  
\nonumber\\
&\lesssim \sup_{Q \in \F_{\Delta_0}} \int_{\Delta_0} \bigg(
\sup_{\substack{Q' \ni x \\ 0 < \ell(Q') < 4A_0 \ell(Q)}} \fint_{Q'} |f|^{(sp')'} d\mu \bigg)^{\frac{p}{(sp')'}} d\mu(x)
\nonumber\\
&\lesssim \int_{\Delta_0} |f|^p d\mu, 
\end{align}
where we have used that $\# \F_{\Delta_0} \lesssim 1$ and the boundedness of the local Hardy-Littlewood maximal function in the second term on $L^{\frac{p}{(sp')'}}(\Delta_0, \mu)$, which follows from $p>(sp')'$ and the doubling property of $\mu$. Observing that the implicit constants are independent of $r_0$ and letting $r_0 \to \infty$, we conclude 
\[
\|\N^{\alpha}u\|_{L^p(\pom, \mu)} \lesssim \|f\|_{L^p(\pom, \mu)}. 
\]
This completes the proof of $\mathrm{(b)}_p$. 
\qed

%%%%%%%%%%%%%%%%%%% SUBSECTION SUBSECTION SUBSECTION %%%%%%%%%%%%%%%%%%%%
\subsection{Proof of $\eqref{list-Lp}_p \Longrightarrow \eqref{list-Ai}_{p'}$}
With Lemmas \ref{lem:wL-G}, \ref{lem:AANN}, and \ref{lem:fNu} in hand, the proof is a slight modification of \cite[Section 4.2]{CDMT}. Fix $p \in (1,\infty)$ and assume that the Dirichlet problem \eqref{eq:D} for $L$ is solvable in $L^p(\mu)$. That is, for some fixed $\alpha_0$, $\|\N^{\alpha_0}u\|_{L^p(\pom, \mu)} \lesssim \|f\|_{L^p(\pom, \mu)}$ holds for $u$ as in \eqref{eq:u-sol} for any $f \in \mathscr{C}(\pom)$. Using this and Lemma \ref{lem:AANN}, we have for any $\alpha>0$ and $r>0$, 
\begin{equation}\label{eq:bpap}
\|\N^{\alpha}_r u\|_{L^p(\pom, \mu)}
\le \|\N^{\alpha} u\|_{L^p(\pom, \mu)}
\lesssim_{\alpha, \alpha_0} \|\N^{\alpha_0} u\|_{L^p(\pom, \mu)}
\lesssim_{\alpha_0} \|f\|_{L^p(\Delta_0, \mu)},
\end{equation}
for $u$ as in \eqref{eq:u-sol} with $f \in \mathscr{C}(\pom)$ with $\supp f \subset \Delta_0$ and for any surface ball $\Delta_0 \subset \pom$. 

To proceed we fix $\Delta_0:=\Delta(x_0, r_0)$ with $x_0 \in \pom$ and $0<r_0<\diam(\pom)$ such that $\mu(\Delta_0)>0$. Let $F \subset \Delta_0$ be a Borel set. Since $\w_L^{X_{\Delta_0}}$ and $\mu$ are Borel regular (see \cite[p.269, Proposition~1.3]{SS}), for each $\varepsilon > 0$, there exist compact set $K$ and an open set $U$ such that $K \subset F \subset U \subset \Delta_0$ satisfying
\begin{equation}\label{eq:UK}
 \w_L^{X_{\Delta_0}}(U \backslash K) + \mu(U \setminus K) < \varepsilon.
\end{equation}
Taking account of Urysohn's lemma, one can construct $f_F \in \mathscr{C}_c(\pom)$ such that $\mathbf{1}_K \leq f_F \leq \mathbf{1}_U$. Then, \eqref{eq:UK}, Lemma \ref{lem:fNu} and \eqref{eq:bpap} yield
\begin{align*}
\omega_L^{X_{\Delta_0}}(F)
&< \varepsilon + \w_L^{X_{\Delta_0}}(K)
\leq \varepsilon + \int_{\pom} f_F \, d \w_L^{X_{\Delta_0}}
\\
& \lesssim \varepsilon + \mu(\Delta_0)^{-\frac1p} \|\N^{\alpha}_{r_0} u\|_{L^p(\Delta_0, \mu)}  
\le \varepsilon + C_{\alpha} \mu(\Delta_0)^{-\frac1p} \|f_F\|_{L^p(\Delta_0, \mu)}
\\
&\lesssim \varepsilon +  C_{\alpha} \mu(U)^{\frac1p} \mu(\Delta_0)^{-\frac1p}
< \varepsilon +  C_{\alpha}(\mu(F) + \varepsilon)^{\frac1p} \mu(\Delta_0)^{-\frac1p}.
\end{align*}
Letting $\varepsilon \to 0+$, we obtain that $\w_L^{X_{\Delta_0}}(F) \lesssim \mu(F)^{\frac1p} \mu(\Delta_0)^{-\frac1p}$, and hence, $\w_L^{X_{\Delta_0}} \ll \mu$ in $\Delta_0$. Note that $\pom$ can be covered by a family of surface balls. These and Harnack's inequality imply that $\w_L \ll \mu$ in $\pom$. Denote $k_L^X :=d\w_L^X/d\mu \in L_{\loc}^1(\pom, \mu)$ which is well-defined $\mu$-a.e. in $\pom$. Thus, for every $f \in \mathscr{C}(\pom)$ with $\supp f \subset \Delta_0$, by Lemma \ref{lem:fNu} and \eqref{eq:bpap}, we obtain 
\begin{align*}
\bigg|\int_{\Delta_0} f\, k_L^{X_{\Delta_0}}\,d\mu \bigg|
=\bigg|\int_{\Delta_0} f\,d\w_L^{X_{\Delta_0}} \bigg|
\lesssim \mu(\Delta_0)^{-\frac1p} \|\N^{\alpha}_r u\|_{L^p(\Delta_0, \mu)}
\lesssim \mu(\Delta_0)^{-\frac1p} \|f\|_{L^p(\Delta_0, \mu)}. 
\end{align*}
This along with Lemma \ref{lem:wL-G} part \eqref{list-1} gives 
\begin{align*}
\bigg(\fint_{\Delta_0} (k_L^{X_{\Delta_0}})^{p'} \,d\mu \bigg)^{\frac{1}{p'}} 
\lesssim \frac{1}{\mu(\Delta_0)} 
\simeq \frac{\w_L^{X_{\Delta_0}}(\Delta_0)}{\mu(\Delta_0)} 
= \fint_{\Delta_0} k_L^{X_{\Delta_0}} \,d\mu. 
\end{align*}
Hence, $\w_L\in RH_{p'}(\mu)$. 
\qed

%%%%%%%%%%%%%%%%%% SUBSECTION SUBSECTION SUBSECTION %%%%%%%%%%%%%%%%%%%%%
\subsection{Proof of $\eqref{list-Ai} \Longrightarrow \eqref{list-CME}$}
We are going to utilize Lemma \ref{lem:disCarleson} to reduce Carleson measure estimates to the discrete setting. For this purpose, we follow the strategy in \cite[p. 7925]{CHMT}. Assume that $\w_L \in A_{\infty}(\mu)$. Let $u\in W_r(\Omega) \cap L^\infty(\Omega)$ satisfy $Lu=0$ in the weak sense in $\Omega$. By homogeneity, we may assume that $\|u\|_{L^{\infty}(\Omega)}=1$. For any $Q \in \D$, we set 
\begin{align*}
\alpha_Q := \iint_{U_Q} |\nabla u(X)|^2 \, \delta(X) \rho(X)^{-1}\, dm(X),  
\end{align*}
where the function $\rho$ is given in \eqref{RHO}. Recall the definition of $\m_{\alpha}$ in \eqref{eq:ma}. It suffices to prove $\m_{\alpha}$ is a discrete Carleson measure, that is, $\|\m_{\alpha}\|_{\mathcal{C}(\pom)}<\infty$. In view of Lemma \ref{lem:disCarleson}, it is reduced to showing that for every $Q_0 \in \D$ there exists a pairwise disjoint family $\F_{Q_0} = \{Q_j\}_j \subset \D_{Q_0} \setminus \{Q_0\}$ such that 
\begin{align}\label{eq:maDF}
\mu\bigg(Q_0 \setminus \bigcup_{Q_j \in \F_{Q_0}}Q_j \bigg) \gtrsim \sigma(Q_0) 
\quad \text{ and } \quad \m_{\alpha}(\D_{\F_{Q_0}, Q_0}) \lesssim \mu(Q_0).   
\end{align}

By Bourgain's estimate and Harnack's inequality, there exists $C_0 \ge 1$ such that $\w_L^{X_{Q_0}}(Q_0) \ge C_0^{-1}$. If we define 
\begin{align}\label{eq:wG}
\w := C_0 \mu(Q_0) \w_L^{X_{Q_0}} 
\quad\text{ and }\quad 
\G := C_0 \mu(Q_0) G_L(X_{Q_0}, \cdot), 
\end{align}
then 
\[
1 \le \frac{\w(Q_0)}{\mu(Q_0)} \le C_0. 
\]
Since $\w_L \in A_{\infty}(\mu)$, we have $\w \ll \mu$ and for some $1<q<\infty$, 
\begin{equation*}
\bigg(\fint_{Q_0} k(y)^q \, d\mu(y) \bigg)^{\frac1q} \lesssim 1, \quad\text{ where } k=d\w/d\mu. 
\end{equation*}
Then, by H\"{o}lder's inequality, we have for Borel set $F \subset Q_0$, 
\begin{align*}
\frac{\w(F)}{\mu(Q_0)} 
=\fint_{Q_0} \mathbf{1}_F \, k\, d\mu 
\le \bigg(\fint_{Q_0} (\mathbf{1}_F)^{q'} d\mu \bigg)^{\frac{1}{q'}} 
\bigg(\fint_{Q_0} k^q \, d\mu \bigg)^{\frac1q} 
\lesssim \bigg(\frac{\mu(F)}{\mu(Q_0)}\bigg)^{\frac{1}{q'}}. 
\end{align*}
This and Lemma \ref{lem:FQj} yield a pairwise disjoint family $\F_{Q_0}=\{Q_j\} \subset \D_{Q_0} \setminus\{Q_0\}$ such that 
\begin{align}\label{eq:FF}
\mu\bigg(Q_0 \setminus \bigcup_{Q_j \in \F_{Q_0}}Q_j \bigg) \gtrsim \mu(Q_0) 
\quad \text{ and } \quad \frac{\w(Q)}{\mu(Q)} \simeq 1, \qquad\forall Q \in \D_{\F_{Q_0}, Q_0}. 
\end{align}
It follows from Lemma \ref{lem:wL-G} and \eqref{eq:FF} that 
\begin{align}\label{eq:GG}
1 \simeq \frac{\w(\Delta_Q)}{\mu(Q)} \simeq \frac{m(B_Q)}{\ell(Q)^2 \mu(Q)} \G(X_Q) 
\simeq \delta(X)^{-1} \rho(X) \G(X), \qquad\forall X \in U_Q. 
\end{align}
Now using \eqref{eq:wG}, \eqref{eq:GG} and Lemma \ref{lem:uGNQ}, we conclude that 
\begin{align}\label{eq:mFF}
\m_{\alpha}(\D_{\F_{Q_0}, Q_0}) 
& = \sum_{Q \in \D_{\F_{Q_0}, Q_0}} \iint_{U_Q} |\nabla u(X)|^2 \delta(X) \rho(X)^{-1} \, dm(X)
\nonumber\\
&\simeq \sum_{Q \in \D_{\F_{Q_0}, Q_0}} \iint_{U_Q} |\nabla u(X)|^2 \G(X) \, dX
\nonumber\\
&\lesssim \iint_{\Omega_{\F_{Q_0}, Q_0}} |\nabla u(X)|^2 \G(X) \, dX 
\lesssim \mu(Q_0), 
\end{align}
where we used the bounded overlap of the family $\{U_Q\}_{Q \in \D}$. Consequently, \eqref{eq:maDF} follows from \eqref{eq:FF} and \eqref{eq:mFF}. 

%%%%%%%%%%%%%%%%%% SUBSECTION SUBSECTION SUBSECTION %%%%%%%%%%%%%%%%%%%%%
\subsection{Proof of ${\rm (\ref{list-CME}')} \Longrightarrow \eqref{list-Ai}$}\label{sec:CMEAi}
We follow the scheme from \cite[p. 7915]{CHMT}, which extended \cite[Theorem 3.1]{KKiPT} in the upper half-space to the case of 1-sided CAD. Assume that $\rm{(\ref{list-CME}')}$ holds, that is, the Carleson measure estimate is true for every solution in the form $u(X)=\w_L^X(S)$ with $S \subset \pom$ being an arbitrary Borel set. Fix $\alpha \in (0, 1)$ and $Q_0 \in \D$, and take a Borel set $F \subset Q_0$ such that $\w_L^{X_{Q_0}}(F) < \beta \w_L^{X_{Q_0}}(Q_0)$, where $\beta \in (0, 1)$ will be chosen later. Considering Lemma \ref{lem:sq-lower} and picking $\beta \in (0, \beta_0)$, we can find a Borel set $S \subset Q_0$ such that $u(X)=\w_L^X(S)$ satisfies 
\begin{align}\label{eq:FSQ}
\inf_{x \in F} S_{Q_0}^{\eta} u(x) \geq C_{\eta}^{-1} (\log \beta^{-1})^{\frac12}. 
\end{align}
To continue, for each $Y \in \Omega$, let $\widehat{y} \in \pom$ be such that $\delta(Y) = |Y-\widehat{y}|$. Observe that 
\begin{align}\label{eq:CY}
\{x \in Q_0: Y \in \Gamma_{Q_0}^{\eta}(x)\} \subset \Delta(\widehat{y}, C'_{\eta} \delta(Y)). 
\end{align}
Indeed, let $x \in Q_0$ and $Y \in \Gamma_{Q_0}^{\eta}(x)$. By definition, there exist $Q \in \D_{Q_0}$ and $Q' \in \D_Q$ with $\ell(Q')>\eta^3 \ell(Q)$ such that $Y \in U_{Q'}$ and $x \in Q$. Then $Y \in I^*$ for some $I \in \W_{Q'}^*$, and hence, 
\begin{align*}
\delta(Y) \simeq \ell(I) \simeq \ell(Q')\le \ell(Q)<\eta^{-3}\ell(Q').
\end{align*}
This further implies that 
\begin{align}\label{eq:Ceta}
|x-\widehat{y}| \le \diam(Q) + \dist(Y, Q') + |Y-\widehat{y}| 
\lesssim \ell(Q) + \ell(Q') + \delta(Y)
< C'_{\eta} \delta(Y). 
\end{align}
Additionally, it follows from \eqref{eq:Ceta} that 
\begin{align}\label{eq:Byy}
B(\widehat{y}, \delta(Y)) \subset B(x, C''_{\eta} \delta(Y)), \qquad\forall x \in Q_0 \text{ and } Y \in \Gamma_{Q_0}^{\eta}(x). 
\end{align}
Collecting \eqref{eq:FSQ}, \eqref{eq:CY} and \eqref{eq:Byy}, we deduce that 
\begin{align*}
&C_{\eta}^{-2} (\log \beta^{-1}) \mu(F)  \leq \int_F S_{Q_0}^{\eta} u(x)^2 d\mu(x)
\\
& \leq \int_{Q_0} \bigg( \iint_{\Gamma_{Q_0}^{\eta}(x)} 
|\nabla u(Y)|^2 \delta(Y)^2 \, \frac{dm(Y)}{m(B(x, \delta(Y)))}  \bigg) d\mu(x) 
\\
&\le \iint_{B_{Q_0}^* \cap \Omega} |\nabla u(Y)|^2 \delta(Y)^2 
\bigg( \int_{Q_0}  \mathbf{1}_{\Gamma_{Q_0}^{\eta}(x)}(Y) \frac{d\mu(x)}{m(B(x, \delta(Y)))} \bigg) dm(Y)
\\
&\lesssim_{\eta} \iint_{B_{Q_0}^* \cap \Omega} |\nabla u(Y)|^2 \delta(Y)^2 
\frac{\mu(\Delta(\widehat{y}, \delta(Y)))}{m(B(\widehat{y}, \delta(Y)))}  dm(Y)
\\
&= \iint_{B_{Q_0}^* \cap \Omega} |\nabla u(Y)|^2 \delta(Y) \rho(Y)^{-1} dm(Y)
\lesssim \mu(Q_0), 
\end{align*}
where the function $\rho$ is given in \eqref{RHO}. Then choosing $\beta \in (0, \beta_0)$ small enough, we have 
\begin{align*}
\mu(F) \leq C'''_{\eta} (\log \beta^{-1})^{-1} \, \mu(Q_0) < \alpha\, \mu(Q_0). 
\end{align*}
Thus, we have proved that for any $\alpha \in (0, 1)$, there exists $\beta \in (0,1)$ such that for every $Q_0 \in \D$ and every Borel set $F \subset Q_0$,  
\begin{equation*}
\w_L^{X_{Q_0}}(F) < \beta \w_L^{X_{Q_0}}(Q_0) 
\quad \Longrightarrow \quad \mu(F) < \alpha \mu(Q_0). 
\end{equation*}  
This and Lemma \ref{lem:Ai-Ai} immediately imply $\w_L \in A_{\infty}(\mu)$. 
\qed

%%%%%%%%%%%%%%%%%% SUBSECTION SUBSECTION SUBSECTION %%%%%%%%%%%%%%%%%%%%%
\subsection{Proof of $\eqref{list-Ai} \Longrightarrow \eqref{list-SN}$}\label{sec:SNN}
The proof is based on the standard good-$\lambda$ argument of \cite{DKP} (see also \cite{CDMT, MZ}). Assume that $\w_L \in A_{\infty}(\mu)$. Let $q \in (0, \infty)$, and $u \in W_r(\Omega)$ be a weak solution of $Lu=0$ in $\Omega$. Assume that $\|\N^{\alpha}u\|_{L^q(\pom, \mu)}<\infty$, otherwise, there is nothing to prove. 

We begin with the case $\diam(\pom)=\infty$. First assume that $\|\S^{\alpha}u\|_{L^q(\pom, \mu)}<\infty$. Under this assumption, to get \eqref{list-SN}, it suffices to show the following good-$\lambda$ inequality: for any $\varepsilon>0$, there exists $\delta=\delta(\varepsilon)>0$ such that for all $\lambda>0$, 
\begin{equation}\label{eq:global}
\mu \big(\big\{x \in \pom: \S^{\alpha} u(x) > 2\lambda, \N^{\beta} u(x) \leq \delta \lambda  \big\}\big)
\leq \varepsilon \mu(\{x \in \pom: \S^{\alpha} u(x) > \lambda\}), 
\end{equation}
where $0<\alpha<\beta<\infty$. Indeed, it follows from \eqref{eq:global} that 
\begin{align}\label{eq:SSNu}
\|\S^{\alpha}u\|_{L^q(\pom, \mu)}^q 
&=q \, 2^q \int_{0}^\infty \lambda^{q-1} 
\mu(\{x \in \pom: \S^{\alpha}u(x) > 2\lambda \}) d\lambda 
\nonumber \\
&\leq \varepsilon q \, 2^q \int_{0}^{\infty} \lambda^{q-1} 
\mu\big(\{x \in \pom: \S^{\alpha}u(x) > \lambda \}\big) d\lambda
\nonumber \\
&\quad+ q \, 2^q \int_{0}^{\infty} \lambda^{q-1} 
\mu\big(\{x \in \pom: \N^{\beta}u(x) > \delta \lambda \}\big) d\lambda
\nonumber \\
& = \varepsilon \, 2^q  \|\S^{\alpha}u\|_{L^q(\pom, \mu)}^q  
+ 2^q \delta^{-q} \| \N^{\beta}u \|_{L^q(\pom, \mu)}^q
\nonumber \\
&<\frac12 \|\S^{\alpha}u\|_{L^q(\pom, \mu)}^q  
+ C_{q, \delta} \| \N^{\beta}u \|_{L^q(\pom, \mu)}^q, 
\end{align}
provided $\varepsilon \in (0, 2^{-q-1})$. This and \eqref{eq:NFNF} imply 
\begin{align}\label{eq:SN-Lq}
\|\S^{\alpha}u\|_{L^q(\pom, \mu)}  \le C_q \| \N^{\beta}u \|_{L^q(\pom, \mu)} 
\le C_{q, \alpha, \beta} \| \N^{\alpha}u \|_{L^q(\pom, \mu)}. 
\end{align}

Let us show \eqref{eq:global}. Given $\lambda>0$ and $\alpha_1 \in (\alpha, \beta)$, we set $E_{\lambda}:=\{x \in \pom: \S^{\alpha_1} u(x) > \lambda/\sqrt{c_3}\}$, where $c_3$ is the constant in \eqref{H3}. We may assume that $E_{\lambda}$ is not empty, otherwise \eqref{eq:global} is trivial. Considering \eqref{eq:AFAF}, we have $\|\S^{\beta}u\|_{L^q(\pom, \mu)} \simeq \|\S^{\alpha}u\|_{L^q(\pom, \mu)}<\infty$, and so $\S^{\beta}u<\infty$ $\mu$-a.e. This along with Lemma \ref{lem:open} leads that $E_{\lambda}$ is open.  Besides, Chebyshev's inequality gives 
\begin{align}\label{eq:Elam}
\mu(E_{\lambda}) 
\le (\lambda/\sqrt{c_3})^{-q} \|\S^{\alpha_1}u\|_{L^q(\pom, \mu)}^q 
< \infty.  
\end{align}
By \cite[Proposition 2.1]{CKP}, we see that $\diam(\pom)=\infty$ if and only if $\mu(\pom)=\infty$. This together with \eqref{eq:Elam} gives that $E_{\lambda} \subsetneq \pom$. Hence, we can find a family of maximal dyadic cubes $\F:=\{Q_j\}$ such that $E_{\lambda}=\bigcup_{Q_j \in \F} Q_j$. Observe that $\{x \in \pom: \S^{\alpha} u(x) > 2\lambda\} \subset E_{\lambda}$. Therefore, to obtain \eqref{eq:global}, it is enough to prove 
\begin{equation}\label{eq:local-mu}
\mu \big(\big\{x \in Q: \S^{\alpha} u(x) > 2\lambda, \N^{\beta} u(x) \leq \delta \lambda  \big\}\big)
\leq \varepsilon \mu(Q), \quad \forall \, Q \in \F. 
\end{equation}
By maximality, we see that $\widehat{Q} \not\subset E_{\lambda}$, where $\widehat{Q}$ is the dyadic father of $Q$. Then there exists $x_1 \in \widehat{Q}$ such that $\S^{\alpha_1}u(x_1)<\lambda/\sqrt{c_3}$. Hence, noting that $\widehat{Q} \subset C_0 \Delta_Q$ and invoking $\w_L \in A_{\infty}(\mu)$, Lemmas \ref{lem:Ainfity} and \ref{lem:good-local}, we derive \eqref{eq:local-mu} immediately.

Next, we are going to remove the priori assumption. To this end, we introduce the notation
\begin{align*}
\Gamma^{\alpha, k}(x) &:= \{Y \in \Gamma^{\alpha}(x): 2^{-k} \le \delta(Y) \le 2^k\}, 
\\
I(x, k) &:= \{Y \in \Omega: |Y-x|<2^{k+1}, 2^k \le \delta(Y) < 2^{k+1}\}, 
\\
J(x, k) &:= \{Y \in \Omega: |Y-x|<2^{k+2}, 2^{k-1} \le \delta(Y) < 2^{k+2}\}. 
\end{align*}
Then, one can check that for any $k \in \Z$ and $K \in \mathbb{N}_+$, 
\begin{align}\label{eq:IJG}
\Gamma^{\frac14, K}(x) \subset \bigcup_{k=-K-1}^K I(x, k) \quad\text{and}\quad 
J(x,k) \subset \Gamma^7(x). 
\end{align}
Observe that $|u|$ is a subsolution of $Lu=0$ in $\Omega$. Then by \eqref{eq:IJG} and Lemma \ref{lem:PDE} part \eqref{list:Moser}, 
\begin{align*}
\sup_{J(x, k)} |u| \lesssim \bariint_{J(x, k)} |u| \, dm \le \N^7 u(x),  
\end{align*}
which together with Lemma \ref{lem:PDE} part \eqref{list:Cacci} gives 
\begin{align*}
\iint_{I(x, k)} |\nabla u|^2 \, dm 
\lesssim 2^{-2k} \iint_{J(x, k)} |u|^2 \, dm
\lesssim 2^{-2k} m(B(x, 2^{k+2})) \N^7 u(x)^2. 
\end{align*}
This and the doubling property of $\mu$ imply 
\begin{equation}\label{eq:IN}
\iint_{I(x, k)} |\nabla u|^2 \delta(Y)^2 \, \frac{dm(Y)}{m(B(x, \delta(Y)))} 
\simeq \frac{2^{2k}}{m(B(x, 2^k))} \iint_{I(x, k)} |\nabla u|^2 \, dm 
\lesssim \N^7 u(x)^2. 
\end{equation}
Accordingly, we obtain from \eqref{eq:IJG} and \eqref{eq:IN} that 
\begin{align}\label{eq:Skuk}
\S^{\frac14, K}u(x)^2 &:= \iint_{\Gamma^{\frac14, K}(x)} |\nabla u(Y)|^2 \delta(Y)^2\, \frac{dm(Y)}{m(B(x, \delta(Y)))}  
\nonumber\\
&\le \sum_{k=-K-1}^K \iint_{I(x, k)} |\nabla u(Y)|^2 \delta(Y)^2 \, \frac{dm(Y)}{m(B(x, \delta(Y)))}
\lesssim K \N^7 u(x)^2, 
\end{align}
where the implicit constants are independent of $K$. On the other hand, in view of \eqref{eq:NFNF}, 
\begin{align*} 
\|\N^{\frac14}u\|_{L^q(\pom, \mu)} \le \|\N^7 u\|_{L^q(\pom, \mu)} \lesssim \|\N^{\alpha} u\|_{L^q(\pom, \mu)} < \infty.
\end{align*} 
Considering \eqref{eq:Skuk}, we have 
\begin{align*}
\|\S^{\frac14, K}u\|_{L^q(\pom, \mu)}<C_K<\infty \quad\text{and}\quad 
\sup_K \|\N^{\frac14, K}u\|_{L^q(\pom, \mu)}<\infty. 
\end{align*}
Then applying the previous arguments for $\S^{\frac14, K}$ and $\N^{\frac14, K}$, we get 
\begin{align*}
\|\S^{\frac14, K}u\|_{L^q(\pom, \mu)} 
\le C \|\N^{\frac14, K}u\|_{L^q(\pom, \mu)} 
\le C \|\N^{\frac14}u\|_{L^q(\pom, \mu)}. 
\end{align*}
where the constant $C$ is independent of $K$. Hence, letting $K \to \infty$ and invoking Lemma \ref{lem:AANN}, we conclude that 
\begin{align*}
\|\S^{\alpha}u\|_{L^q(\pom, \mu)} 
\simeq \|\S^{\frac14}u\|_{L^q(\pom, \mu)} 
\lesssim \|\N^{\frac14}u\|_{L^q(\pom, \mu)}
\simeq \|\N^{\alpha}u\|_{L^q(\pom, \mu)}. 
\end{align*}
This completes the proof of the case $\diam(\partial \Omega)=\infty$.   

It remains to handle the case $\diam(\partial \Omega)<\infty$, for which we adapt the idea in \cite[Section 4.1]{HMM2} to our situation. Then $\partial \Omega$ itself is the largest cube in $\mathbb{D}$, say $\partial \Omega = Q_0$. Recall the traditional cone $\Gamma^{\alpha}$ and the dyadic cone $\Gamma_{\mathbb{D}}$. Note that given $\alpha_1>0$, there exist $\vartheta = \vartheta(\alpha_1)$ (cf. \eqref{eq:WQ}) sufficiently large and $\alpha_2 = \alpha_2(\vartheta)$ sufficiently large so that 
\begin{align}\label{GAA}
\Gamma^{\alpha_1}(x) \subset \Gamma_{\mathbb{D}}(x)  \subset \Gamma^{\alpha_2}(x), 
\quad\text{ for all } x \in \partial \Omega. 
\end{align}
Recall the dyadic square function $\mathcal{S}_{Q_0}$ and the dyadic non-tangential maximal function $\mathcal{N}_{Q_0}$ in Definition \ref{def:SN}. As argued in the preceding case, in light of Lemma \ref{lem:AANN}, \eqref{GAA}, and that $\partial \Omega = Q_0$, we are deduced to proving the dyadic good-$\lambda$ inequality: for any $\varepsilon>0$, there exists $\delta=\delta(\varepsilon)>0$ such that for all $\lambda>0$, 
\begin{equation}\label{LXQ}
\mathcal{I}_{\lambda}
:= \omega_L^{X_{Q_0}} \big(\big\{x \in F: \S_{Q_0} u(x) > 2\lambda \big\}\big)
\lesssim \varepsilon \, \omega_L^{X_{Q_0}}(\{x \in \pom: \S_{Q_0} u(x) > \lambda\}), 
\end{equation}
where $F := \{x \in Q_0: \N_{Q_0} u(x) \leq \delta \lambda\}$.

Much as in \cite[Section 4.1]{HMM2}, we arrive at   
\begin{align*}
\mathcal{I}_{\lambda}
\le \lambda^{-2} \sum_{Q_j \in \mathcal{F}} \int_{F \cap Q_j} \mathcal{S}_{Q_j} u(x)^2 \, d\omega_L^{X_{Q_0}}(x),  
\end{align*}
where $\mathcal{F} := \{Q_j\} \subset \mathbb{D}_{Q_0}$ is a disjoint family satisfying $\bigcup_{Q_j \in \mathcal{F}} Q_j = \{x \in Q_0: \mathcal{S}_{Q_0} u(x) > \lambda\}$. Let $\mathcal{F}_j := \{Q_j^k\} \subset \D_{Q_j} \setminus \{Q_j\}$ be the collection of maximal dyadic cubes such that $Q_j \setminus F = \{x \in Q_j: \N_{Q_0}u(x) > \delta \lambda\} = \bigcup_{Q_j^k \in \F_j} Q_j^k$. Analogously to the proof of Lemmas \ref{lem:good-local}, we obtain 
\begin{align*}
\mathcal{I}_{\lambda} 
&\le \lambda^{-2} \sum_{Q_j \in \mathcal{F}} \int_{F \cap Q_j} \mathcal{S}_{Q_j} u(x)^2 \, d\omega_L^{X_{Q_0}}(x) 
\\ 
&\lesssim  \lambda^{-2} \sum_{Q_j \in \mathcal{F}} 
\iint_{\Omega_{\mathcal{F}_j, Q_j}} |\nabla u(Y)|^2 G_L(X_{Q_0}, Y) \, dm(Y)
\\
&\lesssim \lambda^{-2} \sum_{Q_j \in \mathcal{F}} \w_L^{X_{Q_0}}(Q_j) \|\mathcal{N}_{Q_0} u\|_{L^{\infty}(F)}^2 
\\
&\le \delta^2 \, \omega_L^{X_{Q_0}}(\{x \in Q_0: \mathcal{S}_{Q_0} u(x) > \lambda\}). 
\end{align*}
Choosing $\delta := \varepsilon^{1/2}$, we have proved \eqref{LXQ}. 
\qed

%%%%%%%%%%%%%%%%%% SUBSECTION SUBSECTION SUBSECTION %%%%%%%%%%%%%%%%%%%%%
\subsection{Proof of $\eqref{list-SN} \Longrightarrow \eqref{list-Ai}$}
Assume that $\S<\N$ estimate holds for some fixed $\alpha_0$ and $q=\frac{\log_2 c_3}{n-1}$, where the constant $c_3$ is the same as \eqref{H3}. Given a cube $Q_0 \in \D$ and any arbitrary Borel set $S \subset Q_0$, we denote 
\[
u(X) :=\w_L^X(S) \quad\text{ and }\quad v(X) :=\w_L^X(Q_0),\quad X \in \Omega.
\] 
Observe that both $u$ and $v$ are weak solutions of $Lu=0$ in $\Omega$, vanishing at $\infty$, $u \le v$, and $\|u\|_{L^{\infty}(\Omega)} \le \|v\|_{L^{\infty}(\Omega)} \le 1$. By Lemma \ref{lem:u-point} and that $\mu(2^j Q_0) \lesssim c_3^j \mu(Q_0)$, we have 
\begin{align*}
\|\N^{\alpha} v\|_{L^q(\pom, \mu)}^q 
&=\int_{2Q_0} \N^{\alpha} v(x)^q\, d\mu(x) + \sum_{j=1}^{\infty} \int_{2^{j+1}Q_0 \setminus 2^jQ_0} \N^{\alpha} v(x)^q\, d\mu(x)
\nonumber \\
&\lesssim \|v\|_{L^{\infty}(\Omega)}^q \bigg[\mu(2Q_0) + 
\sum_{j=1}^{\infty} \int_{2^{j+1}Q_0 \setminus 2^j Q_0} 
\bigg(\frac{\ell(Q_0)}{|x-x_{Q_0}|}\bigg)^{q(n-1+\vartheta)} d\mu(x)\bigg]
\nonumber \\
&\lesssim \mu(Q_0) \bigg[1 + \sum_{j=1}^{\infty} 2^{-j[q(n-1+\vartheta)-\log_2 c_3]} \bigg]
\lesssim  \mu(Q_0), 
\end{align*}
where we used $\vartheta>0$ and $q(n-1+\vartheta)>\log_2 c_3$ in the last step. Then, together with Lemma~\ref{lem:AANN} (applied to $F(Y)=|\nabla u(Y)| \delta(Y)$), this implies that for any $\alpha$ large enough, 
\begin{equation}\label{eq:SSN}
\|\S^{\alpha} u\|_{L^q(\pom, \mu)} 
\lesssim \|\S^{\alpha_0} u\|_{L^q(\pom, \mu)} 
\lesssim \|\N^{\alpha_0} u\|_{L^q(\pom, \mu)} 
\le \|\N^{\alpha_0} v\|_{L^q(\pom, \mu)}
\lesssim \mu(Q_0)^{\frac1q}. 
\end{equation}

Fix $\gamma \in (0,1)$ and $Q_0 \in \D$, and pick a Borel set $F \subset Q_0$ so that $\w_L^{X_{Q_0}}(F) \le \beta \w_L^{X_{Q_0}}(Q_0) $, where $\beta\in (0,1)$ is small enough to be chosen. Applying Lemma \ref{lem:sq-lower}, if we assume that $0<\beta<\beta_0$, then there exists a Borel set $S \subset Q_0$ such that $u(X)=\omega^X_L(S)$ satisfies 
\begin{equation}\label{eq:inF}
\inf_{x \in F} \S_{Q_0}^{\eta}u(x) \ge C_\eta^{-1} \big(\log(\beta^{-1})\big)^{\frac12}.
\end{equation}
Then, gathering \eqref{eq:SSN} and \eqref{eq:inF}, we obtain 
\begin{align*}
C_\eta^{-q} \log{(\beta^{-1})}^{\frac{q}{2}} \mu(F)
& \leq \int_F (\S_{Q_0}^{\eta}u)^q \, d\mu 
\leq \int_{Q_0} (\S^{\alpha}u)^q \, d\mu
\lesssim \mu(Q_0).
\end{align*}
Choosing $\beta$ small enough, we have 
\[
\mu(F) \le C_{\eta,q} \, \log{(\beta^{-1})}^{-\frac{q}{2}} \mu(Q_0) 
\le \gamma \mu(Q_0). 
\]
Therefore, we have shown that given $\gamma \in (0,1)$ there exists $\beta \in (0,1)$ such that for every $Q_0 \in \D$ and for every Borel set $F\subset Q_0$, 
\begin{equation*}
\w_L^{X_{Q_0}}(F) \le \beta \, \w_L^{X_{Q_0}}(Q_0) 
\quad\Longrightarrow\quad \mu(F) \le \gamma \, \mu(Q_0).
\end{equation*}
This and Lemma \ref{lem:Ai-Ai} imply $\w_L \in A_{\infty}(\mu)$. 
\qed

%%%%%%%%%%%%%%%%%% SUBSECTION SUBSECTION SUBSECTION %%%%%%%%%%%%%%%%%%%%%
\subsection{Proof of $\eqref{list-Ai} \Longrightarrow \eqref{list-BMO}$}\label{sec:AiBMO}
Assume that \eqref{list-Ai} holds. We borrow the ideas from \cite{DKP, MZ}. Fix $f \in \mathscr{C}(\pom)$ and let $u$ be the associated solution of $L$ as in \eqref{eq:u-sol}. Fix a surface ball $\Delta=B \cap \pom$. Write $\Delta^*=8\Delta$ and $f_{\Delta^*}:=\fint_{\Delta^*} f\, d\mu$. Then, 
\begin{align*}
f &=(f-f_{\Delta^*}) \mathbf{1}_{\Delta^*} + (f-f_{\Delta^*}) \mathbf{1}_{\pom \setminus \Delta^*} + f_{\Delta^*}  
=: f_1 + f_2 +f_3.
\end{align*}
Hence, there holds 
\begin{align*}
u(X) = \int_{\pom} f \, d\w_L^X
=\int_{\pom} f_1 \, d\w_L^X + \int_{\pom} f_2 \, d\w_L^X + f_3 
&=:u_1(X) + u_2(X) + f_3. 
\end{align*}
Note that $u_1, u_2 \in \mathscr{C}(\pom)$. It suffices to show 
\begin{align}\label{eq:u12} 
\frac{1}{\mu(\Delta)} \iint_{B} |\nabla u_i(X)|^2 \delta(X) \, \rho(X)^{-1} \, dm(X)
\lesssim \|f\|^2_{\BMO(\mu)}, \quad i=1,2, 
\end{align}
where the function $\rho$ is given in \eqref{RHO} and the implicit constant is independent of $\Delta$. 

Let us observe that we have already proved that \eqref{list-Ai} implies both \eqref{list-Lp} and \eqref{list-SN} hold. Then, this and Lemma \ref{lem:CME-S} yield 
\begin{align}\label{eq:Car-u1}
\iint_{B} |\nabla u_1(X)|^2 \delta(X) \rho(X)^{-1} \, dm(X)
\lesssim \|\S^{\alpha} u_1\|_{L^2(\mu)}^2 
\lesssim \|\N^{\alpha} u_1\|_{L^2(\mu)}^2 
\lesssim \|f_1\|_{L^2(\mu)}^2 
\nonumber\\
= \int_{\Delta^*} |f-f_{\Delta^*}|^2\, d\mu
\lesssim \mu(\Delta^*) \|f\|_{\BMO(\mu)}^2
\lesssim \mu(\Delta) \|f\|_{\BMO(\mu)}^2,  
\end{align}
where we have used the John-Nirenberg's inequality (cf.~Remark \ref{rem:JN}). 

We next turn our attention to the estimate involving $u_2$. Let $\{I_k\}:=\{I \in \W: I \cap B \neq\emptyset\}$ and set $I^*_k=(1+\theta)I_k$ with $\theta \in (0, 1)$. Note that the family $\{I^*_k\}_k$ has bounded overlap and $\bigcup_k I^*_k \subset 2B \cap \Omega$ provided $\theta$ is sufficiently small. Also, there holds 
\[
\Delta(\widehat{x}, \delta(X)) \subset 3\Delta, \qquad 
I_k \subset B(\widehat{x}, c\ell(I_k)), \quad \forall X \in I_k, 
\]
and hence, 
\[
\rho(X)^{-1} \lesssim \frac{\ell(I_k) \mu(3 \Delta)}{m(B(\widehat{x}, \ell(I_k)))} 
\lesssim \frac{\ell(I_k) \mu(\Delta)}{m(I_k)}, \quad \forall X \in I_k. 
\]
In view of Caccioppoli inequality \eqref{list:Cacci} and Lemma \ref{lem:u-BMO}, we have 
\begin{align}\label{eq:Car-u2}
&\iint_{B \cap \Omega} |\nabla u_2(X)|^2 \delta(X) \rho(X)^{-1} \, dm(X)
\nonumber\\ 
&\le \sum_{k:\ell(I_k)<2r} \iint_{I_k} |\nabla u_2(X)|^2 \delta(X) \,  \rho(X)^{-1} \, dm(X)
\nonumber\\
&\lesssim \mu(\Delta) \sum_{k:\ell(I_k)<2r} \ell(I_k)^2 \bariint_{I_k} |\nabla u_2|^2 \, dm
\nonumber\\
&\lesssim \mu(\Delta) \sum_{k:\ell(I_k)<2r} \bariint_{I^*_k} |u_2|^2 \, dm
\nonumber\\
&\lesssim \mu(\Delta) \|f\|_{\BMO(\mu)}^2 \sum_{k:\ell(I_k)<2r} \bariint_{I^*_k} (\delta(X)/r)^{2\varrho} \, dm 
\nonumber\\
&\lesssim \mu(\Delta) \|f\|_{\BMO(\mu)}^2 \sum_{k:\ell(I_k)<2r} (\ell(I_k)/r)^{2\varrho} 
\nonumber\\
&\lesssim \mu(\Delta) \|f\|_{\BMO(\mu)}^2. 
\end{align}
Therefore, \eqref{eq:Car-u1} and \eqref{eq:Car-u2} give \eqref{eq:u12} as desired.  
\qed

%%%%%%%%%%%%%%%%%% SUBSECTION SUBSECTION SUBSECTION %%%%%%%%%%%%%%%%%%%%%
\subsection{Proof of $\eqref{list-BMO} \Longrightarrow {\rm (\ref{list-CME}')}$}
Assume that \eqref{list-BMO} holds, that is, for every $f \in \mathscr{C}(\pom)$, 
\begin{equation}\label{eq:vxt}
\sup_{\substack{x \in \pom \\ 0<r<\diam(\pom)}} \frac{1}{\mu(B(x, r) \cap \pom)} 
\iint_{B(x, r) \cap \Omega} |\nabla v|^2 \, \delta \rho^{-1}\, dm \lesssim \|f\|_{\BMO(\mu)}^2,
\end{equation}
where $v(X):=\int_{\pom} f \, d\w_L^X$ for all $X \in \Omega$. Take an arbitrary Borel set $S \subset \pom$ and let $u(X) = \w_L^X(S)$, $X \in \Omega$. Fix $X_0 \in \Omega$. By the regularity of $\w_L^{X_0}$, for every $j\ge 1$, there exist a closed set $F_j \subset S$ and an open set $U_j \supset S$ such that $\w_L^{X_0}(U_j \setminus F_j)<j^{-1}$. Using Urysohn's lemma, one can construct $f_j \in \mathscr{C}(\pom)$ such that $\mathbf{1}_{F_j} \leq f_j \leq \mathbf{1}_{U_j}$. Define 
\[
v_j(X):=\int_{\pom} f_j \, d\w_L^X, \quad X \in \Omega.
\]
Note that $\sup_j \|f_j\|_{\BMO(\mu)} \le 2$, and $|\mathbf{1}_S(x)-f_j(x)| \le \mathbf{1}_{U_j \setminus F_j}(x)$ for every $x \in \pom$. Therefore, for every compact set $K \subset \Omega$ and for every $X \in K$, 
\[
|u(X)-v_j(X)| \le \int_{\pom} |\mathbf{1}_S -f_j | \,d\w_L^X
\le \w_L^X(U_j \setminus F_j)
\lesssim _{K,X_0} \w_L^{X_0}(U_j \setminus F_j)
< j^{-1}.
\]
where Harnack's inequality was used in the third inequality. This shows that $v_j \longrightarrow u$ uniformly on compacta in $\Omega$. In view of Caccioppoli's inequality, this implies $\nabla v_j \longrightarrow \nabla u$ in $L^2(K)$ for every compact set $K \subset \Omega$.

As a consequence, for every compact set $K \subset \Omega$,  we apply \eqref{eq:vxt} to $f_j$ yields 
\begin{align*}
&\frac{1}{\mu(B(x, r) \cap \pom)} \iint_{B(x, r) \cap K} |\nabla u|^2 \, \delta \rho^{-1}\, dm
\\ 
&= \lim_{j \to \infty} \frac{1}{\mu(Q)} \iint_{B(x, r) \cap K} |\nabla v_j(X)|^2 \, \delta \rho^{-1}\, dm
\\
&\lesssim \sup_j \|f_j\|_{\BMO(\mu)} \lesssim 1 = \|u\|_{L^{\infty}(\Omega)}^2. 
\end{align*}
where the last equality follows from the maximum principle. By the arbitrariness of $K$, we have proved that for any Borel set $S\subset \pom$, the solution $u(X)=\w_L^X(S)$, $X \in \Omega$, satisfies 
\begin{equation}\label{eq:Car-S}
\sup_{\substack{x \in \pom \\ 0<r<\diam(\pom)}} \frac{1}{\mu(B(x, r) \cap \pom)} 
\iint_{B(x, r) \cap \Omega} |\nabla u|^2 \delta \rho^{-1}\, dm
\lesssim \|u\|_{L^{\infty}(\Omega)}^2. 
\end{equation}
That is, ${\rm (\ref{list-CME}')}$ holds.  
\qed

%%%%%%%%%%%%%%%%%%%%%%% SECTION SECTION SECTION %%%%%%%%%%%%%%%%%%%%%%%
%%%%%%%%%%%%%%%%%%%%%%% SECTION SECTION SECTION %%%%%%%%%%%%%%%%%%%%%%%
\section{Proof of Theorem~\ref{thm:abs}}\label{sec:abs}

Observe that $\eqref{list:abs-2} \Longrightarrow \eqref{list:abs-3} \Longrightarrow \eqref{list:abs-4}$ are trivial. Note also that by Lemma \ref{lem:two-trunc}, it is obvious that $\eqref{list:abs-4} \Longrightarrow \eqref{list:abs-3}$. Hence, it remains to show $\eqref{list:abs-1} \Longrightarrow \eqref{list:abs-2}$ and $\eqref{list:abs-3} \Longrightarrow \eqref{list:abs-1}$. Although the proof is a routine application of the method in \cite{CDMT, CMO}, the function $\rho$ defined in \eqref{RHO} will appear in the transference from the distance function to the Green function (see \eqref{eq:w-sig-com}--\eqref{eq:saw-square}). For the sake of completeness and clearness, we include a detailed proof.

%%%%%%%%%%%%%%%%%%%% SUBSECTION SUBSECTION SUBSECTION %%%%%%%%%%%%%%%%%%%
\subsection{Proof of $\eqref{list:abs-1} \Longrightarrow \eqref{list:abs-2}$}\label{sec:local}
Assume that $\mu \ll \w_L$. Fix an arbitrary $Q_0 \in \D_{k_0}$, where $k_0 \in \Z$. Let $X_0=X_{M_0 \Delta_{Q_0}}$ be a corkscrew point relative to $M_0 \Delta_{Q_0}$, where $M_0$ is large enough so that $X_0 \not\in 4B^*_{Q_0}$. By Lemma~\ref{lem:wL-G} part \eqref{list-1} and Harnack's inequality, there exists $C_0>1$ such that
\begin{equation}\label{eq:lowerbdd}
\w_L^{X_0}(Q_0) \geq C_0^{-1}.
\end{equation}
Write 
\begin{align}\label{eq:normalize} 
\w:= C_0 \mu(Q_0) \w_L^{X_0} \quad\text{and}\quad \G:= C_0 \mu(Q_0) G_L(X_0, \cdot).
\end{align}
By assumption and \eqref{eq:lowerbdd},  we have $\mu \ll \w$ and
\begin{equation}\label{eq:QC}
1 \leq \frac{\w(Q_0)}{\mu(Q_0)} = C_0 \w_{L}^{X_0}(Q_0) \leq C_0.
\end{equation}
For $N > C_0$, we let $\F_N^+ :=\{Q_j\} \subset \D_{Q_0} \backslash \{Q_0\}$, respectively, $\F_N^- :=\{Q_j\} \subset \D_{Q_0} \backslash \{Q_0\}$, be the collection of descendants of $Q_0$ which are maximal (and therefore pairwise disjoint) with respect to the property that
\begin{align}\label{eq:stopping}
\frac{\w(Q_j)}{\mu(Q_j)} < \frac{1}{N}, \qquad\text{ respectively}\quad \frac{\w(Q_j)}{\mu(Q_j)} >N.
\end{align}
Write $\F_N:=\F_N^+\cup\F_N^-$ and note that $\F_N^+\cap\F_N^-=\varnothing$. By maximality, there holds
\begin{align}\label{eq:NN}
\frac{1}{N}\leq \frac{\w(Q)}{\mu(Q)} \leq N, \qquad \forall\,Q \in \D_{\F_N, Q_0}.
\end{align}
Denote, for every $N>C_0$,
\begin{align}\label{eq:E0N-EN}
E_N^\pm := \bigcup_{Q \in \F_N^\pm} Q,
\qquad E_N^0:=E_N^+\cup E_N^-, \qquad E_N := Q_0\setminus E_N^0,
\end{align}
and
\begin{align}\label{eq:Q-decom}
Q_0 = \bigg(\bigcap_{N>C_0} E_N^0\bigg)\cup \bigg(\bigcup_{N>C_0} E_N \bigg)
=: E_0\cup \bigg(\bigcup_{N>C_0} E_N \bigg).
\end{align}

By Lemma~3.61 and Proposition 6.1 in \cite{HM}, $\Omega_{\F_N, Q_0}$ is a bounded $1$-sided NTA domain and 
\begin{align*}
E_N \subset F_N:=\pom \cap \partial \Omega_{\F_N, Q_0} 
\subset \overline{Q}_0 \setminus \bigcup_{Q_j \in \F_N} \inter(Q_j), 
\end{align*}
which leads 
\begin{align*}
F_N \backslash E_N
\subset \bigg(\overline{Q}_0 \backslash \bigcup_{Q_j \in \F_N} \inter(Q_j) \bigg)
\bigg\backslash \bigg(Q_0 \backslash \bigcup_{Q_j \in \F_N} Q_j \bigg)
\subset \partial Q_0 \cup \bigg(\bigcup_{Q_j \in \F_N} \partial Q_j\bigg).
\end{align*}
Since $\mu$ is doubling, this implies 
\begin{align}\label{eq:ENFN}
\mu(F_N \setminus E_N) =0.
\end{align}

Next, we are going to show
\begin{equation}\label{eq:WE}
\mu(E_0)=0.
\end{equation}
Let $x \in E_{N+1}^{ \pm}$. Then there exists $Q_x \in \F_{N+1}^\pm$ such that $x \in Q_x$. By \eqref{eq:stopping}, we have
\begin{align*}
\frac{\w(Q_x)}{\mu(Q_x)} < \frac{1}{N+1} <\frac{1}{N} \quad\text{if $Q_x \in \F_{N+1}^+$} \quad \text{or}\quad
\frac{\w(Q_x)}{\mu(Q_x)} >N+1>N \quad\text{if $Q_x \in \F_{N+1}^-$}.
\end{align*}
By the maximality of the cubes in $\F_N^\pm$, one has $Q_x \subset Q'_x$ for some $Q'_x \in \F_N^\pm$ with $x \in Q'_x \subset E_N^\pm$. Thus, $\{E_N^+\}_N$, $\{E_N^-\}_N$ and $\{E_N^0\}_N$ are decreasing sequences of sets. This, together with $\mu \in L^1_{\loc}(\pom)$ and the fact that $\w(E_N^\pm)\le \w(Q_0)\le C_0 \mu(Q_0)<\infty$ and $\mu(E_N^\pm)\le \mu(Q_0)<\infty$, gives that
\begin{equation}\label{wrqfawfvrw}
\w\bigg(\bigcap_{N>C_0} E_N^\pm\bigg)=\lim_{N\to\infty} \w(E_N^\pm)
\qquad\text{and}\qquad
\mu\bigg(\bigcap_{N>C_0} E_N^\pm\bigg)=\lim_{N\to\infty} \mu(E_N^\pm).
\end{equation}
By \eqref{eq:stopping} and \eqref{eq:E0N-EN},
\[
\w(E_N^+) = \sum_{Q \in \F_N^+} \w(Q) 
<\frac1N \sum_{Q \in \F_N^+} \mu(Q)
=\frac1N \mu(E_N^+) \le \frac1N \mu(Q_0),
\]
which along with \eqref{wrqfawfvrw} yields
\[
\w\bigg(\bigcap_{N>C_0} E_N^+\bigg)=\lim_{N\to\infty} \w(E_N^+)=0.
\]
This and $\mu \ll \w$ in turn lead 
\begin{equation}\label{eq:EN+}
0=\mu\bigg(\bigcap_{N>C_0} E_N^+\bigg)=\lim_{N \to \infty} \mu(E_N^+).
\end{equation}
On the other hand, \eqref{eq:stopping} yields
\[
\mu(E_N^-) = \sum_{Q\in \F_N^-} \mu(Q)
<\frac1N \sum_{Q \in \F_N^-} \w(Q)
=\frac1N \w(E_N^-) \le \frac{C_0}N \mu(Q_0),
\]
and hence,
\begin{equation}\label{eq:EN-}
\mu \bigg(\bigcap_{N>C_0} E_N^-\bigg)=\lim_{N\to\infty} \mu(E_N^-)=0.
\end{equation}
Since $\{E_N^0\}_N$ is a decreasing sequence of sets with $\mu(E_N^0) \le \mu(Q_0) <\infty$, \eqref{eq:EN+} and \eqref{eq:EN-} readily imply 
\begin{equation*}
\mu(E_0) = \lim_{N\to\infty} \mu(E_N^0)
\le \lim_{N\to\infty} \mu(E_N^+) + \lim_{N \to \infty} \mu(E_N^-) = 0.
\end{equation*}
This shows \eqref{eq:WE}.

Now we turn our attention to the square function estimates in $L^2(F_N, \mu)$. Let $u$ be a bounded weak solution of $Lu=0$ in $\Omega$. We may assume that $\|u\|_{L^{\infty}(\Omega)}=1$. Recalling \eqref{RHO} and \eqref{eq:normalize}, we use Lemma \ref{lem:wL-G} part \eqref{list-2}, \eqref{eq:NN} and Harnack's inequality  to conclude that 
\begin{align}\label{eq:w-sig-com}
\frac{\G(X)}{\delta(X)} \rho(X) \simeq \frac{\w(Q)}{\mu(Q)} \simeq_N 1, \quad\forall X \in \Omega_{\F_N, Q_0}. 
\end{align}
By the definition of $\Omega_{\F_N, Q_0}$, \eqref{eq:w-sig-com}, and Lemma \ref{lem:uGNQ},  we deduce 
\begin{align}\label{eq:saw-square}
\iint_{\Omega_{\F_N, Q_0}} |\nabla u|^2\, \delta \rho^{-1} dm
\lesssim_N \iint_{\Omega_{\F_N, Q_0}} |\nabla u|^2 \, \G \, dm
\lesssim \|u\|_{L^{\infty}(\Omega)}^2 \w_L^{X_{Q_0}}(Q_0) \le 1. 
\end{align}

To continue, we note that if $Q \in \D_{Q_0}$ is so that $Q \cap E_N \neq \varnothing$, then necessarily $Q \in \D_{\F_N, Q_0}$, otherwise, $Q \subset Q' \in \F_N$, hence $Q \subset Q_0 \backslash E_N$. Given $X \in U_Q$, pick $\widehat{x} \in \pom$ so that $\delta(X) = |X-\widehat{x}|$. Since 
\begin{align*}
Q \subset C_1 \Delta(\widehat{x}, \delta(X)) \quad\text{ and }\quad 
B(\widehat{x}, \delta(x)) \subset C_2 B(x, \delta(X)), \quad\forall x \in Q, \, X \in U_Q, 
\end{align*}
the doubling property of $\mu$ and $m$, and \eqref{eq:saw-square} imply 
\begin{align}\label{eq:SQk-ENk}
\int_{E_N} (\S_{Q_0} u)^2 d\mu
&=\int_{E_N} \iint_{\bigcup\limits_{x \in Q \in \D_{Q_0}} U_Q} 
|\nabla u(X)|^2 \delta(X)^2 \frac{dm(X)}{m(B(x, \delta(X)))} d\mu(x) 
\nonumber \\
&\lesssim \sum_{Q \in \D_{Q_0}} \iint_{U_Q} |\nabla u(X)|^2 \delta(X)^2 
\bigg(\int_{Q \cap E_N} \frac{d\mu(x)}{m(B(x, \delta(X)))}\bigg) \, dm(X)
\nonumber \\ 
&\lesssim \sum_{Q \in \D_{\F_N, Q_0}} \iint_{U_Q} |\nabla u(X)|^2 \delta(X) \rho(X)^{-1} \, dm(X)
\nonumber \\
&\lesssim \iint_{\Omega_{\F_N, Q_0}} |\nabla u(X)|^2 \delta(X) \rho(X)^{-1}\, dm(X)
\leq C_{N}, 
\end{align}
where we have used that the family $\{U_Q\}_{Q\in\D}$ has bounded overlap. This along with \eqref{eq:ENFN} yields 
\begin{align}\label{eq:square-FNj}
\S_{Q_0} u \in L^2(F_N, \mu).
\end{align} 

We next claim that fixed $\alpha>0$, for any $r_0 \ll 2^{-k_0}=\ell(Q_0)$,  
\begin{align}\label{eq:Sa-SQ}
\S^{\alpha}_{r_0}u(x) \leq \S_{Q_0} u(x), \quad x \in Q_0. 
\end{align}
It suffices to show $\Gamma_{\alpha}^{r_0}(x) \subset \Gamma_{Q_0}(x)$ for any $x \in Q_0$. Indeed, let $Y \in \Gamma_{\alpha}^{r_0}(x)$. Pick $I \in \W$ so that $Y \in I$, and hence, $\ell(I) \simeq \delta(Y) \leq |Y-x|<r_0 \ll 2^{-k_0} = \ell(Q_0)$. Pick $Q_I \in \D_{Q_0}$ such that $x \in Q_I$ and $\ell(Q_I)=\ell(I) \ll \ell(Q_0)$. Thus, one has 
\begin{align*}
\dist(I, Q_I) \leq |Y-x| < (1+\alpha) \delta(Y) \leq C(1+\alpha) \ell(I) = C(1+\alpha) \ell(Q_I).  
\end{align*}
Recall the definition of $\vartheta$ in \eqref{eq:WQ} and choose $\vartheta_0>0$ sufficiently large so that $2^{\vartheta_0}>C(1+\alpha)$. Then for any $\vartheta \ge \vartheta_0$, $Y \in U_{Q_I} \subset \Gamma_{Q_0}(x)$ and consequently \eqref{eq:Sa-SQ} holds.  

To complete the proof we note that, it follows from \eqref{eq:square-FNj} and \eqref{eq:Sa-SQ} that $\S^{\alpha}_{r_0} u \in L^2(F_N, \mu)$. This together with Lemma \ref{lem:two-trunc} easily yields 
\begin{equation}\label{eq:Sr-u-L2-FN}
\S_r^{\alpha} u \in L^2(F_N, \mu),\quad \text{for any } r>0.  
\end{equation}

We note that the previous argument has been carried out for an arbitrary $Q_0\in \D_{k_0}$. Hence, using  \eqref{eq:E0N-EN}, \eqref{eq:Q-decom}, and \eqref{eq:ENFN} with $Q_k\in \D_{k_0}$, we conclude, with the induced notation, that 
\begin{align}\label{eq:EE-FF}
\partial \Omega 
&= \bigcup_{Q_k \in \D_{k_0}} Q_k 
=\bigg(\bigcup_{Q_k \in \D_{k_0}} E^k_0\bigg) \bigcup 
\bigg(\bigcup_{Q_k \in \D_{k_0}} \bigcup_{N>C_0} E^k_N \bigg) 
\nonumber \\
&=\bigg(\bigcup_{Q_k \in \D_{k_0}} E^k_0\bigg) \bigcup 
\bigg(\bigcup_{Q_k \in \D_{k_0}} \bigcup_{N>C_0} F^k_N \bigg) 
=: F_0 \cup \bigg(\bigcup_{k, N} F^k_N \bigg), 
\end{align}
where $\mu(F_0)=0$ and $F^k_N=\pom \cap \partial \Omega_{\F^k_N, Q_k}$ where each $\Omega_{\F^k_N, Q_k} \subset \Omega$ is a bounded 1-sided NTA domain. Combining \eqref{eq:EE-FF} and \eqref{eq:Sr-u-L2-FN} with $F_N^k$ in place of $F_N$, the proof of $\eqref{list:abs-1} \Longrightarrow \eqref{list:abs-2}$ is complete.  \qed

%%%%%%%%%%%%%%%%%%%% SUBSECTION SUBSECTION SUBSECTION %%%%%%%%%%%%%%%%%%%
\subsection{Proof of $\eqref{list:abs-3} \Longrightarrow \eqref{list:abs-1}$} 
Given  $Q_0 \in \D$ and $\eta \in(0,1)$, we define the modified non-tangential cone
\begin{equation}\label{def-whitney-eta}
\Gamma_{Q_0}^{\eta}(x):=\bigcup_{x \in Q\in\D_{Q_0}} U_{Q,\eta^3},\qquad 
U_{Q, \eta^3} = \bigcup_{\substack{Q' \in \D_Q \\ \ell(Q')>\eta^3 \ell(Q)}} U_{Q'}. 
\end{equation} 
With these in hand, we define the modified dyadic square function as 
\begin{equation*}
\S_{Q_0}^{\eta} u(x):=\bigg(\iint_{\Gamma_{Q_0}^{\eta}(x)}  
|\nabla u(X)|^2 \delta(X)^2 \frac{dm(X)}{m(B(x, \delta(X)))} \bigg)^{\frac12}. 
\end{equation*}

The following result was obtained in \cite[Lemma 3.10]{CHMT} for $\beta>0$ and in \cite[Lemma 3.40]{CMO} for $\beta=0$, both in the context of 1-sided CAD, extending  \cite[Lemma~2.6]{KKoPT} and \cite[Lemma~2.3]{KKiPT}. 

%%%%%%%%%%%%%%%%%%%%%%%%% LEMMA LEMMA LEMMA %%%%%%%%%%%%%%%%%%%%%%%
\begin{lemma}\label{lem:sq-lower} 
Let $Lu=-\div(A\nabla u)$ be a real elliptic operator with $A$ satifying \eqref{eq:elli}. There exist $0<\eta\ll 1$  (depending only on $n$ and $\Lambda$), and $\beta_0\in (0,1)$, $C_\eta\ge 1$ both depending only on $n$, $\Lambda$, and $\eta$, such that for every $Q_0 \in \D$, for every $\beta \in (0, \beta_0)$, and for every Borel set $F \subset Q_0$ satisfying $\w_L^{X_{Q_0}}(F) \le\beta \w_L^{X_{Q_0}}(Q_0)$, there exists a Borel set $S\subset Q_0$ such that the bounded weak solution $u(X)=\w_L^X(S)$, $X \in \Omega$, satisfies
\begin{equation}\label{eq:b>0}
\S_{Q_0}^{\eta} u(x) \ge C_\eta^{-1} \big(\log(\beta^{-1})\big)^{\frac12}, \quad \forall\, x \in F.
\end{equation}
Furthermore, in the case $\beta=0$, that is, when $\w_L^{X_{Q_0}}(F)=0$, there exists a Borel set $S\subset Q_0$ such that the bounded weak solution $u(X)=\w_L^X(S)$, $X\in\Omega$, satisfies
\begin{equation}\label{eq:b=0}
\S_{Q_0}^{\eta} u(x) = \infty, \qquad \forall\, x \in F.
\end{equation}
\end{lemma}
%%%%%%%%%%%%%%%%%%%%%%%%% LEMMA LEMMA LEMMA %%%%%%%%%%%%%%%%%%%%%%%

Assuming Lemma \ref{lem:sq-lower} holds momentarily, we now demonstrate $\eqref{list:abs-3} \Longrightarrow \eqref{list:abs-1}$. We first observe that for any $Q_0 \in \D$ and $\eta>0$, there exists $\alpha_0>0$ and $r>0$ such that 
\begin{equation}\label{eq:TQTrQ}
\Gamma_{Q_0}^{\eta}(x) \subset \Gamma^{\alpha}_r(x), \quad\forall \alpha \ge \alpha_0, 
\end{equation}
Indeed, let $Y \in \Gamma_{Q_0}^{\eta}(x)$. By definition, there exist $Q \in \D_{Q_0}$ and $Q' \in \D_Q$ with $\ell(Q')>\eta^3 \ell(Q)$ such that $Y \in U_{Q'}$ and $x \in Q$. Then $Y \in I^*$ for some $I \in \W_{Q'}^*$, and hence, 
\begin{align}\label{eq:sidelength}
\delta(Y) \simeq \ell(I) \simeq \ell(Q')\le \ell(Q)<\eta^{-3}\ell(Q').
\end{align}
This in turn gives that 
\begin{align}\label{eq:Yx}
|Y-x| \leq \diam(I^*) + \dist(I, x) &\leq \diam(I^*) + \dist(I, Q') + \diam(Q) 
\nonumber\\
&\lesssim \ell(Q')+\ell(Q) \lesssim \ell(Q), 
\end{align}
Combining \eqref{eq:sidelength} with \eqref{eq:Yx}, we get 
\begin{align*}
|Y-x| \leq C_1 \ell(Q_0) =: r/2 \quad \text{ and }\quad 
|Y-x| \leq (1+C_2) \delta(Y) =:(1+\alpha_0) \delta(Y), 
\end{align*}
where $C_1$ depends only on the allowable parameters, and $C_2$ depends only on the allowable parameters and also on $\eta$. 

Let $\alpha_0$ be so that \eqref{eq:TQTrQ} holds. Suppose that \eqref{list:abs-3} holds where throughout it is assumed that  $\alpha \ge \alpha_0$. Note that by Harnack's inequality, $\w_L^X\ll \w_L^Y$ for all $X, Y \in \Omega$. In order to prove that $\mu \ll \w_L$ on $\pom$, it suffices to show that for any given $Q_0 \in \D$,
\begin{align}\label{eq:absQ}
F \subset Q_0,\quad \w_L^{X_{Q_0}}(F)=0 \quad \Longrightarrow \quad \mu(F)=0.
\end{align}
Now fix $F \subset Q_0$ with $\w_L^{X_{Q_0}}(F)=0$. By \eqref{eq:TQTrQ} and Lemma~\ref{lem:sq-lower} applied to $F$, there exists a Borel set $S\subset Q_0$ such that $u(X):=\w_L^{X}(S)$, $X\in\Omega$, satisfies
\begin{align}\label{eq:S-lower}
\S^{\alpha}_r u(x) \ge \S_{Q_0}^{\eta} u(x) =\infty, \qquad \forall\,x\in F. 
\end{align}
By assumption and Lemma \ref{lem:two-trunc}, we have that $\S^{\alpha}_r u(x)<\infty$ for $\mu$-a.e.~$x \in \pom$. Hence, $\mu(F)=0$ as desired. \qed

It remains to show Lemma \ref{lem:sq-lower}. To this end, we present some definition and auxiliary results.  

\begin{definition}\label{def:good-cover}
Fix $Q_0 \in \D$ and let $\nu$ be a regular Borel measure on $Q_0$.  Given $\varepsilon_0 \in (0,1)$ and a Borel set $\varnothing \neq F \subset Q_0$, a good $\varepsilon_0$-cover of $F$ with respect to $\nu$, of length $k \in \mathbb{N}$, is a collection $\{\mathcal{O}_{\ell}\}_{\ell=1}^k$ of Borel subsets of $Q_0$, together with pairwise disjoint families $\mathcal{F}_{\ell} =\{Q^{\ell}\} \subset \D_{Q_0}$, $1\le\ell \le  k$, such that the following hold:
\begin{list}{\textup{(\theenumi)}}{\usecounter{enumi}\leftmargin=1cm \labelwidth=1cm \itemsep=0.2cm 
			\topsep=.2cm \renewcommand{\theenumi}{\alph{enumi}}}
		
\item $F \subset \mathcal{O}_k \subset \mathcal{O}_{k-1} \subset \dots \subset \mathcal{O}_2 \subset \mathcal{O}_1 \subset Q_0$. 
		
\item $\mathcal{O}_{\ell}=\bigcup_{Q^{\ell} \in \F_{\ell}} Q^{\ell}$, for every $1 \le \ell \le k$.
		
\item $\nu(\mathcal{O}_{\ell} \cap Q^{\ell-1}) \le \varepsilon_0 \nu(Q^{\ell-1})$,  for each $Q^{\ell-1} \in \F_{\ell-1}$ and $2\le \ell \le k$. 
		
\end{list}

Analogously,  a good $\varepsilon_0$-cover of $F$ with respect to $\nu$, of length $\infty$, is a collection $\{\mathcal{O}_{\ell}\}_{\ell=1}^\infty$ of Borel subsets of $Q_0$, together with pairwise disjoint families $\mathcal{F}_{\ell} =\{Q^{\ell}\} \subset \D_{Q_0}$, $\ell\ge 1$, such that the following hold:
\begin{list}{\textup{(\theenumi)}}{\usecounter{enumi}\leftmargin=1cm \labelwidth=1cm \itemsep=0.2cm 
		\topsep=.2cm \renewcommand{\theenumi}{\alph{enumi}}}
	
\item $F \subset\dots\subset \mathcal{O}_k \subset \mathcal{O}_{k-1} \subset \dots \subset \mathcal{O}_2 \subset \mathcal{O}_1 \subset Q_0$. 
	
\item $\mathcal{O}_{\ell}=\bigcup_{Q^{\ell} \in \F_{\ell}} Q^{\ell}$, for every $\ell\ge 1$,
	
\item $\nu(\mathcal{O}_{\ell} \cap Q^{\ell-1}) \le \varepsilon_0 \nu(Q^{\ell-1})$,  for each $Q^{\ell-1} \in \F_{\ell-1}$ and $\ell \ge 2$. 
	
\end{list}
\end{definition}

\begin{lemma}[{\cite[Lemma~3.5]{CHMT}}]\label{lem:cover}
Let $\nu$ be a regular Borel measure on $Q_0$ and assume that it is dyadically doubling on $Q_0$ with a constant $C_{\nu}$. For every $0<\varepsilon_0 \le e^{-1}$, if $\varnothing\neq F \subset Q_0$ with $\nu(F) \le \alpha \nu(Q_0)$ and $0<\alpha \le \varepsilon_0^2/(2C_{\nu}^2)$ then $F$ has a good $\varepsilon_0$-cover with respect to $\nu$ of length $k=k(\alpha, \varepsilon_0, C_\nu) \in \mathbb{N}$, $k \ge 2$, which satisfies $k \simeq \frac{\log \alpha^{-1}}{\log \varepsilon_0^{-1}}$. In particular, if $\nu(F)=0$, then $F$ has a good $\varepsilon_0$-cover of arbitrary length. 
\end{lemma}

Although in the case $\nu(F)=0$ the result above gives an $\varepsilon_0$-cover of arbitrary length, it requires an $\varepsilon_0$-cover of infinite length to obtain \eqref{eq:b=0}. 

\begin{lemma}[{\cite[Lemma~A.5]{CMO}}]\label{lem:infinite}
Fix $Q_0 \in \D$. Let $\nu$ be a regular Borel measure on $Q_0$ and assume that it is dyadically doubling on $Q_0$. For every $0<\varepsilon_0 \le e^{-1}$, if $\varnothing \neq F \subset Q_0$ with $\nu(F)=0$, then $F$ has a good $\varepsilon_0$-cover of length $\infty$. 
\end{lemma}

To continue we need to introduce some notation. Given $\eta=2^{-k_*}<1$ small enough to be chosen momentarily and given $Q \in \D$, we define $Q^{(\eta)} \in \D_Q$ to be the unique dyadic cube such that $x_Q \in Q^{(\eta)}$ and $\ell(Q^{(\eta)}) = \eta \ell(Q)$. 

\medskip

We are now ready to prove  Lemma~\ref{lem:sq-lower}: 

\begin{proof}[Proof of Lemma~$\ref{lem:sq-lower}$]
Let $0<\varepsilon_0 \le e^{-1}$ and $0<\beta \le \varepsilon_0^2/(2C_{\nu}^2)$. Fix $Q_0 \in \D$, and a Borel set $F \subset Q_0$ such that $\w_L^{X_{Q_0}}(F) \le\beta \w_L^{X_{Q_0}}(Q_0)$. From Lemma~\ref{lem:PDE} and Harnack's inequality, we see that $\w_L^{X_{Q_0}}$ is a regular Borel dyadically doubling measure on $Q_0$. Then,  Lemma~\ref{lem:cover} applied to $\nu=\w_L^{X_{Q_0}}$ yields $\{\mathcal{O}_{\ell}\}_{\ell=1}^k$, a good $\varepsilon_0$-cover of $F$ of length $k \simeq \frac{\log \beta^{-1}}{\log \varepsilon_0^{-1}}$, with $k \ge 2$. Define $\mathcal{O}_{\ell}^{(\eta)}:=\bigcup_{Q\in\F_\ell} Q^{(\eta)}$ for each $1\le \ell \le k$, and consider the Borel set $S_k:=\bigcup_{\ell=2}^k (\mathcal{O}_{\ell-1}^{(\eta)}\setminus \mathcal{O}_{\ell})$. For each $x \in F$ and $1 \le \ell \le k$, let $Q^{\ell}_x \in \F_{\ell}$ be the unique dyadic cube containing $x$, and let $P^{\ell}_x \in \D_{Q^{\ell}_x}$ be the unique dyadic cube containing $x$ with $\ell(P^{\ell}_x) = \eta \ell(Q^{\ell}_x)$. If we write $u_k(X):=\w_L^X(S_k)$ for all $X \in \R^{n+1}_+$, then \cite[eq. (5.26)]{HLM} gives 
\begin{align}\label{eq:uk-osc}
|u_k(X_{(Q^{\ell}_x)^{(\eta)}}) - u_k(X_{(P^{\ell}_x)^{(\eta)}})| \gtrsim 1, \qquad\forall x \in F, \, \, 1\le \ell\le k, 
\end{align}
where the implicit constant depends only on the allowable parameters. This and the argument in \cite[p.~2139]{HLM} imply
\begin{align}\label{eq:Qil}
\iint_{U_{Q^{\ell}_x, \eta^3}} |\nabla u_k(X)|^2 \delta(X)^2 \frac{dm(X)}{m(B(x, \delta(X)))} 
\gtrsim_\eta 1, \qquad	\forall\, 1\le \ell \le k. 
\end{align}	
Thus, summing \eqref{eq:Qil} over $1 \le \ell \le k$, we have for every $x \in F$,  
\begin{align}\label{eq:k-Su}
\log \beta^{-1} \simeq k & \lesssim_\eta \sum_{\ell=1}^{k} \iint_{U_{Q^{\ell}_x, \eta^3}} 
 |\nabla u_k(X)|^2 \delta(X)^2 \frac{dm(X)}{m(B(x, \delta(X)))} 
\nonumber\\ 
& \lesssim_\eta \iint_{\bigcup_{\ell=1}^k U_{Q^{\ell}_x, \eta^3}} 
 |\nabla u_k(X)|^2 \delta(X)^2 \frac{dm(X)}{m(B(x, \delta(X)))} 
\nonumber\\ 
&\le \iint_{\bigcup \limits_{x \in Q \in \D_{Q_0}} U_{Q, \eta^3}} 
 |\nabla u_k(X)|^2 \delta(X)^2 \frac{dm(X)}{m(B(x, \delta(X)))} 
= S_{Q_0}^{\eta} u_k(x)^2,
\end{align}
where we have used that the family $\{U_{Q, \eta^3}\}_{Q \in \D}$ has bounded overlap albeit with a constant depending on $\eta$. This is the case $\beta>0$.

Next, we deal with the case $\beta=0$. Fix $Q_0 \in \D$ and a Borel set $\varnothing\neq F \subset Q_0$ with $\w_L^{X_{Q_0}}(F )=0$. Applying Lemma~\ref{lem:infinite} with $\nu=\w_L^{X_{Q_0}}$, one can find a good $\varepsilon_0$-cover of $F$ of length $\infty$, denoted by $\{\mathcal{O}_{\ell}\}_{\ell=1}^{\infty}$. In particular, for every $N\in \mathbb{N}$,  $\{\mathcal{O}_{\ell}\}_{\ell=1}^{N}$ is a good $\varepsilon_0$-cover of $F$ of length $N$. Invoking \eqref{eq:uk-osc}, we obtain 
\begin{align}\label{eq:uN-osc}
|u_N(X_{(Q^{\ell}_x)^{(\eta)}}) - u_N(X_{(P^{\ell}_x)^{(\eta)}})| \gtrsim 1, \qquad \forall x\in F, \, 1\le \ell\le N.
\end{align}
where $u_N(X):=\w_L^X(S_N)$ and $S_N:=\bigcup_{\ell=2}^N (\mathcal{O}_{\ell-1}^{(\eta)}\setminus \mathcal{O}_{\ell})$. Define $S:=\bigcup_{\ell=2}^{\infty} (\mathcal{O}_{\ell-1}^{(\eta)}\setminus \mathcal{O}_{\ell})$ and $u(X):=\w_L^X(S)$, $X \in \Omega$. The monotone convergence theorem gives $u_N(X)\longrightarrow u(X)$ as $N \to \infty$ for every $X \in \Omega$. Thus, this and \eqref{eq:uN-osc} readily imply 
\begin{align}\label{eq:u-osc}
|u(X_{(Q^{\ell}_x)^{(\eta)}}) - u(X_{(P^{\ell}_x)^{(\eta)}})| \gtrsim 1, \qquad \forall x\in F, \ \ell\ge 1.
\end{align}
Keeping \eqref{eq:u-osc} in mind and using a similar argument as in \eqref{eq:Qil} and \eqref{eq:k-Su}, we conclude that 
\begin{align*}
\S_{Q_0}^{\eta} u(x)^2 \gtrsim N, \qquad \forall x\in F, \, N \ge 1, 
\end{align*}
where the implicit constant is independent of $N$. Letting $N \to \infty$, we get $S_{Q_0}^{\eta} u(x)=\infty$ for all  $x \in F$. This completes the proof.  
\end{proof}
%%%%%%%%%%%%%%%%%%%%%%%%%%% PROOF PROOF PROOF %%%%%%%%%%%%%%%%%%%%%%

\medskip
\noindent\textbf{Acknowledgements }  
The authors wish to thank \'{O}. Dom\'{i}nguez, J.M. Martell, and P. Tradacete, for the ideas of \cite{CDMT}, and thank J.M. Martell for helpful comments concerning Lemma \ref{lem:Psi}. The authors also thank the referee for careful reading and for valuable comments, which lead to the improvement of this  paper.

%%%%%%%%%%%%%%%%%%% BIBLIOGRAPHY BIBLIOGRAPHY BIBLIOGRAPHY %%%%%%%%%%%%%%%
%%%%%%%%%%%%%%%%%%% BIBLIOGRAPHY BIBLIOGRAPHY BIBLIOGRAPHY %%%%%%%%%%%%%%%

\end{document}